\documentclass[11pt]{amsart}
\usepackage[left=10mm,right=10mm]{geometry}
\usepackage[english]{babel}
\usepackage[T1]{fontenc}
\usepackage[utf8]{inputenc}
\usepackage{amsthm}
\usepackage{amsmath}
\title{True self-repelling motion above a general barrier}
\author{Laure Mar{\^e}ch\'e}
\email{laure.mareche@math.unistra.fr}
\address{Institut de Recherche Mathématique Avancée, 
UMR 7501 Université de Strasbourg et CNRS, 
7 rue René Descartes, 67000 Strasbourg, France}
\theoremstyle{plain}
\newtheorem{theorem}{Theorem}

\newtheorem{lemma}[theorem]{Lemma}
\newtheorem{proposition}[theorem]{Proposition}

\theoremstyle{definition}
\newtheorem{definition}[theorem]{Definition}
\theoremstyle{remark}
\newtheorem{remark}[theorem]{Remark}
\usepackage{dsfont}
\usepackage{hyperref}

\begin{document}

\maketitle

\begin{abstract}
 The true self-repelling motion is a continuous-time random process which was introduced by Tóth and Werner in 1998 \cite{Toth_et_al1998} to be a limit for the ``true'' self-avoiding random walk defined by Tóth in 1995 \cite{Toth1995}. The construction of the true self-repelling motion involves an uncountable system of coalescing Brownian motions starting from all points of the upper half-plane, related to the Brownian web, but reflected and absorbed on a ``barrier'' which is the abscissa axis. In this work, we consider much more general barriers, construct an uncountable system of coalescing Brownian motions reflected and absorbed on these barriers, and the true self-repelling motion associated to it. The extension of the proofs of Tóth and Werner to this more general case is surprisingly difficult, especially when the barrier is irregular.   
\end{abstract}


\section{Introduction}

The \emph{true self-repelling motion} is a continuous-time random process which was introduced by Tóth and Werner \cite{Toth_et_al1998} as a candidate scaling limit for the discrete process known as \emph{``true'' self-avoiding walk with bond repulsion on $\mathds{Z}$}. This ``true'' self-avoiding walk, which was defined by Tóth in \cite{Toth1995}, describes the position of a walker moving on $\mathds{Z}$ in discrete time as follows. For each edge of $\mathds{Z}$, we call \emph{local time} on this edge the number of previous jumps of the walk along the edge, and at each time, the walker can move along an edge of $\mathds{Z}$ with a probability proportional to $\exp(-\beta \times($local time on the edge$))$, $\beta>0$. In \cite{Toth1995}, Tóth asked the question of the scaling limit of the ``true'' self-avoiding walk, but could not give a complete answer. Since then, this walk and its generalizations have received a lot of attention \cite{Toth1994,Toth1996,Toth1997,Toth_et_al2008,Erschler2011,Erschler2011stuck,Mountford_et_al2014,Mareche2022,Kosygina_et_al2023,Mareche_et_al2023,Liu_et_al_2024,Bremont_et_al2024,Bremont_et_al2025}, and the convergence of the walk to the true self-repelling motion was finally proven by Kosygina and Peterson in \cite{Kosygina_et_al2025}.

The construction of the true self-repelling motion is very complex and takes a large part of \cite{Toth_et_al1998}. A sketch is as follows. If $(x,h)\in\mathds{R} \times (0,+\infty)$, a \emph{reflected/absorbed Brownian motion (RAB) starting from $(x,h)$} is a random function $[x,+\infty) \to [0,+\infty)$ which ``at the left of 0'' (on the interval $[x,0]$ if $x<0$) is a Brownian motion reflected on the abscissa axis, ``at the right of 0'' (on $[0,+\infty)$ if $x \leq 0$, on $[x,+\infty)$ if $x>0$) is a Brownian motion absorbed by the abscissa axis, and whose value at $x$ is $h$. Inspired by the infinite systems of coalescing Brownian motions introduced by Arratia in \cite{Arratia1979,Arratia_thesis}, Tóth and Werner \cite{Toth_et_al1998} constructed a system of RABs starting from \emph{every} point of $\mathds{R} \times (0,+\infty)$ so that any two of them coalesce when they meet, called system of \emph{forward lines}. Since there are an uncountable number of RABs, their construction is highly nontrivial. Afterwards, they built a system of \emph{backward lines}, which are ``forward lines going backwards'' (denoting $\Lambda_{(x,h)}^*(.)$ the backward lines, the $\Lambda_{(-x,h)}^*(-.)$ have the distribution of forward lines), but entwined with the forward lines so a backward line can never cross a forward line. Again, the arguments are intricate. They finally constructed the true self-repelling motion $(X_t)_{t \geq 0}$ by proving that roughly, for any $t>0$, there is a single $(x,h)\in\mathds{R} \times (0,+\infty)$ so the integral of the function formed by the forward and backward lines starting from $(x,h)$ has value $t$, and they set $X_t=x$. A number of properties of the true self-repelling motion were studied by Tóth and Werner in \cite{Toth_et_al1998}, and more were proven by Dumaz and Tóth in \cite{Dumaz2012,Dumaz_et_al2013,Dumaz2018}.

The system of forward lines built in \cite{Toth_et_al1998} inspired the construction of the celebrated Brownian web by Fontes, Isopi, Newman and Ravishankar \cite{Fontes_et_al2002,Fontes_et_al2003}. The Brownian web is a system of Brownian motions starting from all points of $\mathds{R}^2$ which coalesce where they meet, seen in an appropriate topology (see \cite{Schertzer_et_al2016} for a review, see also the later works \cite{Fontes2018,Cannizzaro_et_al2023}). It turned out to be a limiting object for a lot of systems of coalescing random walks (many examples can be found in Section 7.2 of  \cite{Schertzer_et_al2016}). However, the construction of the true self-repelling motion does not require Brownian web tools, hence we will not use them here. 

A first possible generalization of the true self-repelling motion was already mentioned in the original paper of Tóth and Werner \cite{Toth_et_al1998}: instead of defining the forward and backward lines as Brownian motions reflected/absorbed on the abscissa axis, they can be reflected and absorbed on another kind of ``barrier'', a Brownian motion. A precise definition was given in Section 10 of \cite{Toth_et_al1998}, but the changes needed in the proofs to accomodate for this different barrier were not written down. This particular generalization was studied further by Newman and Ravishankar in \cite{Newman_et_al2006}.

However, one may need a true self-repelling motion above a more general kind of ``barrier'' than the abscissa axis or the Brownian motion. For example, if one studies the ``true'' self-avoiding walk of \cite{Toth1995}, but setting nonzero initial local times, the resulting process would be expected to converge to a true self-repelling motion above a ``barrier'' given by the initial local times profile. The need for a true self-repelling motion above a different barrier also arose in the study of the \emph{nearest-neighbor lifted TASEP}, an interacting particle system in a class of models designed to accelerate computations in Monte Carlo methods. In \cite{Massoulie_et_al2025}, Massoulié, Erignoux, Toninelli and Krauth proved that the pointer that governs the dynamics of this model can move like a ``zero-temperature'' version of the ``true'' self-avoiding walk of \cite{Toth1995}, in which the walker always crosses the edge with the smallest local time, with initial local times depending on the initial particle configuration. This walk is also expected to converge to a true self-repelling motion above a ``barrier'' given by the initial local times profile.

In this paper, we construct a true self-repelling motion above extremely general barriers. The extension of the construction of \cite{Toth_et_al1998} to any general barrier may seem obvious at first glance, but is much more complex than it seems. Indeed, a general barrier, especially if it is irregular, may create a number of indesirable effects near it that we have to deal with. This fact was already noted by Soucaliuc, Tóth and Werner \cite{Soucaliuc_et_al2000}, which considered deterministic barriers and constructed \emph{finite} systems of forward and backward lines reflected/absorbed on them (but neither infinite systems nor true self-repelling motion). They had to require these barriers to be Lipschitz, and noted that their construction failed if they were not. 

In this work, we consider barriers which are random continuous functions $\lambda : (-\infty,+\infty) \to \mathds{R}$. We also choose a random $\chi$; our forward lines will be Brownian motions reflected on $\lambda$ at the left of $\chi$ and absorbed by $\lambda$ at the right of $\chi$ (in the classical process $\chi$ is $0$); such processes are called $(\lambda,\chi)$-RABs. Without any other assumption on the barrier, we are able to construct an uncountable system of coalescing forward lines (Definition \ref{def_forward_lines}). We then construct an uncountable system of backward lines entwined with the forward lines (Definition \ref{def_backward_lines}) under an additional condition much more lenient than the Lipschitz property of \cite{Soucaliuc_et_al2000}: we only ask the barrier $\lambda$ to ``behave well immediately at the left of $\chi$'' (see Definition \ref{def_nice_barriers}). As we explain in Remark \ref{rem_nice_condition}, such a condition is necessary. Finally, we construct the true self-repelling motion above the barrier $\lambda$ stemming from these forward and backward lines, which we call \emph{$(\lambda,\chi)$-true self-repelling motion} (Definition \ref{def_TRSM}), assuming a Brownian motion starting above $\lambda$ always meets $\lambda$ (see Definition \ref{def_good_barriers}), which is also a necessary condition. Our very general conditions on $\lambda$ include the abscissa axis, Brownian motion, and deterministic Lipschitz barriers; in particular this work is also a written down construction of the true self-repelling motion with a Brownian barrier. In addition to these constructions, we prove a number of properties for the forward lines, backward lines, and $(\lambda,\chi)$-true self-repelling motion, choosing the most relevant ones while trying to contain the length of this work. We also show some results on $(\lambda,\chi)$-RABs which may be of independent interest.

This paper unfolds as follows. In Section \ref{sec_def}, we state the definitions and properties of the systems of forward lines, backward lines, and the $(\lambda,\chi)$-true self-repelling motion. In Section \ref{sec_B_RAB}, we gather some known results on Brownian motions, and original ones on $(\lambda,\chi)$-RABs, which we use throughout the sequel. In Section \ref{sec_cons_forward} we prove the properties of the forward lines. In Section \ref{sec_cons_backward} we show the properties of the backward lines. Finally, the construction and properties of the $(\lambda,\chi)$-true self-repelling motion are proven in Section \ref{sec_TSRM_proofs}. 

\section{Definitions and properties}\label{sec_def}

In this section, we give the definitions and properties of all the objects constructed in this work: forward lines in Section \ref{sec_def_forward}, backward lines in Section \ref{sec_def_backward}, and $(\lambda,\chi)$-true self-repelling motion in Section \ref{sec_TSRM}. 

\subsection{Forward lines}\label{sec_def_forward}

This section is devoted to the definition and properties of the forward lines. In order to construct these lines, we need several auxiliary definitions. We begin by defining our general barriers in Definition \ref{def_barrier}. We then introduce the $(\lambda,\chi)$-RABs as Brownian motions reflected by $\lambda$ at the left of $\chi$ and absorbed by $\lambda$ at the right of $\chi$, in Definition \ref{def_RAB}. Then we define finite families of independent coalescing $(\lambda,\chi)$-RABs, or $(\lambda,\chi)$-FICRABs, in Definition \ref{def_FICRAB} and use this definition to construct the forward lines, which are an infinite family of independent coalescing $(\lambda,\chi)$-RABs (Definition \ref{def_forward_lines}). We then state properties of the forward lines in Theorems \ref{thm_cons_forward} and \ref{thm_forward_bis}. 

\begin{definition}\label{def_barrier}
 A \emph{barrier} is a couple $(\lambda,\chi)$ where $\lambda : \mathds{R} \mapsto \mathds{R}$ is a random continuous function and $\chi$ is a real random variable.
\end{definition}

It may seem more natural to consider $\lambda$ alone as the barrier, but later the properties we ask of a ``nice barrier'' will depend on $\chi$, so we chose to include the latter in the definition of the barrier. In what follows, we will assume $(\lambda,\chi)$ is a barrier, and we write $\mathds{P}_{\lambda,\chi}$ for the probability conditionally to $(\lambda,\chi)$.  We denote $\mathds{R}^2_\lambda = \{(x,h) \in \mathds{R}^2 \,|\, h > \lambda(x)\}$. Brownian motions will always have the same variance $v>0$.

\begin{definition}\label{def_RAB}
 For any $(x,h)\in\mathds{R}^2$, the \emph{Brownian motion reflected/absorbed on $(\lambda,\chi)$}, or \emph{$(\lambda,\chi)$-RAB}, \emph{starting from $(x,h)$} is a process $(R_y)_{y \geq x}$ defined as follows. Let $x' \leq x$, let $(W_{y})_{y \geq x'}$ a Brownian motion independent from $(\lambda,\chi)$, we will say $(R_y)_{y \geq x}$ is \emph{driven by $(W_{y})_{y \geq x'}$}. If $h \leq \lambda(x)$, then $(R_y)_{y \geq x}$ is not defined. If $h > \lambda(x)$:
 \begin{itemize}
  \item If $x \geq \chi$, then $(R_y)_{y \geq x}$ is just a Brownian motion starting from $h$ absorbed by $\lambda$. If we denote $Y=\inf\{y \geq x \,|\, W_y-W_x +h=\lambda(y)\}$, then for $y \in [x,Y]$ we set $R_y=W_y-W_x+h$ and for $ y\geq Y$ we have $R_y=\lambda(y)$.
  \item If $x < \chi$, on $[x,\chi)$, $(R_y)_{y \geq x}$ is a Brownian motion reflected on $\lambda$ starting from $h$, and on $[\chi,+\infty)$, $(R_y)_{y \geq x}$ is a Brownian motion absorbed by $\lambda$. If we denote $Y=\inf\{y \in [x,\chi] \,|\, W_y-W_x+h=\lambda(y)\}$, for $y\in[x,\chi]$, if $y \leq Y$ we set $R_y=W_y-W_x+h$ and if $y \geq Y$ we set $R_y=W_y+\sup_{Y \leq z \leq y}(\lambda(z)-W_z)$. If we denote $Y'=\inf\{y \geq \chi \,|\, W_y-W_\chi+R_\chi=\lambda(y)\}$, then for $y \in [\chi,Y']$ we have $R_y=W_y-W_\chi+R_\chi$ and for $y\geq Y'$ we have $R_y=\lambda(y)$.
 \end{itemize}
\end{definition}

\begin{remark}\label{rem_barrier_RABs}
 In most of the paper, it will be convenient that the $(\lambda,\chi)$-RABs starting from $(x,\lambda(x))$, $x\in\mathds{R}$ stay undefined. However, at some points we will still need the notion of $(\lambda,\chi)$-RABs starting from such a point. We will call them \emph{barrier-starting $(\lambda,\chi)$-RAB starting from $(x,\lambda(x))$}, and define them as follows. Let $x' \leq x$, let $(W_{y})_{y \geq x'}$ a Brownian motion independent from $(\lambda,\chi)$. If $x \geq \chi$, $R_y=\lambda(y)$ for any $y \geq x$. If $x < \chi$, for $y\in[x,\chi]$, $R_y=W_y+\sup_{x \leq z \leq y}(\lambda(z)-W_z)$ and for $y \geq \chi$, $R_y$ is defined as for standard $(\lambda,\chi)$-RABs. 
\end{remark}

 We now construct finite families of independent coalescing $(\lambda,\chi)$-RABs.  We will do this inductively: if we already have a family of $n-1$ processes, we build the $n$-th by considering a $(\lambda,\chi)$-RAB independent from the previous ones and following it until it meeets one of them, after which we follow the process it met. 

\begin{definition}\label{def_FICRAB}
For any $p \in \mathds{N}^*$, any $(x_1,h_1),...,(x_p,h_p)\in\mathds{R}^2$, a \emph{finite family of independent coalescing $(\lambda,\chi)$-RABs}, or \emph{$(\lambda,\chi)$-FICRAB}, \emph{starting from $(x_1,h_1),...,(x_p,h_p)$} is a family $((C_{1,y})_{y \geq x_1},...,(C_{p,y})_{y \geq x_p})$ defined as follows. Let $(R_{1,y})_{y \geq x_1}$,..., $(R_{p,y})_{y \geq x_p}$ be $(\lambda,\chi)$-RABs driven by independent Brownian motions and starting respectively from $(x_1,h_1),...,(x_p,h_p)$. For $i\in\{1,...,p\}$, if $(R_{i,y})_{y \geq x_i}$ is not defined, $(C_{i,y})_{y \geq x_i}$ is not defined. We set $(C_{1,y})_{y \geq x_1} = (R_{1,y})_{y \geq x_1}$, and for $j\in\{2,...,p\}$ we define by induction $\omega_j=\inf\{x \geq x_j \,|\, R_{j,x}\in\{C_{1,x},...,C_{j-1,x}\}\}$, $\nu_j=\min\{k \in \{1,...,j-1\} \,|\, R_{j,\omega_j}=C_{k,\omega_j}\}$, and $C_{j,x}=R_{j,x}$ for $x \in [x_j,\omega_j]$, $C_{j,x}=C_{\nu_j,x}$ for $x \in [\omega_j,+\infty)$. The distribution of $(\lambda,\chi,(C_{1,y})_{y \geq x_1},...,(C_{p,y})_{y \geq x_p})$ actually does not depend on the order in which the coalescence is performed.
\end{definition}

We can now define the system of forward lines. Let $\mathds{D}$ be the set of dyadic rational numbers.  
 
\begin{definition}\label{def_forward_lines}
A \emph{system of forward lines above $(\lambda,\chi)$} is a family of random maps $(\Lambda_{(x,h)}(.))_{(x,h)\in\mathds{R}^2}$ defined as follows. Let $\{(\tilde x_n,\tilde h_n)\}_{n \in \mathds{N}}$ be an enumeration of $\mathds{D}^2$. $(\Lambda_{(\tilde x_n,\tilde h_n)}(.))_{n\in \mathds{N}}$ is constructed inductively as a countable family of independent coalescing $(\lambda,\chi)$-RABs starting from $(\tilde x_n,\tilde h_n)$, $n\in\mathds{N}$. Then, for any $(x,h)\in\mathds{R}^2$,  if $h \leq \lambda(x)$ then $\Lambda_{(x,h)}(.)$ is not defined, and if $h>\lambda(x)$, then $\Lambda_{(x,h)}(.)$ is defined on $[x,+\infty)$ by $\Lambda_{(x,h)}(y) = \sup\{\Lambda_{(\tilde x_n,\tilde h_n)}(y) \,|\, \tilde x_n < x, \Lambda_{(\tilde x_n,\tilde h_n)}(x) < h\}$ (it can be shown as in \cite{Toth_et_al1998} that almost surely, for all $(x,h)\in\mathds{R}_\lambda^2$ this set is not empty, and this definition is compatible with the definition of the $\Lambda_{(\tilde x_n,\tilde h_n)}$). 
\end{definition}

In what follows, $(\Lambda_{(x,h)}(.))_{(x,h)\in\mathds{R}^2}$ will be a system of forward lines above $(\lambda,\chi)$. The following theorem states a first set of properties of this system and justifies they form ``an infinite system of independent coalescing $(\lambda,\chi)$-RABs''.  

\begin{theorem}\label{thm_cons_forward}
 The following properties hold.
\begin{itemize}
 \item For any $p \in \mathds{N}^*$, if $(x_1,h_1),...,(x_p,h_p)\in\mathds{R}^2$, then $(\Lambda_{(x_1,h_1)},...,\Lambda_{(x_p,h_p)})$ is a $(\lambda,\chi)$-FICRAB starting from $(x_1,h_1),...,(x_p,h_p)$.
 \item Almost surely, for all $(x,h)\in \mathds{R}_\lambda^2$, then we have $\Lambda_{(x,h)}(x)=h$. 
 \item Almost surely, for all $(x,h),(x',h')\in\mathds{R}_\lambda^2$, if $\max(x,x') \leq y \leq z$, then $\Lambda_{(x,h)}(y) < \Lambda_{(x',h')}(y)$ implies $\Lambda_{(x,h)}(z) \leq \Lambda_{(x',h')}(z)$. 
 \item Almost surely, for any $x \leq y$, the mapping $: h \mapsto \Lambda_{(x,h)}(y)$ is non-decreasing and left-continuous on $(\lambda(x),+\infty)$.
\end{itemize}
In addition, these properties characterize the law of $(\Lambda_{(x,h)})_{(x,h)\in \mathds{R}^2}$. 
\end{theorem}

We now state another set of important properties of the forward lines. For any $x<y$ in $\mathds{R}$, we denote $M(x,y)=\{\Lambda_{(z,h)}(y)\,|\,z<x, (z,h)\in\mathds{R}_\lambda^2\}$, the ``trace'' at $y$ of the forward lines starting before $x$. We also denote $M(x)=M(x,x)$.

\begin{theorem}[Properties of the forward lines]\label{thm_forward_bis}
The following holds almost surely.
 \begin{itemize}
  \item For all $x\in\mathds{R}$, $M(x)$ is dense in $[\lambda(x),+\infty)$.
  \item For any $x < y$, $M(x,y)$ is locally finite and unbounded.
  \item For all $(x,h)\in \mathds{R}_\lambda^2$, $\varepsilon>0$, there exists $n\in\mathds{N}$ so that $\tilde x_n <x$, $\Lambda_{(\tilde x_n,\tilde h_n)}(x) \in (h-\varepsilon,h)$ and for $y \geq x+\varepsilon$, $\Lambda_{(x,h)}(y)=\Lambda_{(\tilde x_n,\tilde h_n)}(y)$. 
  \item For all $x<y$, $M(x,y)=\{\Lambda_{(\tilde x_n,\tilde h_n)}(y)\,|\,n\in\mathds{N}, \tilde x_n < x,(\tilde x_n,\tilde h_n)\in\mathds{R}_\lambda^2\}$. 
  \item For any $(x,h)\in\mathds{R}_\lambda^2$ the function $: y \mapsto \Lambda_{(x,h)}(y)$ is continuous on $[x,+\infty)$.
 \end{itemize}
\end{theorem}

The proof of Theorems \ref{thm_cons_forward} and \ref{thm_forward_bis} are treated in Section \ref{sec_cons_forward}.

\subsection{Backward lines}\label{sec_def_backward}
In this section we state the definition of the backward lines (Definition \ref{def_backward_lines}) and an additional characterization (Proposition \ref{prop_backward_char}). We then give the main property of the backward lines, Theorem \ref{thm_cons_backward}, which claims they behave like forward lines ``oriented backwards''. We also state a result on the number of forward and backward lines ``incoming at a given point'', Proposition \ref{prop_types}. We begin by formulating the additional property on $(\lambda,\chi)$ which is needed to construct the backward lines.

\begin{definition}\label{def_nice_barriers}
 The barrier $(\lambda,\chi)$ is said to be \emph{nice} if almost surely, there does \emph{not} exist $\varepsilon>0$ so that for any $x \in [\chi-\varepsilon,\chi]$ we have $\lambda(x)-\lambda(\chi) \leq -\sqrt{\chi-x}$.
\end{definition}

In the rest of this section and in Section \ref{sec_TSRM} we will always assume $(\lambda,\chi)$ nice. We can now construct the backward lines. 

\begin{definition}\label{def_backward_lines}
 For any barrier $(\lambda,\chi)$ and system of forward lines $(\Lambda_{(x,h)}(.))_{(x,h)\in\mathds{R}^2}$ above $(\lambda,\chi)$, the \emph{system of backward lines above $(\lambda,\chi)$} is the family of random maps $(\Lambda_{(x,h)}^*(.))_{(x,h)\in \mathds{R}^2}$ defined as follows. For any $(x,h) \in \mathds{R}^2$, if $h \leq \lambda(x)$ then $\Lambda_{(x,h)}^*(.)$ is not defined, and if $h >\lambda(x)$ we define $\Lambda_{(x,h)}^* : (-\infty,x] \mapsto \mathds{R}$ by $\Lambda_{(x,h)}^*(y)=\sup\{\Lambda_{(\tilde x_n,\tilde h_n)}(y) \,|\, \tilde x_n < y, \Lambda_{(\tilde x_n,\tilde h_n)}(x) < h\}$, which is set to $\lambda(y)$ if the set is empty.
\end{definition}

The following characterization of the backward lines is easy to see.

\begin{proposition} \label{prop_backward_char}
 Almost surely, for all $(x,h)\in \mathds{R}^2_\lambda$, $y<x$, we have $\Lambda_{(x,h)}^*(y)=\sup\{h' > \lambda(y) \,|\, \Lambda_{(y,h')}(x) < h\}$, and for all $(x,h)\in \mathds{R}^2_\lambda$, $y>x$ we have $\Lambda_{(x,h)}(y)=\sup\{h' > \lambda(y) \,|\, \Lambda_{(y,h')}^*(x) < h\}$ (where the sup is set to $\lambda(y)$ when the set is empty). 
\end{proposition}

The most important property of the system of backward lines is that it is ``a system of forward lines oriented backwards'', which is the following.

\begin{theorem}\label{thm_cons_backward}
 $(\Lambda_{(-x,h)}^*(-.))_{(x,h)\in \mathds{R}^2}$ is a system of forward lines above $(\lambda(-.),-\chi)$. 
\end{theorem}

The proof of Theorem \ref{thm_cons_backward} can be found in Section \ref{sec_cons_backward}.

\begin{remark}\label{rem_nice_condition}
Theorem \ref{thm_cons_backward} will not hold if we only assume $\lambda$ continuous. Indeed, if we choose $\chi=0$ and $\lambda(t)=-|t|^{1/3}$ for $t \geq 0$, which do not satisfy Definition \ref{def_nice_barriers}, we will see the backward lines may hit $(0,\lambda(0))$ but not be absorbed by $\lambda$ there as they should, hence they do not have the correct distribution. This can be seen as follows. Results in Section 4.1 of \cite{Soucaliuc_et_al2000} imply a $(\lambda,0)$-RAB starting at the left of $0$ can hit the point $(0,\lambda(0))$ with positive probability, so we have a positive probability to get $\Lambda_{(\tilde x_n,\tilde h_n)}(0)=\lambda(0)$ for some $n\in\mathds{N}$. If this happens, for any $x>0$, $h>\lambda(x)$, we have $\Lambda_{(\tilde x_n,\tilde h_n)}(x)=\lambda(x)<h$, hence it would be possible to have $\Lambda_{(x,h)}^*(0)=\lambda(0)$, but for $\tilde x_n < y < 0$, $\Lambda_{(x,h)}^*(y) \geq \Lambda_{(\tilde x_n,\tilde h_n)}(y)$ which can be strictly above $\lambda(y)$, hence $\Lambda_{(x,h)}^*$ would hit $(0,\lambda(0))$, but not be absorbed there. In Definition \ref{def_nice_barriers}, we did not try to give an optimal condition, instead settling for a condition easy to state and to work with. This condition is satisfied if $\lambda$ is the abscissa axis, a Brownian motion independent of $\chi$, or a Lipschitz function as in \cite{Soucaliuc_et_al2000}.
\end{remark}

 We will also state a property of the systems of forward lines and backward lines, concerning the possible ``number of forward and backward lines incoming at a point''. In order to do that, we need to introduce some notation. For any $(x,h)\in\mathds{R}_\lambda^2$, we define the ``number of forward lines incoming at $(x,h)$'' as $I(x,h)=\lim_{y \to x, y < x}\mathrm{card}(\{p \in \mathds{N}\,|\,\exists\, (x_1,h_1),...,(x_p,h_p)\in\mathds{R}_\lambda^2$ such that $\forall\, i \in \{1,...,p\}, x_i \leq y, \Lambda_{(x_i,h_i)}(x)=h, \forall \, z \in[y,x), \Lambda_{(x_1,h_1)}(z) < \cdots < \Lambda_{(x_p,h_p)}(z)\})$. We also define the ``number of backward lines incoming at $(x,h)$'' to be $I^*(x,h)=\lim_{y \to x, y > x}\mathrm{card}(\{p \in \mathds{N}\,|\,\exists\, (x_1,h_1),...,(x_p,h_p)\in\mathds{R}_\lambda^2$ such that $\forall\, i \in \{1,...,p\}, x_i \geq y, \Lambda_{(x_i,h_i)}^*(x)=h, \forall \, z \in(x,y], \Lambda_{(x_1,h_1)}^*(z) < \cdots < \Lambda_{(x_p,h_p)}^*(z)\})$. For any $(x,h)\in\mathds{R}_\lambda^2$, the pair of integers $[I(x,h),I^*(x,h)]$ is called the \emph{type of $(x,h)$}. The following result sets heavy restrictions on the possible types of the points. 

\begin{proposition}\label{prop_types}
We have the following results on types.
 \begin{itemize}
  \item For any $(x,h)\in\mathds{R}^2$, almost surely if $h>\lambda(x)$ then $(x,h)$ is of type $[0,0]$.
  \item For any $x\in\mathds{R}$, almost surely for any $h>\lambda(x)$, the point $(x,h)$ is of type $[0,0]$, $[1,0]$ or $[0,1]$.
  \item Almost surely for all $(x,h)\in\mathds{R}_\lambda^2$, the point $(x,h)$ is of type $[0,0]$, $[1,0]$, $[0,1]$, $[1,1]$, $[2,0]$ or $[0,2]$. 
 \end{itemize}
\end{proposition}

Proposition \ref{prop_types} can be proven in the same way as Proposition 2.4 of \cite{Toth_et_al1998} given the results in Section \ref{sec_backward_coal}.

\subsection{$(\lambda,\chi)$-true self-repelling motion}\label{sec_TSRM}

In this section, we begin by defining the $(\lambda,\chi)$-true self-repelling motion (Proposition \ref{prop_P_singleton} and Definition \ref{def_TRSM}) and stating a first set of its properties in Proposition \ref{prop_X_basic}. We then construct a ``local time'' for the $(\lambda,\chi)$-true self-repelling motion (Definition \ref{def_loctime} and Theorem \ref{thm_loctime}) and study its properties, in particular its relationship with the lines system (Proposition \ref{prop_L_basic} and Theorem \ref{thm_RK}). All the proofs are postponed to Section \ref{sec_TSRM_proofs}. We begin by stating the additional condition on the barrier necessary to construct the $(\lambda,\chi)$-true self-repelling motion.

\begin{definition}\label{def_good_barriers}
 A barrier $(\lambda,\chi)$ is called a \emph{good barrier} if it is nice and when for all $(x,h)\in\mathds{R}^2$, if $(W_y)_{y \in\mathds{R}}$ is a two-sided Brownian motion independent from $(\lambda,\chi)$ with $W_x=h$, then $\mathds{P}(h>\lambda(x),\forall\,y>x, W_y>\lambda(y))=0$ and $\mathds{P}(h>\lambda(x),\forall\,y<x, W_y>\lambda(y))=0$. 
\end{definition}

In the rest of this section, $(\lambda,\chi)$ will always be a good barrier. We now need to define several notations. As in \cite{Toth_et_al1998}, the notation $\bar \Lambda$ will encompass both forward and backward lines, as follows. For any $(x,h)\in\mathds{R}^2$, if $h \leq \lambda(x)$ then $\bar\Lambda_{(x,h)}(.)$ is not defined, and if $h > \lambda(x)$ then $\bar\Lambda_{(x,h)} : \mathds{R} \mapsto \mathds{R}$ is defined thus: for all $y\in\mathds{R}$, we set $\bar\Lambda_{(x,h)}(y)=\Lambda_{(x,h)}(y)$ if $y \geq x$ and $\bar\Lambda_{(x,h)}(y)=\Lambda_{(x,h)}^*(y)$ if $y < x$. This allows us to introduce, for any $(x,h)\in\mathds{R}_\lambda^2$, the quantity $T(x,h)=\int_{-\infty}^{+\infty}(\bar\Lambda_{(x,h)}(y)-\lambda(y))\mathrm{d}y$. The additional property of Definition \ref{def_good_barriers} is used to enforce the fact that almost surely, for all $(x,h)\in\mathds{R}_\lambda^2$, $T(x,h)$ is finite (see Lemma \ref{lem_D_bounded}). For any $t \geq 0$, we define the set $P_t=\bigcap_{\varepsilon>0} \overline{\{(x,h)\in\mathds{R}_\lambda^2, T(x,h)\in(t-\varepsilon,t+\varepsilon)\}}$ (where $\overline{A}$ denotes the closure of a set $A$). We then have the following. 

\begin{proposition}\label{prop_P_singleton}
 Almost surely, for any $t\in[0,+\infty)$, $P_t$ is a singleton.  
\end{proposition}

\begin{definition}\label{def_TRSM}
 For any $t \geq 0$, we denote $P_t=\{(X_t,H_t)\}$. The process $(X_t)_{t \geq 0}$ is called \emph{$(\lambda,\chi)$-true self-repelling motion}. 
\end{definition}

We now state a first set of properties for the process $((X_t,H_t))_{t \geq 0}$.

\begin{proposition}\label{prop_X_basic}
 The following holds almost surely.
 \begin{itemize}
  \item For all $t \in [0,+\infty)$, we have $H_t \geq \lambda(X_t)$.
  \item $(X_t,H_t)_{t \geq 0}$ is continuous.
  \item For any $x\in\mathds{R}$, the set $\{t \geq 0 \,|\, X_t=x\}$ is unbounded. 
 \end{itemize}
\end{proposition}

\begin{remark}
  The classical true self-repelling motion defined in \cite{Toth_et_al1998} has the following invariance properties: $(X_t,H_t)_{t \geq 0}$ has the same distribution as $(-X_t,H_t)_{t \geq 0}$ and as $(a^{2/3}X_{t/a},a^{1/3}H_{t/a})_{t \geq 0}$. We cannot expect such invariance properties in our more general context since $\lambda$ will not necessarily have them.
\end{remark}

We now construct the ``local time'' $(L_t(.))_{t \geq 0}$ of the process $(X_t)_{t \geq 0}$ and state some of its properties.

\begin{definition}\label{def_loctime}
For any $t \geq 0$, $x\in\mathds{R}$, we define $L_t(x)=\sup\{h > \lambda(x) \,|\,T(x,h) < t\}$, setting the sup to $\lambda(x)$ if the set is empty.
\end{definition}

\begin{theorem}\label{thm_loctime}
 $(L_t-\lambda)_{t \geq 0}$ is the \emph{local time} of the process $(X_t)_{t \geq 0}$ in the sense that almost surely, for any Borel set $A \subset \mathds{R}$, $t \geq 0$, we have $\int_0^t \mathds{1}_{\{X_s\in A\}}\mathrm{d}s = \int_{A}(L_t(x)-\lambda(x))\mathrm{d}x$.
\end{theorem}

\begin{proposition}\label{prop_L_basic}
 The following holds almost surely.
 \begin{itemize}
  \item For all $t \geq 0$, $x\in\mathds{R}$, $L_t(x)<+\infty$.
  \item For any $x\in\mathds{R}$, the function $t \mapsto L_t(x)$ is non-decreasing and continuous on $[0,+\infty)$. 
  \item For all $t \geq 0$, $H_t=L_t(X_t)$. 
 \end{itemize}
\end{proposition}

\begin{remark}
 In \cite{Toth_et_al1998}, for any $t \geq 0$, the function $L_t$ has compact support. This cannot be the case here, since $L_0=\lambda$ which does not necessarily has compact support. However, Lemmas \ref{lem_D_bounded} and \ref{lem_T_large} allow to prove that almost surely for all $t \geq  0$, the function $L_t-\lambda$ has compact support.  
\end{remark}

We now give a ``Ray-Knight Theorem'' which shows a relationship between the local times $(L_t(.))_{t \geq 0}$ and the lines system. For all $(x,h)\in\mathds{R}^2$ with $h \geq \lambda(x)$, we define $T^+(x,h)=\lim_{\varepsilon \to 0,\varepsilon>0}T(x,h+\varepsilon)$ (which exists since $T(x,.)$ is increasing on $(\lambda(x),+\infty)$). Then the following stems directly from the definitions. 

\begin{theorem}[Ray-Knight Theorem]\label{thm_RK}
The following holds almost surely.
\begin{itemize}
 \item For all $(x,h)\in\mathds{R}_\lambda^2$, $y \in \mathds{R}$, we have $L_{T(x,h)}(y)=\bar\Lambda_{(x,h)}(y)$. 
 \item For all $(x,h)\in\mathds{R}^2$ with $h \geq \lambda(x)$, we have $L_{T^+(x,h)}(y)=\lim_{\varepsilon \to 0, \varepsilon>0}\bar\Lambda_{(x,h+\varepsilon)}(y)$. 
\end{itemize}
\end{theorem}

\section{Brownian motion and $(\lambda,\chi)$-RAB}\label{sec_B_RAB}

 In this section we gather a set of results on Brownian motion and $(\lambda,\chi)$-RAB which are used thoughout this paper. The Brownian motion results are not original and are only given here for reference, but the $(\lambda,\chi)$-RAB results are new to the author's knowledge and may be of independent interest.
 
 \subsection{Brownian motion results}
 
  The first two lemmas give bounds on the fluctuations of a Brownian motion, an upper bound for Lemma \ref{lem_BM_max} and a lower bound for Lemma \ref{lem_BM_max2}. Lemma \ref{lem_3BM} states an upper bound on the probability three Brownian motions do not intersect. 
  
 \begin{lemma}\label{lem_BM_max}
  Let $x\in\mathds{R}$, $(W_{y})_{y \geq x}$ a Brownian motion (with variance $v$), $x \leq y_1 < y_2$ and $a>0$. Then $\mathds{P}(\exists\, y \in [y_1,y_2],|W_{y}-W_{y_1}| \geq a) \leq 2\frac{\sqrt{2v(y_2-y_1)}}{a\sqrt{\pi}}\exp(-\frac{a^2}{2v(y_2-y_1)})$.
 \end{lemma}
 
 For a proof, see for example Remark 2.22 of \cite{Morters_Peres_Brownian_motion}.
 
 \begin{lemma}\label{lem_BM_max2}
  Let $x\in\mathds{R}$, $(W_{y})_{y \geq x}$ a Brownian motion (with variance $v$), $x \leq y_1 < y_2$ and $a>0$. $\mathds{P}(\max_{y \in [y_1,y_2]}(W_{y}-W_{y_1}) \leq a) \leq \frac{2a}{\sqrt{2\pi v(y_2-y_1)}}$. 
 \end{lemma}
 
 Lemma \ref{lem_BM_max2} comes from the fact that the maximum of the Brownian motion on an interval is the absolute value of a Gaussian random variable (see for example Theorem 2.21 of \cite{Morters_Peres_Brownian_motion}). 
 
 The next lemma gives an upper bound on the probability three Brownian motions do not intersect. It is not original, but it is surprisingly difficult to find a proof spelled out in the litterature, hence we give an argument.
 
 \begin{lemma}\label{lem_3BM}
  There exists a constant $\tilde C>0$ (depending on $v$) so that for any $(x,h)\in\mathds{R}^2$, $\varepsilon>0$, $\delta>0$, if $(W_1(y))_{y \geq x}$, $(W_2(y))_{y \geq x}$, $(W_3(y))_{y \geq x}$ are independent Brownian motions so that $W_1(x)=h$, $W_2(x)=h+\delta$, $W_3(x)=h+2\delta$, if $Y=\inf\{y \geq x \,|\, \exists\, i,j \in \{1,2,3\}, W_i(y)=W_j(y)\}$, then $\mathds{P}(Y \geq x + \varepsilon) \leq \tilde C (\frac{\delta}{\sqrt{\varepsilon}})^3$. 
 \end{lemma}

 \begin{proof}
  Defining $\gamma(t)=2(1-\int_{-\infty}^t\frac{1}{\sqrt{2\pi}}e^{-s^2/2}ds)$ for any $t\in\mathds{R}$, Theorem 3.1 of \cite{OConnell_et_al1992} states that $\mathds{P}(Y \leq x + \varepsilon)=2\gamma(\frac{\delta}{\sqrt{2v\varepsilon}})-\gamma(\frac{2\delta}{\sqrt{2v\varepsilon}})$ (note that we cannot use directly their Corollary 3.2 (ii), since it only gives an equivalent for $\delta \to 0$, $\varepsilon$ fixed). Now, $\gamma'(t)=-\frac{2}{\sqrt{2\pi}}e^{-t^2/2}$, $\gamma''(t)=\frac{2t}{\sqrt{2\pi}}e^{-t^2/2}$, $\gamma'''(t)=\frac{2}{\sqrt{2\pi}}e^{-t^2/2}-\frac{2t^2}{\sqrt{2\pi}}e^{-t^2/2}$, hence $\gamma(t)=1-\frac{2}{\sqrt{2\pi}}t+\frac{1}{3\sqrt{2\pi}}t^3+o(t^3)$, so when $\frac{\delta}{\sqrt{\varepsilon}}$ is small enough, $\mathds{P}(Y \leq x + \varepsilon)=2-\frac{4}{\sqrt{2\pi}}\frac{\delta}{\sqrt{2v\varepsilon}}+\frac{2}{3\sqrt{2\pi}}(\frac{\delta}{\sqrt{2v\varepsilon}})^3-1+\frac{2}{\sqrt{2\pi}}\frac{2\delta}{\sqrt{2v\varepsilon}}-\frac{1}{3\sqrt{2\pi}}(\frac{2\delta}{\sqrt{2v\varepsilon}})^3+o((\frac{\delta}{\sqrt{\varepsilon}})^3)=1-\frac{6}{3\sqrt{2\pi}}(\frac{\delta}{\sqrt{2v\varepsilon}})^3+o((\frac{\delta}{\sqrt{\varepsilon}})^3)$, so $\mathds{P}(Y > x + \varepsilon)=\frac{1}{2\sqrt{\pi v^3}}(\frac{\delta}{\sqrt{\varepsilon}})^3+o((\frac{\delta}{\sqrt{\varepsilon}})^3)$, which suffices to find $\tilde C>0$ so that $\mathds{P}(Y > x + \varepsilon) \leq \tilde C (\frac{\delta}{\sqrt{\varepsilon}})^3$ for all values of $\frac{\delta}{\sqrt{\varepsilon}}$. Moreover, the probability two of the $W_i$ meet precisely at $x+\varepsilon$ is 0, so $\mathds{P}(Y=x+\varepsilon)=0$, which ends the proof of the lemma.
 \end{proof}
 
 \subsection{$(\lambda,\chi)$-RAB results}
 
 In this section, we first prove an upper bound on the fluctuations of a $(\lambda,\chi)$-RAB as long as it stays away from the barrier (Lemma \ref{lem_RAB_max}). We then study the atoms of the marginals of $(\lambda,\chi)$-RABs, showing the marginal at $y$ has no atom except perhaps at $\lambda(y)$ (Lemma \ref{lem_cons_prob0}) and that if $(\lambda,\chi)$ is nice, the marginal at $\chi$ has no atom at $\lambda(\chi)$ (Lemma \ref{lem_meeting_nice}). Finally, we prove that it is impossible for a $(\lambda,\chi)$-RAB to stay equal to $\lambda$ on an entire interval at the left of $\chi$ (Lemma \ref{lem_no_stick_lambda}).
 
 \begin{lemma}\label{lem_RAB_max}
  Let $x\in\mathds{R}$, $(R_{y})_{y \geq x}$ a $(\lambda,\chi)$-RAB, $x \leq y_1 < y_2$ and $a>0$. Then $\mathds{P}(\{R_{y_1}-a \geq \max_{y\in[y_1,y_2]}\lambda(y)\} \cap \{ \exists\, y \in [y_1,y_2],|R_{y}-R_{y_1}| \geq a\}) \leq 2\frac{\sqrt{2v(y_2-y_1)}}{a\sqrt{\pi}}\exp(-\frac{a^2}{2v(y_2-y_1)})$.
 \end{lemma}
 
 \begin{proof}
  This proof uses the fact a $(\lambda,\chi)$-RAB away from the barrier behaves like a Brownian motion, and the bound on the fluctuations of a Brownian motion given in Lemma \ref{lem_BM_max}. We denote by $(W_{y})_{y \geq x}$ the Brownian motion driving $(R_{y})_{y \geq x}$. We notice that if $R_{y_1}-a \geq \max_{y\in[y_1,y_2]}\lambda(y)$ and for all $y \in [y_1,y_2]$, $|W_{y}-W_{y_1}| < a$, then $R$ does not meet $\lambda$ in $[y_1,y_2]$, hence the increments of $R$ in this interval are those of $W$, hence for all $y \in [y_1,y_2]$, $|R_{y}-R_{y_1}| < a$. This implies $\mathds{P}(\{R_{y_1}-a \geq \max_{y\in[y_1,y_2]}\lambda(y)\} \cap \{ \exists\, y \in [y_1,y_2],|R_{y}-R_{y_1}| \geq a\}) \leq \mathds{P}(\{R_{y_1}-a \geq \max_{y\in[y_1,y_2]}\lambda(y)\} \cap \{ \exists\, y \in [y_1,y_2],|W_{y}-W_{y_1}| \geq a\}) \leq \mathds{P}(\exists\, y \in [y_1,y_2],|W_{y}-W_{y_1}| \geq a) \leq 2\frac{\sqrt{2v(y_2-y_1)}}{a\sqrt{\pi}}\exp(-\frac{a^2}{2v(y_2-y_1)})$ by Lemma \ref{lem_BM_max}. 
 \end{proof}
 
 The following lemma shows the marginal at $y$ of a $(\lambda,\chi)$-RAB has no atoms except perhaps at $\lambda(y)$. 
 
 \begin{lemma}\label{lem_cons_prob0}
 Let $(x,h)\in\mathds{R}^2$ and $(R_y)_{y \geq x}$ a $(\lambda,\chi)$-RAB starting from $(x,h)$, or let $x\in\mathds{R}$ and $(R_y)_{y \geq x}$ a barrier-starting $(\lambda,\chi)$-RAB starting from $(x,\lambda(x))$, then for any $y > x$, we have $\mathds{P}(a > \lambda(y),R_y=a)=0$.
\end{lemma}

\begin{proof}
We use the fact that when a $(\lambda,\chi)$-RAB is strictly above $\lambda$, it behaves like a Brownian motion which has no atoms. We denote $(W_z)_{z \geq x}$ the Brownian motion driving $(R_z)_{z \geq x}$. Then if $a > \lambda(y)$ and $R_y=a$, there exists some $z \in (x,y)$ rational so that for $y' \in [z,y]$, we have $R_{y'} > \lambda(y')$, hence $R_{y}-R_{z}=W_{y}-W_{z}$, hence $W_{y}-W_{z}=a-R_{z}$. This yields $\mathds{P}(a>\lambda(y),R_{y}=a) \leq \sum_{z \in (x,y) \cap \mathds{Q}}\mathds{P}(W_{y}-W_{z}=a-R_{z})=0$ since for $z \in (x,y)$, we have $W_{y}-W_{z}$ Gaussian independent from $R_{z}$. 
\end{proof}

The following lemma states the marginal at $\chi$ of a $(\lambda,\chi)$-RAB has no atom at $\lambda(\chi)$. In order to do that, we crucially need $(\lambda,\chi)$ to be a nice barrier.

\begin{lemma}\label{lem_meeting_nice}
 Let $(\lambda,\chi)$ be a deterministic nice barrier. Let $(x,h)\in\mathds{R}^2_\lambda$ with $x<\chi$, and let $(R_y)_{y \geq x}$ be a $(\lambda,\chi)$-RAB starting from $(x,h)$. Then $\mathds{P}(R_\chi=\lambda(\chi))=0$. 
\end{lemma}

\begin{proof}
 We denote by $(W_y)_{y \geq x}$ the Brownian motion driving $(R_y)_{y \geq x}$, and $Y=\inf\{y \geq x \,|\, W_y-W_x +h=\lambda(y)\}$. If $R_\chi=\lambda(\chi)$, there are two possibilities: either $Y=\chi$, which implies $W_\chi-W_x=\lambda(\chi)-h$, which has probability 0, or $Y<\chi$. Thus it is enough to prove $\mathds{P}(R_\chi=\lambda(\chi),Y<\chi)=0$. Furthermore, if $R_\chi=\lambda(\chi)$ and $Y<\chi$, then $W_\chi+\sup_{Y \leq y \leq \chi}(\lambda(y)-W_y)=\lambda(\chi)$, thus $\sup_{Y \leq y \leq \chi}((\lambda(y)-\lambda(\chi))-(W_y-W_\chi))=0$, hence $\lambda(y)-\lambda(\chi) \leq W_y-W_\chi$ for all $Y \leq y \leq \chi$. Moreover $(\lambda,\chi)$ is nice, hence there exists a strictly increasing sequence $x_n$ tending to $\chi$ so that $\lambda(x_n)-\lambda(\chi)>-\sqrt{\chi-x_n}$. Therefore, if $R_\chi=\lambda(\chi)$ and $Y<\chi$, there exists some $N \in \mathds{N}$ so that for $n\geq N$ we have $W_{x_n}-W_\chi > - \sqrt{\chi-x_n}$. Denoting $\mathcal{E}=\{\forall\, N \in \mathds{N},\exists\, n \geq N, W_{x_n}-W_\chi \leq - \sqrt{\chi-x_n}\}$, we deduce that it is enough to prove that $\mathds{P}(\mathcal{E})=1$. In addition, $\mathcal{E}$ is in the germ $\sigma$-algebra of the Brownian motion $(W_{\chi-y}-W_\chi)_{0 \leq z \leq \chi-x}$, hence $\mathds{P}(\mathcal{E})=0$ or 1 (see for example Theorem 2.7 of \cite{Morters_Peres_Brownian_motion}). Furthermore, $\mathds{P}(\mathcal{E})=\lim_{N \to +\infty}\mathds{P}(\exists\, n \geq N, W_{x_n}-W_\chi \leq - \sqrt{\chi-x_n}) \geq \lim_{N \to +\infty}\mathds{P}(W_{x_N}-W_\chi \leq - \sqrt{\chi-x_N})$, which is the probability a centered Gaussian random variable of variance $v$ is smaller than -1, which is positive. Hence $\mathds{P}(\mathcal{E})>0$, thus $\mathds{P}(\mathcal{E})=1$, which ends the proof.
\end{proof}

The last lemma of this section proves that in the part of the space where it is reflected, a $(\lambda,\chi)$-RAB cannot ``stick'' to $\lambda$, in the sense that it cannot stay equal to $\lambda$ on a whole open interval.

\begin{lemma}\label{lem_no_stick_lambda}
 Let $(\lambda,\chi)$ be a deterministic barrier. Let $(x,h)\in\mathds{R}^2_\lambda$ with $x<\chi$, and let $(R_y)_{y \geq x}$ be a $(\lambda,\chi)$-RAB starting from $(x,h)$. For any $x \leq a < b \leq \chi$, $\mathds{P}(\forall \, y \in [a,b], R_y=\lambda(y))=0$. 
\end{lemma}

\begin{proof}
 We will first prove it is enough to find some $y_0 \in [a,b)$ so that there exist $\varepsilon\in(0,b-y_0)$, $\kappa \in \mathds{R}$ so that for all $y\in[y_0,y_0+\varepsilon]$ we have $\lambda(y)-\lambda(y_0) \leq \kappa(y-y_0)$. Indeed, let $(W_y)_{y \geq x}$ be the Brownian motion driving $(R_y)_{y \geq x}$. For any $y\in[y_0,y_0+\varepsilon]$, we have $R_y -R_{y_0} \geq W_y-W_{y_0}$, hence $R_y -\lambda(y) \geq W_y-W_{y_0} +\lambda(y_0)-\lambda(y) \geq W_y-W_{y_0} - \kappa(y-y_0)$. Moreover, the law of iterated logarithm states that almost surely $\limsup_{y \to y_0,y>y_0}\frac{W_y-W_{y_0}}{\sqrt{2(y-y_0)\ln(\ln(1/(y-y_0)))}}=1$, hence almost surely there exists $y\in[y_0,y_0+\varepsilon]$ so that $W_y-W_{y_0} - \kappa(y-y_0)>0$, thus $R_y>\lambda(y)$, which is enough. 
 
 Consequently, we only have to find a suitable $y_0$. If there exists $y\in[a,b)$ so that $\lambda$ has a local maximum in $[y,b]$ at $y$, then we set $y_0=y$. If there is no such maximum, then $\lambda$ is strictly increasing on $[a,b]$. Indeed, if there were $a \leq y < y' \leq b$ with $\lambda(y) \geq \lambda(y')$, then $\lambda$ would have a global maximum on $[y,y']$ which would be a local maximum as above. Furthermore, since $\lambda$ is increasing on $[a,b]$, it is differentiable almost everywhere on $[a,b]$ by Lebesgue's Theorem (see Theorem 7.2 in \cite{Knapp2016}), so we can choose a point in $[a,b)$ at which $\lambda$ is differentiable, and this point is a suitable $y_0$.
\end{proof}
 
\section{Forward lines: proof of Theorems \ref{thm_cons_forward} and \ref{thm_forward_bis}}\label{sec_cons_forward}

Theorems \ref{thm_cons_forward} and \ref{thm_forward_bis}, which gather properties of the forward lines, were proven in the case of the classical true self-repelling motion in \cite{Toth_et_al1998} (Theorem 2.1 and Proposition 1.2). A large part of the proofs of \cite{Toth_et_al1998} can be extended to our setting, but the greater generality of our barriers requires changes at several points. Here we detail these modifications, which are in their Lemma 8.1, Equation (8.22), Lemma 8.2 and Equation (8.49).  

\subsection{Lemma 8.1 of \cite{Toth_et_al1998}}

The proof of part (i) of Lemma 8.1 in \cite{Toth_et_al1998} relies on an estimate on the probability two RABs with close starting points meet quickly, itself relying on the fact their barrier is constant. To compensate for the lack of such an estimate for $(\lambda,\chi)$-RABs, we prove they will meet before hitting the barrier, when still behaving as Brownian motions. This leads us to replace part (i) of Lemma 8.1 of \cite{Toth_et_al1998} by part (i) of the following lemma, which is a bit weaker that the result of \cite{Toth_et_al1998}, but is sufficient. Moreover, the random nature of the barrier requires an additional argument for part (ii) of Lemma 8.1 of \cite{Toth_et_al1998}, which is part (ii) of the following lemma. 

\begin{lemma}\label{lem_cons_coal}
 Let $(x,h) \in \mathds{R}^2$. 
 \begin{enumerate}
 \item There exists a deterministic sequence $(n(k))_{k \geq 1}$ so that for any $k \geq 1$, $\tilde x_{n(k)}<x$, $\lim_{k \to +\infty}\tilde x_{n(k)}=x$, $\lim_{k \to +\infty}\tilde h_{n(k)}=h$, and for all $\varepsilon>0$, if $\lambda(x)<h$, when $k$ is large enough $\mathds{P}_{\lambda,\chi}(\exists\, y \geq x+\varepsilon, \Lambda_{(x,h)}(y) \neq \Lambda_{(\tilde x_{n(k)},\tilde h_{n(k)})}(y)) \leq \frac{7}{\sqrt{\pi v}}2^{-k/2}$.
 \item Almost surely, if $\lambda(x)<h$, for all $\varepsilon>0$, there exists $k_0 \geq 1$ so for all $k \geq k_0$, we have $\Lambda_{(x,h)}(y)=\Lambda_{(\tilde x_{n(k)},\tilde h_{n(k)})}(y)$ for all $y \geq x+\varepsilon$.
 \end{enumerate}
\end{lemma}

\begin{proof}
(i) As in \cite{Toth_et_al1998}, we ‘‘squeeze $\Lambda_{(x,h)}$ between two families of lines $\Lambda_{(\tilde x_{n(k)},\tilde h_{n(k)})}$ and $\Lambda_{(\tilde x_{m(k)},\tilde h_{m(k)})}$ which have high probability of coalescing before $x+\varepsilon$, so after $x+\varepsilon$, $\Lambda_{(x,h)}$ is $\Lambda_{(\tilde x_{n(k)},\tilde h_{n(k)})}$''. We choose two sequences $((\tilde x_{n(k)},\tilde h_{n(k)}))_{k \geq 1}$ and $((\tilde x_{m(k)},\tilde h_{m(k)}))_{k \geq 1}$ in $\mathds{R}^2$ so that when $k$ is large enough, $\tilde x_{n(k)}=\tilde x_{m(k)} \in (x-5^{-k},x)$, $h-2 \cdot 2^{-k} \leq \tilde h_{n(k)} \leq h - 2^{-k}$ and $h + 2^{-k} \leq \tilde h_{m(k)} \leq h+2 \cdot 2^{-k}$. In this proof we assume $\lambda(x)<h$. 

We now build ``bad events'' such that $\{\exists\, y \geq x+\varepsilon, \Lambda_{(x,h)}(y) \neq \Lambda_{(\tilde x_{n(k)},\tilde h_{n(k)})}(y)\}$ ensures the occurrence of a bad event. Since $\lambda$ is continuous and $\lambda(x)<h$, there exists $0<\delta<\varepsilon$ so that for all $y \in [x-\delta,x+\delta]$, we have $\lambda(y)< h-\delta$. We consider only $k$ large enough to have $4 \cdot 2^{-k} \leq \delta$. We define the events $\mathcal{A}_{1,k}=\{\exists \, y \in [\tilde x_{n(k)}, x+5^{-k}], \Lambda_{(\tilde x_{n(k)},\tilde h_{n(k)})}(y) \not\in(h-3 \cdot 2^{-k},h)\}$ and $\mathcal{A}_{2,k}=\{\exists \, y \in [\tilde x_{m(k)}, x+5^{-k}], \Lambda_{(\tilde x_{m(k)},\tilde h_{(k)})}(y) \not\in(h,h+3 \cdot 2^{-k})\}$. By Lemma \ref{lem_RAB_max}, when $k$ is large enough, 
\begin{equation}\label{eq_cons_coal1}
\mathds{P}_{\lambda,\chi}(\mathcal{A}_{1,k}) \leq 2 \frac{\sqrt{4 v 5^{-k}}}{2^{-k}\sqrt{\pi}}\exp\left(-\frac{2^{-2k}}{4 v 5^{-k}}\right) \leq \exp\left(-\frac{5^k}{v4^{k+1}}\right)\text{ and }\mathds{P}(\mathcal{A}_{2,k}) \leq \exp\left(-\frac{5^k}{v4^{k+1}}\right).
\end{equation}
In \cite{Toth_et_al1998}, the authors also study the event $\{\Lambda_{(\tilde x_{n(k)},\tilde h_{n(k)})}(x+\varepsilon) \neq \Lambda_{(\tilde x_{m(k)},\tilde h_{m(k)})}(x+\varepsilon)\}$, and prove it has small probability by estimating the probability that two RABs coalesce. Since we cannot do this here, we replace their event with $\mathcal{A}_{3,k} = \{\Lambda_{(\tilde x_{n(k)},\tilde h_{n(k)})}(x+2^{-k}) \neq \Lambda_{(\tilde x_{m(k)},\tilde h_{m(k)})}(x+2^{-k})\}$. The idea will be to prove that in $[x,x+2^{-k}]$, the processes $\Lambda_{(\tilde x_{n(k)},\tilde h_{n(k)})}$ and $\Lambda_{(\tilde x_{m(k)},\tilde h_{m(k)})}$ will not go too far from their positions at $x$, so will not meet $\lambda$ hence will behave like two independent Brownian motions, thus since their starting points are close they will meet with high probability. Since $2^{-k} \leq \delta \leq \varepsilon$, we have $\{\exists\, y \geq x+\varepsilon, \Lambda_{(x,h)}(y) \neq \Lambda_{(\tilde x_{n(k)},\tilde h_{n(k)})}(y)\} \subset \mathcal{A}_{1,k} \cup \mathcal{A}_{2,k} \cup \mathcal{A}_{3,k}$.  Consequently (recalling \eqref{eq_cons_coal1}), to prove (i), it is enough to prove 
\begin{equation}\label{eq_cons_coal2}
\mathds{P}_{\lambda,\chi}(\mathcal{A}_{1,k}^c \cap \mathcal{A}_{2,k}^c \cap \mathcal{A}_{3,k}) \leq \frac{6+1/2}{\sqrt{\pi v}}2^{-k/2}
\end{equation}
when $k$ is large enough.

 We now study $\mathds{P}_{\lambda,\chi}(\mathcal{A}_{1,k}^c \cap \mathcal{A}_{2,k}^c \cap \mathcal{A}_{3,k})$. We can assume the processes $\Lambda_{(\tilde x_{n(k)},\tilde h_{n(k)})}$ and $\Lambda_{(\tilde x_{m(k)},\tilde h_{m(k)})}$ form a $(\lambda,\chi)$-FICRAB built from the $(\lambda,\chi)$-RABs $(R_{n(k)}(y))_{y \geq \tilde x_{n(k)}}$, $(R_{m(k)}(y))_{y \geq \tilde x_{m(k)}}$, themselves driven by the independent Brownian motions $(W_{n(k)}(y))_{y \geq \tilde x_{n(k)}}$ and $(W_{m(k)}(y))_{y \geq \tilde x_{m(k)}}$. Let $\mathcal{A}_{4,k}=\{(\tilde x_{n(k)},\tilde h_{n(k)}),(\tilde x_{m(k)},\tilde h_{m(k)})\in\mathds{R}^2_\lambda,\forall \,y \in [x,x+2^{-k}], W_{n(k)}(y)-W_{n(k)}(x)+\Lambda_{(\tilde x_{n(k)},\tilde h_{n(k)})}(x) \neq W_{m(k)}(y)-W_{m(k)}(x)+\Lambda_{(\tilde x_{m(k)},\tilde h_{m(k)})}(x)\}$ and $\mathcal{A}_{5,k}=\{\exists \, y \in [x,x+2^{-k}], |W_{n(k)}(y)-W_{n(k)}(x)| > \delta/4$ or $|W_{m(k)}(y)-W_{m(k)}(x)| > \delta/4\}$. We assume $\mathcal{A}_{1,k}^c \cap \mathcal{A}_{2,k}^c$ occurs and $\mathcal{A}_{4,k} \cup \mathcal{A}_{5,k}$ does not occur. Since $\mathcal{A}_{1,k}^c \cap \mathcal{A}_{2,k}^c$ occurs, $\Lambda_{(\tilde x_{n(k)},\tilde h_{n(k)})}$ and $\Lambda_{(\tilde x_{m(k)},\tilde h_{m(k)})}$ have not coalesced yet at $x$. Since $\mathcal{A}_{1,k}^c \cap \mathcal{A}_{2,k}^c \cap \mathcal{A}_{5,k}^c$ occurs, for all $y \in [x,x+2^{-k}]$ we have $R_{n(k)}(y) > \lambda(y)$ thus $R_{n(k)}(y) = W_{n(k)}(y)-W_{n(k)}(x)+R_{n(k)}(x)= W_{n(k)}(y)-W_{n(k)}(x)+\Lambda_{(\tilde x_{n(k)},\tilde h_{n(k)})}(x)$, and similarly $R_{m(k)}(y) = W_{m(k)}(y)-W_{m(k)}(x)+\Lambda_{(\tilde x_{m(k)},\tilde h_{m(k)})}(x)$. Since $\mathcal{A}_{4,k} $ does not occur, $\Lambda_{(\tilde x_{n(k)},\tilde h_{n(k)})}$ and $\Lambda_{(\tilde x_{m(k)},\tilde h_{m(k)})}$ have coalesced before $x+2^{-k}$, thus before $x+\varepsilon$, hence $\mathcal{A}_{3,k}$ does not occur. We deduce $\mathds{P}_{\lambda,\chi}(\mathcal{A}_{1,k}^c \cap \mathcal{A}_{2,k}^c \cap \mathcal{A}_{3,k}) \leq \mathds{P}_{\lambda,\chi}((\mathcal{A}_{4,k} \cup \mathcal{A}_{5,k}) \cap \mathcal{A}_{1,k}^c \cap \mathcal{A}_{2,k}^c)$, so it is enough to prove $\mathds{P}_{\lambda,\chi}((\mathcal{A}_{4,k} \cup \mathcal{A}_{5,k}) \cap \mathcal{A}_{1,k}^c \cap \mathcal{A}_{2,k}^c)\leq \frac{6+1/2}{\sqrt{\pi v}}2^{-k/2}$ when $k$ is large enough.

We now study $\mathds{P}_{\lambda,\chi}(\mathcal{A}_{4,k} \cap \mathcal{A}_{1,k}^c \cap \mathcal{A}_{2,k}^c)$ and $\mathds{P}_{\lambda,\chi}(\mathcal{A}_{5,k})$. By the standard estimate on Brownian motion given in Lemma \ref{lem_BM_max}, we have $\mathds{P}_{\lambda,\chi}(\mathcal{A}_{5,k}) \leq \frac{16\sqrt{2 v 2^{-k}}}{\delta\sqrt{\pi}}\exp(-\frac{\delta^2}{32v2^{-k}}) \leq \exp(-\frac{\delta^22^k}{32v}) \leq \frac{1}{2\sqrt{\pi v}}2^{-k/2}$ when $k$ is large enough. To deal with $\mathcal{A}_{4,k}$, we notice that if $(\tilde x_{n(k)},\tilde h_{n(k)}),(\tilde x_{m(k)},\tilde h_{m(k)})\in\mathds{R}^2_\lambda$, then $\mathds{P}_{\lambda,\chi}(\mathcal{A}_{4,k}|\Lambda_{(\tilde x_{n(k)},\tilde h_{n(k)})}(x),\Lambda_{(\tilde x_{m(k)},\tilde h_{m(k)})}(x))$ is the probability that the Brownian motion $W_{n(k)}(y)-W_{n(k)}(x)-(W_{m(k)}(y)-W_{m(k)}(x))$ does not reach $\Lambda_{(\tilde x_{m(k)},\tilde h_{m(k)})}(x)-\Lambda_{(\tilde x_{n(k)},\tilde h_{n(k)})}(x)$ in $[x,x+2^{-k}]$, and if $\mathcal{A}_{1,k}^c \cap \mathcal{A}_{2,k}^c$ occurs we get $|\Lambda_{(\tilde x_{m(k)},\tilde h_{m(k)})}(x)-\Lambda_{(\tilde x_{n(k)},\tilde h_{n(k)})}(x)| \leq 6 \cdot 2^{-k}$. Therefore Lemma \ref{lem_BM_max2} implies $\mathds{P}_{\lambda,\chi}(\mathcal{A}_{4,k} \cap \mathcal{A}_{1,k}^c \cap \mathcal{A}_{2,k}^c) \leq \mathds{E}_{\lambda,\chi}(\mathds{P}_{\lambda,\chi}(\mathcal{A}_{4,k}|\Lambda_{(\tilde x_{n(k)},\tilde h_{n(k)})}(x),\Lambda_{(\tilde x_{m(k)},\tilde h_{m(k)})}(x))\mathds{1}_{\{|\Lambda_{(\tilde x_{m(k)},\tilde h_{m(k)})}(x)-\Lambda_{(\tilde x_{n(k)},\tilde h_{n(k)})}(x)| \leq 6 \cdot 2^{-k}\}}) \leq \mathds{P}(\max_{y \in [x,x+2^{-k}]}(W_{n(k)}(y)-W_{n(k)}(x)-(W_{m(k)}(y)-W_{m(k)}(x))) \leq 6 \cdot 2^{-k}) \leq \frac{6 \cdot 2^{-k}}{\sqrt{\pi v 2^{-k}}}=\frac{6}{\sqrt{\pi v}}2^{-k/2}$, which tends to 0 when $k$ tends to $+\infty$. Consequently, $\mathds{P}((\mathcal{A}_{4,k} \cup \mathcal{A}_{5,k}) \cap \mathcal{A}_{1,k}^c \cap \mathcal{A}_{2,k}^c)\leq \frac{6+1/2}{\sqrt{\pi v}}2^{-k/2}$ when $k$ is large enough, which ends the proof of (i).

(ii) If $\lambda(x)<h$, then by \eqref{eq_cons_coal1} and \eqref{eq_cons_coal2}, we have $\mathds{P}_{\lambda,\chi}(\bigcap_{\ell \in\mathds{N}}\bigcup_{k \geq \ell}(\mathcal{A}_{1,k} \cup \mathcal{A}_{2,k} \cup \mathcal{A}_{3,k})))=0$. Therefore $\mathds{P}(\{\lambda(x)<h\}\cap\bigcap_{\ell \in\mathds{N}}\bigcup_{k \geq \ell}(\mathcal{A}_{1,k} \cup \mathcal{A}_{2,k} \cup \mathcal{A}_{3,k})))=0$, which suffices. 
\end{proof}

\subsection{Equation (8.22) of \cite{Toth_et_al1998}}

To understand Equation (8.22) of \cite{Toth_et_al1998}, we need to give some notation. Let $p \in \mathds{N}^*$, $(x_1,h_1),...,(x_p,h_p) \in \mathds{R}^2$, we denote $((C_1(y))_{y \geq x_1},...,(C_p(y))_{y \geq x_p})$ a $(\lambda,\chi)$-FICRAB starting from $(x_1,h_1),...,(x_p,h_p)$. Let $r \in \mathds{N}^*$, set a map $j : \{1,...,r\} \mapsto \{1,...,p\}$. For $i\in\{1,...,r\}$, set $y_i > x_{j(i)}$ and $\lambda(y_i) < a_i < b_i$. For each $j\in\{1,...,p\}$, $k\in\mathds{N}$, let $(\tilde x_{n_j(k)},\tilde h_{n_j(k)}) \in \mathds{D}^2$ so that $\tilde x_{n_j(k)}\in(x_j-5^{-k},x_j)$ and $\tilde h_{n_j(k)} \in (h_j-2\cdot2^{-k},h_j-2^k)$. The equivalent of Equation (8.22) of \cite{Toth_et_al1998} is proving the following. 

\begin{proposition}\label{prop_cons}
If $(x_1,h_1),...,(x_p,h_p) \in \mathds{R}_\lambda^2$, $\mathds{P}_{\lambda,\chi}(\forall\, i \in \{1,...,r\}, \Lambda_{(\tilde x_{n_{j(i)}(k)},\tilde h_{n_{j(i)}(k)})}(y_i)\in(a_i,b_i))$ converges to $\mathds{P}_{\lambda,\chi}(\forall\, i \in \{1,...,r\}, C_{j(i)}(y_i)\in(a_i,b_i))$ when $k$ tends to $+\infty$.
\end{proposition}

In \cite{Toth_et_al1998}, the only justification given to this is that $(\tilde x_{n_j(k)},\tilde h_{n_j(k)})$ converges to $(x_j,h_j)$ for $j \in \{1,...,p\}$. Here, since the barriers are more complex, we need more arguments.

\begin{proof}[Proof of Proposition \ref{prop_cons}.]
We will construct processes $(\Lambda_{j,k}')_{1 \leq j \leq p}$ and $(C_j')_{1 \leq j \leq p}$ which have the same distribution as the $(\Lambda_{(\tilde x_{n_{j}(k)},\tilde h_{n_{j}(k)})})_{1 \leq j \leq p}$ and $(C_j)_{1 \leq j \leq p}$, and so that the finite-dimensional marginals of $(\Lambda_{j,k}')_{1 \leq j \leq p}$ converge in probability to those of $(C_j')_{1 \leq j \leq p}$. Let $(W_{j}(x))_{x \geq x_j-1}$, $j\in\{1,...,p\}$ be independent Brownian motions. For any $j\in\{1,...,p\}$, let $R_j'$ the $(\lambda,\chi)$-RAB starting from $(x_j,h_j)$ and driven by $W_j$. $(C_1',...,C_p')$ is then defined as the $(\lambda,\chi)$-FICRAB constructed from $(R_1',...,R_p')$. For all $k\in\mathds{N}$, $j\in\{1,...,p\}$, let $\Lambda_{j,k}^R$ the $(\lambda,\chi)$-RAB starting from $(\tilde x_{n_{j(i)}(k)},\tilde h_{n_{j(i)}(k)})$ and driven by $W_j$. $(\Lambda_{1,k}',...,\Lambda_{p,k}')$ is then defined as the $(\lambda,\chi)$-FICRAB constructed from $(\Lambda_{1,k}^R,...,\Lambda_{p,k}^R)$. In the whole proof of Proposition \ref{prop_cons}, we work conditionally to $(\lambda,\chi)$. We begin by proving the following lemma, which states the $\Lambda_{j,k}^R$ are close to the $R_j'$ with high probability. 

\begin{lemma}\label{lem_cons1}
 For any $j\in\{1,...,p\}$, $y_0 > \max_{1 \leq i \leq r}y_i$, $\varepsilon > 0$, $\mathds{P}_{\lambda,\chi}(\exists\, x_j \leq y \leq y_0, R_j'(y)-\Lambda_{j,k}^R(y) \not\in [0,\varepsilon])$ tends to 0 when $k$ tends to $+\infty$. 
\end{lemma}

\begin{proof}
It is enough to prove the lemma for $\varepsilon$ small, so let $y_0 > \max_{1 \leq i \leq r}y_i$, $j\in\{1,...,p\}$ and $\varepsilon > 0$ small. Let $\varepsilon' > 0$, it is enough to prove that $\mathds{P}_{\lambda,\chi}(\exists\, x_j \leq y \leq y_0, R_j'(y)-\Lambda_{j,k}^R(y) \not\in [0,\varepsilon]) \leq \varepsilon'$ when $k$ is large enough. Let $k\in\mathds{N}$ large enough to have $3 \cdot 2^{-k} \leq \varepsilon$ and $\frac{2\sqrt{2 v 5^{-k}}}{2^{-k}\sqrt{\pi}}\exp(-\frac{2^{-2k}}{2 v 5^{-k}}) \leq \varepsilon'/3$. We are going to construct events with small probability so that if none of them occurs, for all $x_j \leq y \leq y_0$, $R_j'(y)-\Lambda_{j,k}^R(y) \in [0,\varepsilon]$. We denote $\mathcal{A}_{1,k}=\{\exists\, x \in [\tilde x_{n_{j}(k)},x_j],W_j(x)-W_j(\tilde x_{n_{j}(k)}) \not \in [-2^{-k},2^{-k}]\}$. Then, by the classical estimate on Brownian motion recalled in Lemma \ref{lem_BM_max}, we have $\mathds{P}_{\lambda,\chi}(\mathcal{A}_{1,k}) \leq \frac{2\sqrt{2v 5^{-k}}}{2^{-k}\sqrt{\pi}}\exp(-\frac{2^{-2k}}{2 v 5^{-k}}) \leq \varepsilon'/3$. Furthermore, since $h_j > \lambda(x_j)$ and $\lambda$ is continuous, when $k$ is large enough, for all $x \in [x_j-5^{-k},x_j]$, $\lambda(x) < h_j-3 \cdot 2^{-k}$. In the following we only consider $k$ large enough for this to happen. Then if $\mathcal{A}_{1,k}$ does not occur, 
\begin{equation}\label{eq_cons1}
\forall\,y \in [\tilde x_{n_{j}(k)},x_j],\Lambda_{j,k}^R(y)\neq\lambda(y)\text{  and }\Lambda_{j,k}^R(x_j) \in [h_j-3 \cdot 2^{-k},h_j],\text{ hence }\lambda(x_j) < R_j'(x_j)-3 \cdot 2^{-k} \leq \Lambda_{j,k}^R(x_j) \leq R_j'(x_j).
\end{equation}
Moreover, $\lambda$ and $W_j$ are uniformly continuous on $[x_j,y_0+1]$, so we can choose $0 <\delta_2 <1$ so that for $y,y' \in [x_j,y_0+1]$, if $|y-y'|\leq\delta_2$ then $|\lambda(y)-\lambda(y')| \leq \varepsilon/2$, and denoting $\mathcal{A}_2=\{\exists \,y,y' \in [x_j,y_0+1]$ so that $|y-y'|\leq\delta_2$ and $|W_j(y)-W_j(y')| > \varepsilon/2\}$, then $\mathds{P}_{\lambda,\chi}(\mathcal{A}_2) \leq \varepsilon'/3$.

We now need to define events ensuring ``$\Lambda_{j,k}^R$ will be absorbed by $\lambda$ almost at the same place as $R_j'$''. In order to do that, we will use Lemma 8.2 of \cite{Mareche_et_al2023}. It states that if $a < b$, if $f : [a,b] \mapsto \mathds{R}$ is a continuous function, if $(W(y))_{y \in [a,b]}$ is a Brownian motion so that $W(a)>f(a)$ almost surely, if for any $\delta \in \mathds{R}$ we denote $\sigma(\delta)=\inf\{t \in [a,b] \,|\, W(y) \leq f(y)+\delta\}$ (the inf being $+\infty$ when the set is empty), then $\sigma(\delta)$ converges in probability to 0 when $\delta$ tends to 0 (Lemma 8.2 of \cite{Mareche_et_al2023} is stated for $a=0$, $b=1$, but easily extends to general $a$ and $b$). For all $\delta \in \mathds{R}$, we denote $Y(\delta) = \inf\{ y \in [x_j,y_0]\,|\, W_j(y)-W_j(x_j)+h_j\leq\lambda(y)+\delta\}$ (the inf being $+\infty$ when the set is empty), then since $h_j > \lambda(x_j)$, by Lemma 8.2 of \cite{Mareche_et_al2023}, $Y(\delta)$ converges in $\mathds{P}_{\lambda,\chi}$-probability to $Y(0)$ when $\delta$ tends to 0. Therefore we can choose $\delta_3 > 0$ so that $\mathds{P}_{\lambda,\chi}(Y(0)-Y(\delta_3) > \delta_2) \leq \varepsilon'/6$. Here we will need to distinguish between the cases $x_j \geq \chi$ and $x_j < \chi$. If $x_j \geq \chi$, we denote $\mathcal{A}_3=\{Y(0)-Y(\delta_3) > \delta_2\}$. If $x_j < \chi$ and $y_0 > \chi$ (defining this event is not necessary if $y_0 \leq \chi$), for all $\delta \in \mathds{R}$, we denote $Y'(\delta) = \inf\{ y \in [\chi,y_0]\,|\, W_j(y)-W_j(\chi)+R_j'(\chi)\leq\lambda(y)+\delta\}$ (the inf being $+\infty$ when the set is empty). Then by the Lemma 8.2 of \cite{Mareche_et_al2023}, $\mathds{P}_{\lambda,\chi}(Y'(0)-Y'(\delta) > \delta_2|R_j'(\chi))\mathds{1}_{\{R_j'(\chi) > \lambda(\chi)\}}$ converges almost surely to 0 when $\delta$ tends to 0, hence $\mathds{P}_{\lambda,\chi}(R_j'(\chi) > \lambda(\chi),Y'(0)-Y'(\delta) > \delta_2)$ tends to 0 when $\delta$ tends to 0. We thus choose $\delta_4 > 0$ so that $\mathds{P}_{\lambda,\chi}(R_j'(\chi) > \lambda(\chi),Y'(0)-Y'(\delta_4) > \delta_2) \leq \varepsilon'/6$ and set $\mathcal{A}_3=\{Y(0)-Y(\delta_3) > \delta_2\} \cup \{R_j'(\chi) > \lambda(\chi),Y'(0)-Y'(\delta_4) > \delta_2\}$. We then have $\mathds{P}_{\lambda,\chi}(\mathcal{A}_{1,k} \cup \mathcal{A}_{2} \cup \mathcal{A}_3) \leq \varepsilon'$. Consequently, if we show that for $k$ large enough, when $\mathcal{A}_{1,k} \cup \mathcal{A}_{2} \cup \mathcal{A}_3$ does not occur we have that for all $x_j \leq y \leq y_0$, $R_j'(y)-\Lambda_{j,k}^R(y) \in [0,\varepsilon]$, this yields $\mathds{P}_{\lambda,\chi}(\exists\, x_j \leq y \leq y_0, R_j'(y)-\Lambda_{j,k}^R(y) \not\in [0,\varepsilon]) \leq \varepsilon'$ when $k$ is large enough, which proves the lemma. 

We now assume that $x_j \geq \chi$ and $\mathcal{A}_{1,k} \cup \mathcal{A}_{2} \cup \mathcal{A}_3$ does not occur, and prove that for all $x_j \leq y \leq y_0$, $R_j'(y)-\Lambda_{j,k}^R(y) \in [0,\varepsilon]$. Since $\mathcal{A}_{1,k}$ does not occur, we have \eqref{eq_cons1}, thus if we denote $Y_k=\inf\{y \in [x_j,y_0] \,|\, \Lambda_{j,k}^R(y)=\lambda(y)\}$, we have $Y_k > x_j$ and $R_j'(x_j)-3 \cdot 2^{-k} \leq \Lambda_{j,k}^R(x_j) \leq R_j'(x_j)$. Then $Y(0)\geq Y_k$, and if $y \in [x_j,Y_k]$, we have $R_j'(y)-\Lambda_{j,k}^R(y)=R_j'(x_j)-\Lambda_{j,k}^R(x_j) \in [0, 3 \cdot 2^{-k}] \subset [0,\varepsilon]$. If $Y_k=+\infty$, we thus have $R_j'(y)-\Lambda_{j,k}^R(y) \in [0,\varepsilon]$ for all $y \in [x_j,y_0]$. If $Y_k < +\infty$, we have $R_j'(y)-\Lambda_{j,k}^R(y) \in [0,3 \cdot 2^{-k}]$ for $y \in [x_j,Y_k]$. In particular, $R_j'(Y_k) \leq \lambda(Y_k) + 3 \cdot 2^{-k}$, so when $k$ is large enough to have $3 \cdot 2^{-k} < \delta_3$, we get $Y(\delta_3) \leq Y_k$. Moreover, since $\mathcal{A}_3$ does not occur, we obtain $Y(0) \leq Y_k+\delta_2$, and for $y \geq Y(0)$ we have $R_j'(y)=\Lambda_{j,k}^R(y)=\lambda(y)$. It remains to consider the case $y\in[Y_k,Y_k+\delta_2]$, $y \leq Y(0)$. We then have $\Lambda_{j,k}^R(y) =\lambda(y) \leq R_j'(y)$. In addition, $|\Lambda_{j,k}^R(y)-R_j'(y)|=|\lambda(y)-R_j'(y)|\leq|\lambda(y)-\lambda(Y(0))|+|\lambda(Y(0))-R_j'(y)|$. Furthermore, by the definition of $\delta_2$ we have $|\lambda(y)-\lambda(Y(0))| \leq \varepsilon/2$, and $|\lambda(Y(0))-R_j'(y)|=|W_j(Y(0))-W_j(y)| \leq \varepsilon/2$ since $\mathcal{A}_2$ does not occur. Therefore $|\Lambda_{j,k}^R(y)-R_j'(y)| \leq \varepsilon$, hence $R_j'(y)-\Lambda_{j,k}^R(y) \in [0,\varepsilon]$, which ends the proof in the case $x_j \geq \chi$.

We now deal with the case $x_j < \chi$. We assume $\mathcal{A}_{1,k} \cup \mathcal{A}_{2} \cup \mathcal{A}_3$ does not occur and prove that for all $x_j \leq y \leq y_0$, $R_j'(y)-\Lambda_{j,k}^R(y) \in [0,\varepsilon]$. We begin by studying $R_j'(y)-\Lambda_{j,k}^R(y)$ for $y \in [x,\chi]$. We denote $Y=\inf\{y \in [x,\chi]\,|\,R_j'(y)=\lambda(y)\}$, $Y_k=\inf\{y \in [x,\chi]\,|\,\Lambda_{j,k}^R(y)=\lambda(y)\}$ (the inf being $+\infty$ if the set is empty). Since $\mathcal{A}_{1,k}$ does not occur, \eqref{eq_cons1} yields $Y_k > x_j$ and $R_j'(x_j)-3 \cdot 2^{-k} \leq \Lambda_{j,k}^R(x_j) \leq R_j'(x_j)$. Then if $y \in [x_j,Y_k]$, we have $R_j'(y)-\Lambda_{j,k}^R(y)=R_j'(x_j)-\Lambda_{j,k}^R(x_j) \in [0, 3 \cdot 2^{-k}] \subset [0,\varepsilon]$. This also implies $Y_k \leq Y$. For $Y_k \leq y \leq Y$ (if $Y_k$ is finite), we have $\Lambda_{j,k}^R(y) \leq R_j'(y)$ (indeed, if $\Lambda_{j,k}^R(y) > R_j'(y)$ then $\Lambda_{j,k}^R(y) > \lambda(y)$, and if $y' = \sup\{y'' < y \,|\, \Lambda_{j,k}^R(y'') =\lambda(y'')\}$ then $\Lambda_{j,k}^R(y') \leq R_j'(y')$, but $\Lambda_{j,k}^R(y)-\Lambda_{j,k}^R(y')=W_j(y)-W_j(y')=R_j'(y)-R_j'(y')$ which is a contradiction), and $\Lambda_{j,k}^R(y)-\Lambda_{j,k}^R(Y_k) \geq W_j(y)-W_j(Y_k)=R_j'(y)-R_j'(Y_k)$, hence $0 \leq R_j'(y)-\Lambda_{j,k}^R(y) \leq R_j'(Y_k)-\Lambda_{j,k}^R(Y_k) \leq 3 \cdot 2^{-k}$, hence $R_j'(y)-\Lambda_{j,k}^R(y) \in [0,\varepsilon]$. Finally (if $Y$ is finite), $\lambda(Y) \leq \Lambda_{j,k}^R(Y) \leq R_j'(Y) = \lambda(Y)$, so $\Lambda_{j,k}^R(Y) = R_j'(Y)$ thus for $y \geq Y$ we have $\Lambda_{j,k}^R(y) = R_j'(y)$. Therefore $R_j'(y)-\Lambda_{j,k}^R(y) \in [0,\varepsilon]$ for all $y\in[x,\chi]$. 

We now consider $x \in [\chi,y_0]$ (of course, this is not necessary if $y_0 \leq \chi$). If $\lambda(\chi)=R_j'(\chi)=\Lambda_{j,k}^R(\chi)$, then $R_j'(y)=\Lambda_{j,k}^R(y)$ for all $y \geq \chi$. If $\lambda(\chi)=\Lambda_{j,k}^R(\chi)<R_j'(\chi)$, we can obtain as in the case $x_j>\chi$ that $Y(\delta_3) \leq Y_k \leq \chi$ and $Y(0)\in[Y_k,Y_k+\delta_2]$. If $y\in[Y_k,Y_k+\delta_2]$, $y \leq Y(0)$, we have $R_j'(y)-\Lambda_{j,k}^R(y)\geq 0$ (for the part of the interval in $[x_j,\chi]$, it works as in the previous paragraph; for the other part both processes start evolving as $W_j$, until $\Lambda_{j,k}^R$ meets $\lambda$, at which point it is even lower). Moreover, $|R_j'(y)-\Lambda_{j,k}^R(y)| \leq |R_j'(y)-\lambda(y)| \leq \varepsilon$, which can be proven as in the case $x_j \geq \chi$. Moreover, we notice $\lambda(Y(0)) \leq \Lambda_{j,k}^R(Y(0)) \leq R_j'(Y(0))=\lambda(Y(0))$, hence $\Lambda_{j,k}^R(y) = R_j'(y)$ for $y \geq Y(0)$. This implies $R_j'(y)-\Lambda_{j,k}^R(y) \in [0,\varepsilon]$ for all $y\in[\chi,y_0]$. Finally, if $\lambda(\chi) < \Lambda_{j,k}^R(\chi) \leq R_j'(\chi)$, we notice that $0 \leq R_j'(\chi)-\Lambda_{j,k}^R(\chi) \leq 3 \cdot 2^{-k}$, and we can complete the proof with the same arguments as in the case $x_j \geq \chi$, replacing the $Y(\delta)$ by the $Y'(\delta)$. 
\end{proof}

We now prove the following lemma, which implies the finite-dimensional marginals of the $\Lambda_{j,k}'$ converge in probability to those of the $C_j'$.

\begin{lemma} \label{lem_cons2}
 For any $j\in\{1,...,p\}$, $y_0 > \max_{1 \leq i \leq r}y_i$, $\varepsilon > 0$, $\mathds{P}_{\lambda,\chi}(\exists\, \ell \in \{1,...,j\}, \exists\, x_{\ell} \leq y \leq y_0, |C_\ell'(y)-\Lambda_{\ell,k}'(y)| > \varepsilon)$ tends to 0 when $k$ tends to $+\infty$. 
\end{lemma}

\begin{proof}
We prove the lemma by induction on $j$. The case $j=1$ is given by Lemma \ref{lem_cons1}. We now set $j \in \{2,...,p\}$ and assume Lemma \ref{lem_cons2} holds for $j-1$. Let $y_0 > \max_{1 \leq i \leq r}y_i$, $\varepsilon > 0$ small, $\varepsilon'> 0$, and let $k\in\mathds{N}$ large. We will define events with small probability so that if they do not occur then for all $\ell \in \{1,...,j\}$, $x_{\ell} \leq y \leq y_0$, we have $|C_\ell'(y)-\Lambda_{\ell,k}'(y)| \leq \varepsilon$. In order to do that, we denote $Y=\inf\{x \geq x_j \,|\, R_{j}'(y)\in\{C_{1}'(y),...,C_{j-1}'(y)\}\}$ (which is $+\infty$ if the set is empty) the place at which $R_j'$ coalesces with one of the $C_\ell'$, $\ell \leq j-1$. If $Y$ is finite, we denote $\mathcal{L}=\{\ell \in \{1,...,j-1\} \,|\, R_{j}'(Y)=C_{\ell}'(Y)\}$. We choose $0 <\delta_1 \leq \varepsilon$ so that, denoting $\mathcal{A}_1=\{Y<+\infty,\exists\, \ell \in \{1,...,j-1\} \setminus \mathcal{L}, |C_\ell'(Y)-R_j'(Y)| < \delta_1\}$, we have $\mathds{P}_{\lambda,\chi}(\mathcal{A}_1)\leq \varepsilon'/8$, so that, denoting $\mathcal{A}_2=\{\lambda(Y) < C_j'(Y) < \lambda(Y)+\delta_1\}$, we have $\mathds{P}_{\lambda,\chi}(\mathcal{A}_2)\leq\varepsilon'/8$, and so that, denoting $\mathcal{A}_3=\{\exists \, \ell \in \{1,...,j-1\}, Y \in (x_\ell-\delta_1,x_\ell)\}$, we have $\mathds{P}_{\lambda,\chi}(\mathcal{A}_3)\leq\varepsilon'/8$. In addition, $\lambda,C_1',...,C_j',R_{j}'$ are uniformly continuous on the respective intervals $[x_j,y_0+1]$, $[x_1,y_0+1]$, ..., $[x_{j},y_0+1]$, $[x_{j},y_0+1]$ so we can choose $0 <\delta_2 <\min(1/2,\delta_1)$ so that for $y,y' \in [x_j,y_0+1]$, if $|y-y'|\leq\delta_2$ then $|\lambda(y)-\lambda(y')| \leq \delta_1/6$, and denoting $\mathcal{A}_4=\{\exists\, \ell\in\{1,...,j\}, \exists \,y,y' \in [x_\ell,y_0+1]$ so that $|y-y'|\leq \delta_2$ and $|C_\ell'(y)-C_\ell'(y')| > \delta_1/6\}\cup \{\exists \,y,y' \in [x_j,y_0+1]$ so that $|y-y'|\leq \delta_2$ and $|R_j'(y)-R_j'(y')| > \delta_1/6\}$, then $\mathds{P}_{\lambda,\chi}(\mathcal{A}_4) \leq \varepsilon'/8$. We now define processes $(W_{\ell,k}(y))_{y \geq x_\ell}$, $\ell \in \{1,...,j-1\}$, where $W_{\ell,k}$ will be ``the Brownian motion driving $\Lambda_{\ell,k}'$''. We define them by induction as follows. $W_{1,k}=W_1$. For all $\ell \in \{2,...,j-1\}$, if $Y_{\ell,k}=\inf\{y \geq x_\ell \,|\, \Lambda_{\ell,k}^R(y)\in\{\Lambda_{1,k}'(y),...,\Lambda_{\ell-1,k}'(y)\}\}$ and $L_{\ell,k} = \min\{\ell'\in\{1,...,\ell-1\}\,|\,\Lambda_{\ell,k}^R(Y_{\ell,k})=\Lambda_{\ell',k}'(Y_{\ell,k})\}$, then $W_{\ell,k}(y)=W_\ell(y)$ for $x_\ell \leq y \leq Y_{\ell,k}$ and $W_{\ell,k}(y)=W_{L_{\ell,k},k}(y)-W_{L_{\ell,k},k}(Y_{\ell,k})+W_\ell(Y_{\ell,k})$ for $y \geq Y_{\ell,k}$. The $W_{\ell,k}$, $\ell\in \{1,...,j-1\}$ are then Brownian motions. We choose $0 < \delta_3 < \delta_1$ so that, calling $\mathcal{A}_5=\{\exists \, \ell \in\{1,...,j-1\}, \exists \, y \in [\max(x_j,x_\ell),\min(Y-\delta_2,y_0)], |R_j'(y)-C_\ell'(y)| \leq \delta_3\}$, then $\mathds{P}_{\lambda,\chi}(\mathcal{A}_5) \leq \varepsilon'/8$, and so that, denoting $\mathcal{A}_{6,k}=\{\exists\, \ell\in \{1,...,j-1\}, \forall \, y  \in [Y,Y+\delta_2],W_j(y)-W_j(Y)-(W_{\ell,k}(y)-W_{\ell,k}(Y)) \leq \delta_3\} \cup \{\exists\, \ell\in \{1,...,j-1\}, \forall \, y  \in [Y,Y+\delta_2],W_j(y)-W_j(Y)-(W_{\ell,k}(y)-W_{\ell,k}(Y)) \geq -\delta_3\}$, then $\mathds{P}_{\lambda,\chi}(\mathcal{A}_{6,k}) \leq \varepsilon'/8$. Moreover, by Lemma \ref{lem_cons1}, if we denote $\mathcal{A}_{7,k}=\{\exists\, \ell \in \{1,...,j-1\},\exists\, x_\ell \leq y \leq y_0+1, R_\ell'(y)-\Lambda_{\ell,k}^R(y) \not\in [0,\delta_3/6]\}$, then when $k$ is large enough, $\mathds{P}_{\lambda,\chi}(\mathcal{A}_{7,k}) \leq \varepsilon'/8$. Finally, by the induction hypothesis, if we denote $\mathcal{A}_{8,k}=\{\exists\, \ell \in \{1,...,j-1\}, \exists\, x_{\ell} \leq y \leq y_0+1, |C_\ell'(y)-\Lambda_{\ell,k}'(y)| > \delta_3/6\}$, then when $k$ is large enough, $\mathds{P}_{\lambda,\chi}(\mathcal{A}_{8,k}) \leq \varepsilon'/8$. We then have $\mathds{P}_{\lambda,\chi}((\bigcup_{i=1}^5 \mathcal{A}_i) \cup (\bigcup_{i=6}^8 \mathcal{A}_{i,k})) \leq \varepsilon'$ when $k$ is large enough. Therefore if we can show that when $k$ is large enough and $(\bigcup_{i=1}^5 \mathcal{A}_i) \cup (\bigcup_{i=6}^8 \mathcal{A}_{i,k})$ does not occur, then for all $\ell \in \{1,...,j\}$, $x_{\ell} \leq y \leq y_0$, $|C_\ell'(y)-\Lambda_{\ell,k}'(y)| \leq\varepsilon$, we have $\mathds{P}_{\lambda,\chi}(\exists\, \ell \in \{1,...,j\}, \exists\, x_{\ell} \leq y \leq y_0, |C_\ell'(y)-\Lambda_{\ell,k}'(y)| > \varepsilon) \leq \varepsilon'$ for $k$ large enough, which proves the lemma for $j$ and allows to complete the induction. 

We now assume $(\bigcup_{i=1}^5 \mathcal{A}_i) \cup (\bigcup_{i=6}^8 \mathcal{A}_{i,k})$ does not occur. For any $\ell\in\{1,...,j-1\}$, since $\mathcal{A}_5$ does not occur, for any $y \in [\max(x_j,x_\ell),\min(Y-\delta_2,y_0)]$ we have $|R_j'(y)-C_\ell'(y)| > \delta_3$. Furthermore, since $\mathcal{A}_{7,k}$ and $\mathcal{A}_{8,k}$ do not occur, we have $|R_j'(y)-\Lambda_{j,k}^R(y)| \leq \delta_3/6$ and $|\Lambda_{\ell,k}'(y)-C_\ell'(y)| \leq \delta_3/6$, hence $|\Lambda_{j,k}^R(y)-\Lambda_{\ell,k}'(y)| \geq 2\delta_3/3 > 0$. Therefore $\Lambda_{j,k}^R$ does not coalesce with one of the $\Lambda_{\ell,k}'$ before $\min(Y-\delta_2,y_0)$, which implies $\Lambda_{j,k}'(y)=\Lambda_{j,k}^R(y)$ for $y \in [x_j,\min(Y-\delta_2,y_0)]$. Since $\mathcal{A}_{7,k}$ does not occur, for all $x_{j} \leq y \leq \min(Y-\delta_2,y_0)$ we deduce $C_j'(y)-\Lambda_{j,k}'(y) \in [0,\varepsilon]$ (remember the definition of $Y$). If $Y \geq y_0+\delta_2$, since $\mathcal{A}_{8,k}$ does not occur, for all $\ell \in \{1,...,j\}$, $x_{\ell} \leq y \leq y_0$, $|C_\ell'(y)-\Lambda_{\ell,k}'(y)| \leq \varepsilon$, which is enough. We now assume $Y\leq y_0+\delta_2$. It remains to prove that for all $\min(x_j,Y-\delta_2) \leq y \leq y_0$, $|C_j'(y)-\Lambda_{j,k}'(y)| \leq \varepsilon$. 

Let $\ell \in \{1,...,j-1\}\setminus\mathcal{L}$, we are going to show $\Lambda_{j,k}^R$ does not coalesce with one of the $\Lambda_{\ell,k}'$ before $Y+\delta_2$. Since $\mathcal{A}_3$ does not occur, either $x_\ell \leq Y$ or $x_\ell \geq Y+\delta_1 \geq Y+\delta_2$. Furthermore, if $x_\ell \leq Y$, since $\mathcal{A}_1$ does not occur, $|C_\ell'(Y)-C_j'(Y)| \geq \delta_1$. Since $\mathcal{A}_4$ occurs, if $y \geq \max(x_\ell,x_j)$ and $y \in [Y-\delta_2,Y+\delta_2]$, $|C_\ell'(y)-C_\ell'(Y)| \leq \delta_1/6$ and  $|R_j'(y)-R_j'(Y)| \leq \delta_1/6$, hence $|C_\ell'(y)-R_j'(y)| \geq 2\delta_1/3$. Moreover, $\mathcal{A}_{7,k}$ and $\mathcal{A}_{8,k}$ do not occur, so $|R_j'(y)-\Lambda_{j,k}^R(y)|\leq \delta_3/6 \leq \delta_1/6$ and $|\Lambda_{\ell,k}'(y)-C_\ell'(y)|\leq \delta_3/6 \leq \delta_1/6$, thus $|\Lambda_{j,k}^R(y)-\Lambda_{\ell,k}'(y)| \geq \delta_1/3 > 0$, which means $\Lambda_{j,k}^R$ does not coalesce with one of the $\Lambda_{\ell,k}'$ before $Y+\delta_2$. 

We now study the interval $[Y+\delta_2,y_0+1]$. We need to dinstinguish between the cases $C_j'(Y)=\lambda(Y)$ and $C_j'(Y)>\lambda(Y)$. We begin by assuming $C_j'(Y)>\lambda(Y)$. Since $\mathcal{A}_2$ occurs, this implies $C_j'(Y)\geq\lambda(Y)+\delta_1$. Let $y \in [Y,Y+\delta_2]$. By the definition of $\delta_2$, $\lambda(y) \leq \lambda(Y)+\delta_1/6$. Moreover, since $\mathcal{A}_4$ and $\mathcal{A}_{8,k}$ do not occur, for $\ell \in \mathcal{L}$, we have $|\Lambda_{\ell,k}'(y)-C_j'(Y)| \leq |\Lambda_{\ell,k}'(y)-C_\ell'(y)|+|C_\ell'(y)-C_\ell'(Y)| \leq \delta_3/6+\delta_1/6 \leq \delta_1/3$, hence $\Lambda_{\ell,k}'(y) > \lambda(y)$. Similarly, since $\mathcal{A}_4$ and $\mathcal{A}_{7,k}$ do not occur, $\Lambda_{j,k}^R(y) > \lambda(y)$. This implies $\Lambda_{\ell,k}'(y)-\Lambda_{\ell,k}'(Y)=W_{\ell,k}(y)-W_{\ell,k}(Y)$ and $\Lambda_{j,k}^R(y)-\Lambda_{j,k}^R(Y)=W_j(y)-W_j(Y)$. In addition, the fact that $\mathcal{A}_{6,k}$ does not occur yields the existence of $y,y' \in [Y,Y+\delta_2]$ so that $W_j(y)-W_j(Y)-(W_{\ell,k}(y)-W_{\ell,k}(Y)) > \delta_3$ and $W_j(y')-W_j(Y)-(W_{\ell,k}(y')-W_{\ell,k}(Y)) < -\delta_3$. Thus $\Lambda_{j,k}^R(y)-\Lambda_{j,k}^R(Y)-(\Lambda_{\ell,k}'(y)-\Lambda_{\ell,k}'(Y)) > \delta_3$ and $\Lambda_{j,k}^R(y')-\Lambda_{j,k}^R(Y)-(\Lambda_{\ell,k}'(y')-\Lambda_{\ell,k}'(Y)) < -\delta_3$. Furthermore, $|\Lambda_{\ell,k}'(Y)-\Lambda_{j,k}^R(Y)| \leq |\Lambda_{\ell,k}'(Y)-C_\ell'(Y)|+|R_j'(Y)-\Lambda_{j,k}^R(Y)| \leq \delta_3/3$ since $\mathcal{A}_{7,k}$, $\mathcal{A}_{8,k}$ do not occur. We deduce $\Lambda_{j,k}^R(y)-\Lambda_{\ell,k}'(y) \geq 2\delta_3/3> 0$ and $\Lambda_{j,k}^R(y')-\Lambda_{\ell,k}'(y')\leq -2\delta_3/3< 0$. Therefore $\Lambda_{j,k}^R$ has coalesced in $[Y-\delta_2,Y+\delta_2]$ with some $\Lambda_{\ell',k}'$, $\ell'\in\mathcal{L}$. This implies that for any $y \in [Y+\delta_2,y_0+1]$, we have $\Lambda_{j,k}'(y)=\Lambda_{\ell',k}'(y)$, with $|\Lambda_{\ell',k}'(y)-C_{\ell'}'(y)| \leq \delta_3/6$ since $\mathcal{A}_{8,k}$ does not occur, and $C_j'(y)=C_{\ell'}'(y)$, so $|\Lambda_{j,k}'(y)-C_{j}'(y)| \leq \delta_3/6 \leq \varepsilon$. 

We now consider the case $C_j'(Y)=\lambda(Y)$. Then since $\mathcal{A}_{7,k}$ does not occur, $\lambda(Y) \leq \Lambda_{j,k}^R(Y) \leq R_j'(Y)=\lambda(Y)$, and if $\ell$ is the smallest element of $\mathcal{L}$, we have $\Lambda_{\ell,k}^R(Y) \leq R_\ell'(Y)$. If we had $C_\ell'(Y) \neq R_\ell'(Y)$, then there would be some $\ell'<\ell$ so that $C_\ell'(Y) = C_{\ell'}'(Y)$, but this would contradict the minimality of $\ell$. Hence $R_\ell'(Y) = C_\ell'(Y)=\lambda(Y)$, which yields $\lambda(Y) \leq \Lambda_{\ell,k}^R(Y) \leq R_\ell'(Y)=\lambda(Y)$, hence $\Lambda_{j,k}^R(Y)=\Lambda_{\ell,k}^R(Y)$. In addition, if we had $\Lambda_{\ell,k}'(Y) \neq \Lambda_{\ell,k}^R(Y)$, there would be some $\ell' < \ell$ so that $\Lambda_{\ell,k}'(Y) = \Lambda_{\ell',k}'(Y)$. Since $\mathcal{A}_1$ and $\mathcal{A}_{8,k}$ do not occur, this would imply $|C_{\ell'}'(Y)-R_{j}'(Y)| \geq \delta_1$ and $|C_{\ell'}'(Y)-\Lambda_{\ell',k}'(Y)| \leq \delta_3/6$, hence $|\Lambda_{\ell',k}'(Y)-R_{j}'(Y)| \geq 5\delta_1/6$, thus $|\Lambda_{\ell,k}'(Y)-C_{\ell}'(Y)| \geq 5\delta_1/6$, which would contradict the fact $\mathcal{A}_{8,k}$ does not occur. Therefore we have $\Lambda_{\ell,k}'(Y) = \Lambda_{\ell,k}^R(Y)$, hence $\Lambda_{j,k}^R(Y)=\Lambda_{\ell,k}'(Y)$. This implies $\Lambda_{j,k}^R$ coalesces in $[Y-\delta_2,Y]$ with some $\Lambda_{\ell',k}'$, $\ell'\in\mathcal{L}$. Consequently, for any $y \in [Y,y_0+1]$, we have $\Lambda_{j,k}'(y)=\Lambda_{\ell',k}'(y)$, thus since $\mathcal{A}_{8,k}$ does not occur, $|\Lambda_{j,k}'(y)-C_{\ell'}'(y)|\leq \delta_3/6$, hence $|\Lambda_{j,k}'(y)-C_{j}'(y)|\leq \delta_3/6\leq \varepsilon$ for any $y \in [Y,y_0+1]$.

It remains to consider the case $y \in [Y-\delta_2,Y+\delta_2]$, $y \geq x_j$. Then $\Lambda_{j,k}'(y)=\Lambda_{j,k}^R(y)$ or $\Lambda_{\ell,k}'(y)$ with $\ell \in \mathcal{L}$. Since $\mathcal{A}_{7,k}$ and $\mathcal{A}_{8,k}$ do not occur, we have $|\Lambda_{j,k}'(y)-R_j'(y)| \leq \delta_3/6$ in the first case and $|\Lambda_{j,k}'(y)-C_\ell'(y)| \leq \delta_3/6$ in the second. Since $\mathcal{A}_4$ occurs, $|R_j'(y)-R_j'(Y)| \leq \delta_1/6$, $|C_\ell'(y)-C_\ell'(Y)| \leq \delta_1/6$ and $|C_j'(y)-C_j'(Y)| \leq \delta_1/6$, and remembering $R_j'(Y)=C_\ell'(Y)=C_j'(Y)$, we obtain $|\Lambda_{j,k}'(y)-C_j'(Y)| \leq \delta_3/6 + \delta_1/6 \leq \delta_1/3$, hence $|\Lambda_{j,k}'(y)-C_j'(y)| \leq \delta_1/3+\delta_1/6=\delta_1/2 \leq \varepsilon$. Consequently, for each $y \in [Y-\delta_2,Y+\delta_2]$, $y \geq x_j$, we have $|\Lambda_{j,k}'(y)-C_j'(y)| \leq \varepsilon$, which ends the proof of the lemma.
\end{proof}

Lemma \ref{lem_cons2} implies $(\Lambda_{j(1),k}'(y_1),...,\Lambda_{j(r),k}'(y_r))$ converges in $\mathds{P}_{\lambda,\chi}$-probability to $(C_{j(1)}'(y_1),...,C_{j(r)}'(y_r))$ when $k$ tends to $+\infty$, hence in distribution under $\mathds{P}_{\lambda,\chi}$. To prove that $\mathds{P}_{\lambda,\chi}(\forall\, i \in \{1,...,r\}, \Lambda_{j(i),k}'(y_i)\in(a_i,b_i))$ converges to $\mathds{P}_{\lambda,\chi}(\forall\, i \in \{1,...,r\}, C_{j(i)}'(y_i)\in(a_i,b_i))$ when $k$ tends to $+\infty$, which implies Proposition \ref{prop_cons}, by the Portmanteau Theorem it is enough to prove that for all $i \in \{1,...,r\}$, we have $\mathds{P}_{\lambda,\chi}(C_{j(i)}'(y_i)=a_i)=\mathds{P}(C_{j(i)}'(y_i)=b_i)=0$, which is given by Lemma \ref{lem_cons_prob0}.
\end{proof}

\subsection{Lemma 8.2 of \cite{Toth_et_al1998}}

The proof of Lemma 8.2 of \cite{Toth_et_al1998} relies on an upper bound on the variations of a RAB, which does not hold in our case, since if a $(\lambda,\chi)$-RAB hits the barrier, it can fluctuate like the barrier itself, which may be very wildly. Therefore we cannot use the argument of \cite{Toth_et_al1998} as such. However, we can have an estimate on the fluctuations of a $(\lambda,\chi)$-RAB when it is ``far from the barrier'', so we modify the proof in order to use the estimate only ``far from the barrier''. For any $x\in\mathds{R}$, we denote $M'(x)=\{\Lambda_{(\tilde x_n,\tilde h_n)}(x)\,|\,n\in\mathds{N},\tilde x_n < x,(\tilde x_n,\tilde h_n)\in\mathds{R}_\lambda^2\}$. The equivalent of Lemma 8.2 of \cite{Toth_et_al1998} in our setting is the following lemma.

\begin{lemma}\label{lem_cons_82}
 Almost surely, for all $x\in \mathds{R}$, the set $M'(x)$ is dense in $[\lambda(x),+\infty)$. 
\end{lemma}

\begin{proof}
It is enough to show that for any dyadic numbers $K>0$, $0 < \alpha < h$, almost surely for all $x\in(-K,K]$ we have $M'(x) \cap (\lambda(x)+h-\alpha,\lambda(x)+h+\alpha)\neq\emptyset$. Let $K>0$, $0 < \alpha < h$ be dyadic numbers. $\lambda$ is a continuous function, hence uniformly continuous on $[-K,K]$, so there exists some finite random $k_0$ so that when $k \geq k_0$, for all $x,y\in[-K,K]$ so that $|x-y| \leq 2^{-k}K$ we have $|\lambda(x)-\lambda(y)| \leq \alpha/3$. For any $j\in\{-2^k,...,2^k-1\}$, we choose $a_{k,j}$ a random dyadic number so that $|a_{k,j}-(\lambda(j2^{-k}K)+h)| \leq \alpha/3$. If for all $j\in\{-2^k,...,2^k-1\}$, for all $x \in[j2^{-k}K,(j+1)2^{-k}K]$ we have $|\Lambda_{(j2^{-k}K,a_{k,j})}(x)-a_{k,j}| <\alpha/3$, then for all $x\in(-K,K]$, for $j\in\{-2^k,...,2^k-1\}$ such $x\in(j2^{-k}K,(j+1)2^{-k}K]$, when $k \geq k_0$ we have
\[
|\Lambda_{(j2^{-k}K,a_{k,j})}(x)-(\lambda(x)+h)| \leq |\Lambda_{(j2^{-k}K,a_{k,j})}(x)-a_{k,j}|+|a_{k,j}-(\lambda(j2^{-k}K)+h)|+|\lambda(j2^{-k}K)-\lambda(x)|<\alpha,
\]
hence $M'(x) \cap (\lambda(x)+h-\alpha,\lambda(x)+h+\alpha)\neq\emptyset$. Moreover, thanks to Lemma \ref{lem_RAB_max}, for any $j\in\{-2^k,...,2^k-1\}$, we have 
\[
\mathds{P}\left(k\geq k_0,\exists\,x\in[j2^{-k}K,(j+1)2^{-k}K],|\Lambda_{(j2^{-k}K,a_{k,j})}(x)-a_{k,j}| \geq \alpha/3\right)
\]
\[
  \leq \sum_{\ell\in\mathds{D}}\mathds{E}\left(\mathds{1}_{\{k \geq k_0,a_{k,j}=\ell\}}\mathds{P}_{\lambda,\chi}\left(\exists\,x\in[j2^{-k}K,(j+1)2^{-k}K],|\Lambda_{(j2^{-k}K,\ell)}(x)-\ell| \geq \alpha/3\right)\right)
 \] 
 \[ 
  \leq \sum_{\ell\in\mathds{D}}\mathds{E}\left(\mathds{1}_{\{k \geq k_0,a_{k,j}=\ell\}}6\frac{\sqrt{2vK2^{-k}}}{\alpha\sqrt{\pi}}\exp\left(-\frac{\alpha^2}{18vK2^{-k}}\right)\right)
  \leq 6\frac{\sqrt{2vK2^{-k}}}{\alpha\sqrt{\pi}}\exp\left(-\frac{\alpha^2}{18vK2^{-k}}\right).
\]
 Therefore $\mathds{P}(k \geq k_0, \exists\,j\in\{-2^k,...,2^k-1\},\exists\,x\in[j2^{-k}K,(j+1)2^{-k}K],|\Lambda_{(j2^{-k}K,a_{k,j})}(x)-a_{k,j}| \geq \alpha/3) \leq 12\cdot 2^{k/2}\frac{\sqrt{2vK}}{\alpha\sqrt{\pi}}\exp(-\frac{\alpha^2}{18vK}2^k)$, which tends to 0 when $k$ tends to $+\infty$. This implies $\mathds{P}(k \geq k_0,\forall\,j\in\{-2^k,...,2^k-1\},\forall\,x\in[j2^{-k}K,(j+1)2^{-k}K],|\Lambda_{(j2^{-k}K,a_{k,j})}(x)-a_{k,j}| < \alpha/3)$ tends to 1 when $k$ tends to $+\infty$. We deduce that almost surely, for all $x\in(K,-K]$ we have $M'(x) \cap (\lambda(x)+h-\alpha,\lambda(x)+h+\alpha)\neq\emptyset$, which ends the proof of the lemma.
\end{proof}

\subsection{Equation (8.49) of \cite{Toth_et_al1998}}

Like Lemma 8.1 of \cite{Toth_et_al1998}, their Equation (8.49) relies on an estimate of the probability two RABs coalesce, which we do not have in our setting, so we have to give another proof. Let $q < q'$ be two elements of $\mathds{D}$. For any $K \in \mathds{N}^*$, $p\in\mathds{N}^*$, we denote $N_p^K=\mathrm{card}(\{\Lambda_{(q,2^{-p}j)}(q')\,|\,j\in\{-2^pK,...,2^pK\},2^{-p}j>\lambda(q)\})$. To replace Equation (8.49) of \cite{Toth_et_al1998}, we need an upper bound on $\mathds{E}(N_p^K)$ which does not depend on $p$, which will be the following lemma. 

\begin{lemma}
 $\mathds{E}(N_p^K) \leq 2+\frac{2K}{\sqrt{\pi v(q'-q)}}$. 
\end{lemma}

\begin{proof}
The proof in \cite{Toth_et_al1998} relies on a small upper bound on $\mathds{P}(\Lambda_{(q,2^{-p}(j-1))}(q')<\Lambda_{(q,2^{-p}j)}(q'))$, which we do not have in our setting because of the possibly wild fluctuations of $\lambda$. However, we will be able to bound $\mathds{P}(\lambda(q') < \Lambda_{(q,2^{-p}(j-1))}(q')<\Lambda_{(q,2^{-p}j)}(q'))$, as away from $\lambda$ the processes $\Lambda_{(q,2^{-p}(j-1))}$, $\Lambda_{(q,2^{-p}j)}$ behave like Brownian motions. To make this weaker bound sufficient, our idea will be to notice that only the lowest $\Lambda_{(q,2^{-p}j)}(q')$ can be equal to $\lambda(q')$. We consider $\{\Lambda_{(q,2^{-p}j)}(q')\,|\,j\in\{-2^pK,...,2^pK\},2^{-p}j>\lambda(q)\}$. It may contain $\lambda(q')$, the smallest $\Lambda_{(q,2^{-p}j)}(q')$ strictly larger than $\lambda(q')$, and all the $\Lambda_{(q,2^{-p}j')}(q')$ larger than this $\Lambda_{(q,2^{-p}j)}(q')$. Therefore 
\[
\mathds{E}(N_p^K) \leq 2+\sum_{j=-2^pK+2}^{2^pK}\mathds{P}(2^{-p}(j-1)>\lambda(q),\lambda(q')<\Lambda_{(q,2^{-p}(j-1))}(q')<\Lambda_{(q,2^{-p}j)}(q')).
\]
Moreover, for any $j\in\{-2^pK+2,...,2^pK\}$, by Theorem \ref{thm_cons_forward} the processes $\Lambda_{(q,2^{-p}(j-1))}$ and $\Lambda_{(q,2^{-p}j)}$ form a $(\lambda,\chi)$-FICRAB starting from $(q,2^{-p}(j-1))$ and $(q,2^{-p}j)$. Let $x\in\mathds{R}$, $h<h'$, let $(W_y)_{y \geq x}$, $(W_y')_{y \geq x}$ be two independent Brownian motions with $W_x=h$, $W_x'=h'$, let $(R_y)_{y \geq x}$, $(R_y')_{y \geq x}$ be the $(\lambda,\chi)$-RABs driven by $(W_y)_{y \geq x}$, $(W_y')_{y \geq x}$ starting from $(x,h)$, $(x,h')$, and $((C_y)_{y \geq x}$, $(C_y')_{y \geq x})$ be the $(\lambda,\chi)$-FICRAB constructed from $(R_y)_{y \geq x}$, $(R_y')_{y \geq x}$. In order to prove the lemma, it is enough to prove that for any $y>x$ we have $\mathds{P}(h>\lambda(x),\lambda(y) < C_y < C_y') \leq \frac{h'-h}{\sqrt{\pi v (y-x)}}$. Indeed, this implies 
\[
 \mathds{E}(N_p^K) \leq 2+\sum_{j=-2^pK+2}^{2^pK}\frac{2^{-p}}{\sqrt{\pi v(q'-q)}} 
 \leq 2+2 \cdot 2^pK\frac{2^{-p}}{\sqrt{\pi v(q'-q)}} = 2+\frac{2K}{\sqrt{\pi v(q'-q)}}.
\]

We now prove that if $h>\lambda(x)$ and $\lambda(y) < C_y < C_y'$ then for all $z \in [x,y]$, $W_z < W_z'$. We assume by contradiction that there exists $z_0\in[x,y]$ so that $W_{z_0} \geq W_{z_0}'$. There are several cases. If there exists $z\in[x,y]$ so that $W_z'=\lambda(z)$, then if $\bar z$ is the inf of these $z$ we have $R_{\bar z}'=W_{\bar z}'=\lambda(\bar z)$. We then have $R_{\bar z}' \leq R_{\bar z}$ while $R_x'>R_x$, so there exists $z\in[x,\bar z]$ so that $R_z=R_z'$, hence $C_y=C_y'$, which gives a contradiction. We now assume that for all $z\in[x,y]$ we have $W_z'>\lambda(z)$, which implies for all $z\in[x,y]$ we have $R_z'=W_z'$. In addition, if $R_y=\lambda(y)$ then $C_y=\lambda(y)$ which would be a contradiction, so we may assume $R_y>\lambda(y)$. By the construction of $(R_y)_{y \geq x}$, it implies that for all $z\in[x,y]$ we have $R_z \geq W_z$. In particular, $R_{z_0} \geq W_{z_0} \geq W_{z_0}' = R_{z_0}'$, while $R_x'>R_x$, hence there exists $z\in[x,z_0]$ so that $R_z=R_z'$, hence $C_y=C_y'$ which gives a contradiction. Consequently, if $h>\lambda(x)$ and $\lambda(y) < C_y < C_y'$ then for all $z \in [x,y]$, $W_z < W_z'$. This yields $\mathds{P}(h>\lambda(x),\lambda(y) < C_y < C_y') \leq \mathds{P}(\forall\,z\in[x,y],W_z<W_z') \leq \frac{h'-h}{\sqrt{\pi v (y-x)}}$ by Lemma \ref{lem_BM_max2}, which ends the proof of the lemma. 
\end{proof}

\section{Backward lines: proof of Theorem \ref{thm_cons_backward}}\label{sec_cons_backward}

To prove Theorem \ref{thm_cons_backward}, which states $(\Lambda_{(-x,h)}^*(-.))_{(x,h)\in \mathds{R}^2}$ is a system of forward lines above $(\lambda(-.),-\chi)$, we show $(\Lambda_{(-x,h)}^*(-.))_{(x,h)\in \mathds{R}^2}$ satisfies the properties of Theorem \ref{thm_cons_forward}, which characterize the law of a system of forward lines. The last three properties of Theorem \ref{thm_cons_forward} can be proven as in \cite{Toth_et_al1998} (see their Lemma 9.1); the difficult part is to prove that for any $p \in \mathds{N}^*$, if $(x_1,h_1),...,(x_p,h_p)\in\mathds{R}^2$, then $(\Lambda_{(-x_1,h_1)}^*(-.),...,\Lambda_{(-x_p,h_p)}^*(-.))$ is a $(\lambda(-.),-\chi)$-FICRAB. As in \cite{Toth_et_al1998}, we first prove that conditionally to $(\lambda,\chi)$, the backward lines are Markov processes in Section \ref{sec_backward_Markov}, then determine their transition probabilities in Section \ref{sec_backward_transition}, continue by showing they are independent as long as they stay apart in Section \ref{sec_backward_indep}, and finally prove that they merge when they meet in Section \ref{sec_backward_coal}. Each of these points needs modifications with respect to \cite{Toth_et_al1998}. In this section, we work conditionally to $(\lambda,\chi)$ even when it is not mentioned.

\subsection{Markov property of the backward lines given $(\lambda,\chi)$}\label{sec_backward_Markov}

 The proof of the Markov property of the $\Lambda_{(x,h)}^*$ in \cite{Toth_et_al1998} relies on the fact that for $x_0<x$, at the left of $x_0$ we have almost surely $\Lambda_{(x,h)}^*=\Lambda_{(x_0,\Lambda_{(x,h)}^*(x_0))}^*$, which does not depend on what happens at the right of $x_0$. However, this does not work if $\Lambda_{(x,h)}^*(x_0)=\lambda(x_0)$, which may happen with positive probability with our more general $\lambda$, thus we need additional arguments, which are given in the following proposition. It is in its proof that we need the assumption $(\lambda,\chi)$ nice. 

 \begin{proposition}\label{prop_cons_Markov_lambda}
 For any $(x,h)\in \mathds{R}_\lambda^2$, $x_0 < x$, we have the following.
 \begin{itemize}
  \item If $x_0 \leq \chi$, almost surely if $\Lambda_{(x,h)}^*(x_0)=\lambda(x_0)$ then for all $y \leq x_0$ we have $\Lambda_{(x,h)}^*(y)=\lambda(y)$.
  \item If $x_0 > \chi$, if $\Lambda_{(x,h)}^*(x_0)=\lambda(x_0)$ then for all $y \leq x_0$ we have $\Lambda_{(x,h)}^*(y)=\sup\{\Lambda_{(\tilde x_n,\tilde h_n)}(y) \,|\, \tilde x_n < y, \Lambda_{(\tilde x_n,\tilde h_n)}(x_0) =\lambda(x_0)\}$ (where the sup is $\lambda(y)$ if the set is empty).
 \end{itemize}
\end{proposition}
 
 If we assume Proposition \ref{prop_cons_Markov_lambda}, we can prove the Markov property of the backward lines with arguments similar to those of \cite{Toth_et_al1998}. Indeed, one can prove as in Lemma 9.2 of \cite{Toth_et_al1998} that for any $(x,h)\in \mathds{R}_\lambda^2$, $x_0 \leq x$, almost surely if $\Lambda_{(x,h)}^*(x_0)>\lambda(x_0)$ then for all $y \leq x_0$ we have $\Lambda_{(x_0,\Lambda_{(x,h)}^*(x_0))}^*(y)=\Lambda_{(x,h)}^*(y)$. Moreover, one can prove as they do just before that lemma that the families $\{\Lambda_{(x,h)}(y),x \leq y \leq x_0, h>\lambda(x)\}\cup\{\Lambda_{(x,h)}^*(y),y \leq x \leq x_0,h>\lambda(x)\}$ and $\{\Lambda_{(x,h)}^*(y),x_0 \leq y \leq x,h>\lambda(x)\}$ are independent. This implies that for any $(x_1,h_1),...,(x_p,h_p) \in \mathds{R}_\lambda^2$, $(\Lambda_{(x_1,h_1)}^*(-.),...,\Lambda_{(x_p,h_p)}^*(-.))$ is a Markov process. 
 
\begin{proof}[Proof of Proposition \ref{prop_cons_Markov_lambda}.]
 We will use different arguments depending on whether $x_0 > \chi$, $x_0 = \chi$ or $x_0 < \chi$.

We begin with the case $x_0>\chi$. We assume $\Lambda_{(x,h)}^*(x_0)=\lambda(x_0)$. Let $y \leq x_0$. Then if $m\in\mathds{N}$ is so that $\tilde x_m < y$ and $\Lambda_{(\tilde x_m,\tilde h_m)}(x)<h$, we have $\tilde x_m < x_0$ and $\Lambda_{(\tilde x_m,\tilde h_m)}(x)<h$ hence $\Lambda_{(\tilde x_m,\tilde h_m)}(x_0)=\lambda(x_0)$. Conversely, if $m\in\mathds{N}$ is so that $\tilde x_m < y$ and $\Lambda_{(\tilde x_m,\tilde h_m)}(x_0)=\lambda(x_0)$, then $\Lambda_{(\tilde x_m,\tilde h_m)}(x)=\lambda(x)<h$. We deduce $\{m\in\mathds{N}\,|\,\tilde x_m < y, \Lambda_{(\tilde x_m,\tilde h_m)}(x)<h\}=\{m\in\mathds{N}\,|\,\tilde x_m < y, \Lambda_{(\tilde x_m,\tilde h_m)}(x_0)=\lambda(x_0)\}$, hence $\Lambda_{(x,h)}^*(y)=\sup\{\Lambda_{(\tilde x_n,\tilde h_n)}(y) \,|\, \tilde x_n < y, \Lambda_{(\tilde x_n,\tilde h_n)}(x_0) =\lambda(x_0)\}$.

We now consider the case $x_0=\chi$. It is enough to show that almost surely, for all $n\in\mathds{N}$ so that $\tilde x_n < \chi$, $\tilde h_n > \lambda(\tilde x_n)$, we have $\Lambda_{(\tilde x_n,\tilde h_n)}(\chi)>\lambda(\chi)$. Since $(\lambda,\chi)$ is nice, this comes from Lemma \ref{lem_meeting_nice}. 

We now deal with the case $x_0 < \chi$. We will use the following lemma, proven later. 

\begin{lemma}\label{lem_cons_back_Markov}
 For any $x_0<\chi$, $x>x_0$, $n\in\mathds{N}$ so that $\tilde x_n < x_0$, almost surely if $\Lambda_{(\tilde x_n,\tilde h_n)}(x_0)=\lambda(x_0)$, there exists $h'>\lambda(x_0)$ so that $\Lambda_{(x_0,h')}(x)=\Lambda_{(\tilde x_n,\tilde h_n)}(x)$.
\end{lemma}

 We now assume that for any $n\in\mathds{N}$ so that $\tilde x_n < x_0$, if $\Lambda_{(\tilde x_n,\tilde h_n)}(x_0)=\lambda(x_0)$, there exists $h'>\lambda(x_0)$ so that $\Lambda_{(x_0,h')}(x)=\Lambda_{(\tilde x_n,\tilde h_n)}(x)$, which happens almost surely by Lemma \ref{lem_cons_back_Markov}. We recall Lemma \ref{lem_cons_82}: almost surely for each $y \in \mathds{R}$, the set $M'(y)=\{\Lambda_{(\tilde x_m,\tilde h_m)}(y) \,|\,m\in\mathds{N},\tilde x_m < y,(\tilde x_m,\tilde h_m)\in\mathds{R}_\lambda^2\}$ is dense in $[\lambda(y),+\infty)$ (since it holds for any choice of $(\lambda,\chi)$, it holds under our conditioning). Together these almost sure events imply that, for all $n\in\mathds{N}$ so that $\tilde x_n < x_0$, if $\Lambda_{(\tilde x_n,\tilde h_n)}(x_0)=\lambda(x_0)$, there exists $m\in \mathds{N}$ so that $\tilde x_m < x_0$, $\Lambda_{(\tilde x_m,\tilde h_m)}(x_0)>\lambda(x_0)$ and $\Lambda_{(\tilde x_m,\tilde h_m)}(x)=\Lambda_{(\tilde x_n,\tilde h_n)}(x)$. Consequently, if there exists $n\in\mathds{N}$ so that $\tilde x_n < x_0$, $\Lambda_{(\tilde x_n,\tilde h_n)}(x) < h$ and $\Lambda_{(\tilde x_n,\tilde h_n)}(x_0)=\lambda(x_0)$, then there exists $m\in \mathds{N}$ so that $\tilde x_m < x_0$, $\Lambda_{(\tilde x_m,\tilde h_m)}(x_0)>\lambda(x_0)$ and $\Lambda_{(\tilde x_m,\tilde h_m)}(x)<h$, hence $\Lambda_{(x,h)}^*(x_0) > \lambda(x_0)$. Therefore if $\Lambda_{(x,h)}^*(x_0) = \lambda(x_0)$, there is no $n \in \mathds{N}$ so that $\tilde x_n < x_0$, $\Lambda_{(\tilde x_n,\tilde h_n)}(x) < h$ and $\Lambda_{(\tilde x_n,\tilde h_n)}(x_0)=\lambda(x_0)$, thus there is no $n \in \mathds{N}$ so that $\tilde x_n < x_0$ and $\Lambda_{(\tilde x_n,\tilde h_n)}(x) < h$. This implies $\Lambda_{(x,h)}^*(y) = \lambda(y)$ for all $y \leq x_0$, which ends the proof of Proposition \ref{prop_cons_Markov_lambda} pending the proof of Lemma \ref{lem_cons_back_Markov}.
\end{proof}

\begin{proof}[Proof of Lemma \ref{lem_cons_back_Markov}.]
Let $x_0<\chi$, $x>x_0$, $n\in\mathds{N}$ so that $\tilde x_n < x_0$. If $x>\chi$, we replace $x$ by $\chi$, since if there is coalescence before $\chi$ there is coalescence before $x$. The idea of the proof is that if $\Lambda_{(x_0,h')}$ and $\Lambda_{(\tilde x_n,\tilde h_n)}$ were Brownian motions with $h'$ close to $\Lambda_{(\tilde x_n,\tilde h_n)}(x_0)$ they would be very likely to meet, and since we work in the portion of the space where they are reflected instead, the barrier will ``push up the process which is below'', increasing even more the likelihood of the two processes meeting. For any $k \in \mathds{N}^*$, we remember $(\Lambda_{(\tilde x_n,\tilde h_n)},\Lambda_{(x_0,\lambda(x_0)+2^{-k})})$ is a $(\lambda,\chi)$-FICRAB starting from $(\tilde x_n,\tilde h_n),(x_0,\lambda(x_0)+2^{-k})$. Let $(\tilde W_z)_{z \geq \tilde x_n}$, $(W_z)_{z \geq x_0}$ be two independent Brownian motions. We define $(\tilde R(y))_{y \geq \tilde x_n}$ as the $(\lambda,\chi)$-RAB starting from $(\tilde x_n,\tilde h_n)$ and driven by $(\tilde W_z)_{z \geq \tilde x_n}$, and for any $k \geq 1$, we define $(R_k(y))_{y \geq \tilde x_n}$ as the $(\lambda,\chi)$-RAB starting from $(x_0,\lambda(x_0)+2^{-k})$ and driven by $(W_z)_{z \geq x_0}$. Denoting $\mathcal{A}_k=\{\tilde R(x_0)=\lambda(x_0),\forall\, y \in [x_0,x], \tilde R(y)<R_k(y)\}$, it is enough to prove $\mathds{P}_{\lambda,\chi}(\mathcal{A}_k)$ tends to 0 when $k$ tends to $+\infty$. 

We now consider the event $\mathcal{A}_k'=\{\forall \, y \in[x_0,x], \tilde W_y-\tilde W_{x_0} < W_y-W_{x_0}+2^{-k}\}$. We are going to prove that if $\mathcal{A}_k'$ does not occur and $\tilde R(x_0)=\lambda(x_0)$, then there exists $y \in [x_0,x]$ so that $\tilde R(y) \geq R_k(y)$, which implies $\mathcal{A}_k \subset \mathcal{A}_k'$. We assume $\mathcal{A}_k'$ does not occur and $\tilde R(x_0)=\lambda(x_0)$.  If there exists $y \in [x_0,x]$ so that $R_k(y)=\lambda(y)$, then $\tilde R(y) \geq R_k(y)$. We now consider the case in which $R_k(y)>\lambda(y)$ for all $y \in [x_0,x]$, which implies $R_k(y)=\lambda(x_0)+2^{-k}+W_y-W_{x_0}$ for all $y \in [x_0,x]$. Since $\mathcal{A}_k'$ does not occur, there exists $y \in [x_0,x]$ so that $\tilde W_y-\tilde W_{x_0} \geq W_y-W_{x_0}+2^{-k}$. In addition, $\tilde R(y) \geq \tilde R(x_0)+\tilde W_y-\tilde W_{x_0}\geq \lambda(x_0)+W_y-W_{x_0}+2^{-k}=R_k(y)$. Consequently $\mathcal{A}_k \subset \mathcal{A}_k'$. In addition, Lemma \ref{lem_BM_max2} yields $\mathds{P}_{\lambda,\chi}(\mathcal{A}_k') \leq \frac{2^{-k}}{\sqrt{\pi v (x-x_0)}}$. We deduce $\mathds{P}_{\lambda,\chi}(\mathcal{A}_k)$ tends to 0 when $k$ tends to $+\infty$, which ends the proof of Lemma \ref{lem_cons_back_Markov}.
\end{proof}

\subsection{Transition probabilities of the backward lines}\label{sec_backward_transition}

In this section we prove that the backward lines have the right transition probabilities, which is Proposition \ref{prop_backward_transition}. The extension to our setting of the study of the transition probabilities made in \cite{Toth_et_al1998} has two major obstacles. The first one is that it relies on a relationship they establish between the distributions of reflected and absorbed Brownian motions (their Equation (9.25)), which crucially depends on the fact their barrier is the abscissa axis. We thus have to replace their equation by Lemma \ref{lem_cons_back_back}. Moreover, the arguments of \cite{Toth_et_al1998} do not allow to deal with the case $\Lambda_{(x,h)}^*(x') = \lambda(x')$. This was not a problem for them since they had $\mathds{P}(\Lambda_{(x,h)}^*(x') = \lambda(x'))=0$, but with our more general barriers this event may have a positive probability, hence we have to treat this case separately. 
 
 \begin{proposition}\label{prop_backward_transition}
   Let $(x,h)\in\mathds{R}_\lambda^2$, $x'' < x' \leq x$, $h' \geq \lambda(x')$, if $\Lambda_{(x,h)}^*(x') = h'$, then $\Lambda_{(x,h)}^*(x'')$ has the law of the marginal at $-x''$ of a $(\lambda(-.),-\chi)$-RAB starting at $(-x',h')$, and a barrier-starting one if $h'=\lambda(x')$ (see Remark \ref{rem_barrier_RABs}).
 \end{proposition}

\begin{proof}
 We can consider only the cases $x' \leq \chi$ and $x'' \geq \chi$, as knowing the transition probabilities from $x'$ to $\chi$ and from $\chi$ to $x''$ yields the transition probability from $x'$ to $x''$. 
 
We first assume $h' > \lambda(x')$. The result follows from the proof on Page 433 of \cite{Toth_et_al1998} if we replace their Equation (9.25) by the following lemma. 

\begin{lemma}\label{lem_cons_back_back}
 Let $(W_{-y}^\leftarrow)_{-x' \leq y \leq -x''}$ be a $(\lambda(-.),-\chi)$-RAB starting from $(-x',h')$\footnote{Or rather its restriction to $[-x',-x'']$.}, then for any $h''>\lambda(x'')$ we have $\mathds{P}_{\lambda,\chi}(\Lambda_{(x'',h'')}(x')>h')=\mathds{P}_{\lambda,\chi}(W_{x''}^\leftarrow < h'')$.
\end{lemma}

\begin{proof}
Let $(W_{y})_{y \in [x'',x']}$ be a Brownian motion with $W_{x''}=0$. We define $(W_{-y}^\leftarrow)_{-x' \leq y \leq -x''}$ as the $(\lambda(-.),-\chi)$-RAB starting from $(-x',h')$ and driven by $(W_{-y}-W_{x'})_{-x' \leq y \leq -x''}$. We also define $(W_y^\rightarrow)_{x'' \leq y \leq x'}$ as the $(\lambda,\chi)$-RAB starting from $(x'',h'')$ and driven by $(W_{y})_{y \in [x'',x']}$. $(W_y^\rightarrow)_{x'' \leq y \leq x'}$ has the same distribution as $\Lambda_{(x'',h'')}|_{[x'',x']}$, hence it is enough to prove $\mathds{P}_{\lambda,\chi}(W_{x'}^\rightarrow > h') = \mathds{P}_{\lambda,\chi}(W_{x''}^\leftarrow < h'')$.

We begin by assuming $x'' \geq \chi$. Then $(W_y^\rightarrow)_{x'' \leq y \leq x'}$ is a Brownian motion absorbed on $\lambda$, which may meet $\lambda$ or not. We first deal with the situation in which $(W_y^\rightarrow)_{x'' \leq y \leq x'}$ does not meet $\lambda$. Then, for all $y \in [x'',x']$ we have $h''+W_y > \lambda(y)$. We first assume $W_{x'}^\rightarrow > h'$, hence $h''+W_{x'} > h'$. Then since $(W_{-y}^\leftarrow)_{-x' \leq y \leq -x''}$ starts at $(-x',h')$, for any $y \in [-x',-x'']$, we have $W_{-y}^\leftarrow < h''+W_{-y}$ (indeed the two processes $W_{-.}^\leftarrow$ and $h''+W_{-.}$ cannot cross), thus $W_{x''}^\leftarrow < h''$. We now assume $W_{x'}^\rightarrow \leq h'$, which means $h''+W_{x'} \leq h'$. Then for all $y \in [-x',-x'']$, we have $h'+W_{-y}-W_{x'} \geq h''+W_{-y} > \lambda(y)$, hence $(W_y^\leftarrow)_{x'' \leq y \leq x'}$ does not meet $\lambda$, hence $W_{x''}^\leftarrow = h'+W_{x''}-W_{x'}\geq h''$. Therefore, if $(W_y^\rightarrow)_{x'' \leq y \leq x'}$ does not meet $\lambda$, then $W_{x'}^\rightarrow > h'$ is equivalent to $W_{x''}^\leftarrow < h''$ (in the same way, $W_{x'}^\rightarrow \geq h'$ is equivalent to $W_{x''}^\leftarrow \leq h''$). We now deal with the situation in which $(W_y^\leftarrow)_{x'' \leq y \leq x'}$ meets $\lambda$. Since $h'> \lambda(x')$, we have $W_{x'}^\rightarrow < h'$. Furthermore, let $Y=\inf\{y \geq x' \,|\, h''+W_y=\lambda(y)\}$ the place where $(W_y^\leftarrow)_{x'' \leq y \leq x'}$ meets $\lambda$. For $y \in [x'',Y]$, we have $W_y^\leftarrow-W_Y^\leftarrow \geq W_y-W_Y = W^\rightarrow_y-W^\rightarrow_Y=W^\rightarrow_y-\lambda(Y)$. Since $W_Y^\leftarrow \geq \lambda(Y)$, we get $W_{x''}^\leftarrow \geq W^\rightarrow_{x''}=h''$, thus $W_{x''}^\leftarrow \geq h''$. We deduce $\{W_{x'}^\rightarrow > h'\} = \{W_{x''}^\leftarrow < h''\}$, hence $\mathds{P}_{\lambda,\chi}(W_{x'}^\rightarrow > h') = \mathds{P}_{\lambda,\chi}(W_{x''}^\leftarrow < h'')$.

We now consider the case $x' \leq \chi$. Then the roles of $(W_y^\rightarrow)_{x'' \leq y \leq x'}$ and $(W_y^\leftarrow)_{x'' \leq y \leq x'}$ are reversed with respect to the previous case: $(W_y^\rightarrow)_{x'' \leq y \leq x'}$ is reflected and $(W_y^\leftarrow)_{x'' \leq y \leq x'}$ is absorbed. Therefore the previous reasoning yields $\{W_{x'}^\rightarrow > h'\} \subset \{W_{x''}^\leftarrow < h''\}$ and $\{W_{x''}^\leftarrow < h''\} \subset \{W_{x'}^\rightarrow \geq h'\}$. Moreover, Lemma \ref{lem_cons_prob0} implies $\mathds{P}_{\lambda,\chi}(W_{x'}^\rightarrow = h')=0$, hence $\mathds{P}_{\lambda,\chi}(W_{x'}^\rightarrow > h') = \mathds{P}_{\lambda,\chi}(W_{x''}^\leftarrow < h'')$. 
\end{proof}

 We also need to deal with the case $h'=\lambda(x')$. Then only $x' < x$ is interesting. If $x'\leq\chi$, Proposition \ref{prop_cons_Markov_lambda} yields that almost surely, if $\Lambda_{(x,h)}^*(x') = \lambda(x')$ then $\Lambda_{(x,h)}^*(y) = \lambda(y)$ for all $y \leq x'$, which is enough. We now assume $x'' \geq \chi$ and study the distribution of $\Lambda_{(x,h)}^*(x'')$ if $\Lambda_{(x,h)}^*(x') = \lambda(x')$. We notice that if $\Lambda_{(x,h)}^*(x') = \lambda(x')$, for any $\hat x<x'$, $\hat h >\lambda(\hat x)$, either $\Lambda_{(\hat x,\hat h)}(x')=\lambda(x')$ or $\Lambda_{(\hat x,\hat h)}(x)\geq h$. Consequently, for $h''>\lambda(x'')$, an approach similar to that of Equations (9.22) and (9.23) of \cite{Toth_et_al1998} yields 
 \[
 \mathds{P}_{\lambda,\chi}(\Lambda_{(x,h)}^*(x'')> h''|\Lambda_{(x,h)}^*(x') = \lambda(x'))\geq \mathds{P}_{\lambda,\chi}(\Lambda_{(x'',h'')}(x)<h|\Lambda_{(x,h)}^*(x') = \lambda(x'))=\mathds{P}_{\lambda,\chi}(\Lambda_{(x'',h'')}(x')=\lambda(x')),
 \]
 \[
 \mathds{P}_{\lambda,\chi}(\Lambda_{(x,h)}^*(x'')\leq  h''|\Lambda_{(x,h)}^*(x') = \lambda(x'))\geq \mathds{P}_{\lambda,\chi}(\Lambda_{(x'',h'')}(x)\geq h|\Lambda_{(x,h)}^*(x') = \lambda(x'))=\mathds{P}_{\lambda,\chi}(\Lambda_{(x'',h'')}(x')\neq\lambda(x')),
 \]
 therefore 
 \[
 \mathds{P}_{\lambda,\chi}(\Lambda_{(x,h)}^*(x'')> h''|\Lambda_{(x,h)}^*(x') = \lambda(x'))=\mathds{P}_{\lambda,\chi}(\Lambda_{(x'',h'')}(x')=\lambda(x')).
 \]
 In addition, we notice that $\mathds{P}_{\lambda,\chi}(\Lambda_{(x'',h'')}(x')=\lambda(x'))$ is the probability a $(\lambda,\chi)$-RAB starting from $(x'',h'')$ is absorbed before $x'$. Let $(W_y)_{y \in \mathds{R}}$ be a Brownian motion with $W_{x''}=0$, we thus have $\mathds{P}_{\lambda,\chi}(\Lambda_{(x'',h'')}(x')=\lambda(x')) = \mathds{P}_{\lambda,\chi}(\exists\, y \in [x'',x'], h''+W_y \leq \lambda(y)) = \mathds{P}_{\lambda,\chi}(\sup_{x'' \leq y \leq x'}(\lambda(y)-W_y) \geq h'')=\mathds{P}_{\lambda,\chi}(W_{x''}-W_{x'}+\sup_{x'' \leq y \leq x'}(\lambda(y)-(W_y-W_{x'})) \geq h'')$. Denoting  $(R_y)_{y \geq -x'}$ the barrier-starting $(\lambda(-.),-\chi)$-RAB starting from $(-x',\lambda(x'))$ and driven by $(W_{-y}-W_{x'})_{y \geq -x'}$, we deduce $\mathds{P}_{\lambda,\chi}(\Lambda_{(x'',h'')}(x')=\lambda(x'))=\mathds{P}_{\lambda,\chi}(R_{x''} \geq h'')=\mathds{P}_{\lambda,\chi}(R_{x''} > h'')$ by Lemma \ref{lem_cons_prob0}, which ends the proof of Proposition \ref{prop_backward_transition}. 
 \end{proof}
 
 \subsection{Independence of the backward lines until they meet}\label{sec_backward_indep}
 
 We need to prove that if $(x,h_1),(x,h_2) \in \mathds{R}_\lambda^2$, then $\Lambda_{(x,h_1)}^*(-.)$ and $\Lambda_{(x,h_2)}^*(-.)$ are independent until they meet. Intuitively, the idea of the proof is that the processes are independent as long as they live in disjoint boxes of $\mathds{R}^2$, so we enclose their trajectories in a sequence of disjoint boxes until they get too close to each other. Unfortunately, because of our more general barriers, we must change the definition of the boxes given in \cite{Toth_et_al1998}. Indeed, the boxes are constructed so that when $\Lambda_{(x,h_i)}^*$ enters one of them, it is very unlikely to fluctuate enough to go above the top of the box or below its bottom. In \cite{Toth_et_al1998}, they have uniform estimates on the fluctuations of the RABs, so their boxes are simple rectangles. But in our setting, when a $(\lambda,\chi)$-RAB is absorbed or reflected by the barrier, it may fluctuate as much as $\lambda$, which may be very wildly, so we need more complex boxes to contain it, and this makes the gestion of the boxes more complicated. For any $i\in\{1,2\}$, $n\in\mathds{N}^*$, $k\in\mathds{N}$, $y \in [x-5^{-n}(k+1),x-5^{-n}k)$, we define the ``lower frontier of the box''
 \[
 a_{n,k}^{(i,-)}(y)=\left\{\begin{array}{l}
                     \Lambda_{(x,h_i)}^*(x-5^{-n}k)-2^{-n}\text{ if }\lambda(z) < \Lambda_{(x,h_i)}^*(x-5^{-n}k)-2^{-n}\text{ for all }z \in (y,x-5^{-n}k], \\
                     \lambda(y)-2^{-n}\text{ otherwise.}
                    \end{array}\right.
\]
The second case was introduced because if $\Lambda_{(x,h_i)}^*$ is too close to the barrier, it may be absorbed by it, hence the bottom of the box needs to be at least as low as the barrier. If $y=x-5^{-n}k$, the definition is the same except that the condition ``$\lambda(z) < \Lambda_{(x,h_i)}^*(x-5^{-n}k)-2^{-n}$ for all $z \in (y,x-5^{-n}k]$'' is replaced by ``$\lambda(x-5^{-n}k) < \Lambda_{(x,h_i)}^*(x-5^{-n}k)-2^{-n}$''. For $y \in [x-5^{-n}(k+1),x-5^{-n}k]$, we also define the ``upper frontier of the box'' 
 \[
 a_{n,k}^{(i,+)}(y)=\max\left(\Lambda_{(x,h_i)}^*(x-5^{-n}k)+2^{-n},\sup_{y \leq z\leq x-5^{-n}k}\lambda(z)+2^{-n}\right).
 \]
 The second term in the maximum is there in case the barrier has a large jump and ``pushes $\Lambda_{(x,h_i)}^*$ upwards'', to keep the process in the box. Our definition of the box $A^{(i)}_{n,k}$ is $\{(y,h)\in\mathds{R}^2 \,|\, y \in (x-5^{-n}(k+1),x-5^{-n}k], a_{n,k}^{(i,-)}(y) < h < a_{n,k}^{(i,+)}(y)\}$. For $i\in\{1,2\}$, $n\in\mathds{N}^*$, we also denote $u_n^{(i)}=\inf\{y \geq -x \,|\,\exists\, k \in \mathds{N}, -y \in (x-5^{-n}(k+1),x-5^{-n}k], (-y,\Lambda_{(x,h_i)}^*(-y))\not\in A^{(i)}_{n,k}\}$ the first moment a process goes above the top of a box or below its bottom, and $v_n=\min\{-x+5^{-n}k \,|\, k\in\mathds{N}, A^{(1)}_{n,k} \cap A^{(2)}_{n,k} \neq \emptyset\}$ the first moment the boxes are not disjoint anymore. Finally, we set $Y=\inf\{y \geq -x \,|\, \Lambda_{(x,h_1)}^*(-y)=\Lambda_{(x,h_2)}^*(-y)\}$ the meeting time of the processes. In order to prove the backward lines are independent until they meet, we replace Equations (9.31) and (9.32) in the proof of \cite{Toth_et_al1998} (Pages 433 and 434) by the following Lemmas \ref{lem_cons_back1} and \ref{lem_cons_back2}. 
 
 \begin{lemma}\label{lem_cons_back1}
 For $i \in \{1,2\}$, almost surely $u_n^{(i)}$ tends to $+\infty$ when $n$ tends to $+\infty$.
 \end{lemma}
 
 \begin{lemma}\label{lem_cons_back2}
 Almost surely, $v_n$ tends to $Y$ when $n$ tends to $+\infty$ and $v_n \leq Y$ when $n$ is large enough.
 \end{lemma}
 
 \begin{proof}[Proof of Lemma \ref{lem_cons_back1}.]
Let $i \in \{1,2\}$. To prove this lemma, we need to show the fluctuations of $\Lambda_{(x,h_i)}^*$ are small enough that it stays in its boxes, which we will control with the help of estimates on the fluctuations of its driving Brownian motion. Let us introduce some notation. We already know $\Lambda_{(x,h_i)}^*(-.)$ is a $(\lambda(-.),-\chi)$-RAB starting from $(x,h)$; we can assume it is driven by a Brownian motion $(W_y)_{y \geq -x}$. For $n\in\mathds{N}^*$, $k\in\mathds{N}$, we define $\mathcal{A}_{i,n,k}=\{\forall\, y \in [-x+5^{-n}k,-x+5^{-n}(k+1)], |W_y-W_{-x+5^{-n}k}| \leq 2^{-n-2}\}$. 
 
 Let us prove that for $K \in \mathds{N}^*$, $\bigcap_{0 \leq k \leq 5^{n}K-1}\mathcal{A}_{i,n,k}$ implies $u_n^{(i)} \geq -x+K$. We assume $\bigcap_{0 \leq k \leq 5^{n}K-1}\mathcal{A}_{i,n,k}$. For all $0 \leq k \leq 5^{n}K-1$, we always have $(x-5^{-n}k,\Lambda_{(x,h_i)}^*(x-5^{-n}k))\in A^{(i)}_{n,k}$. Furthermore, for $y \in (-x+5^{-n}k,-x+5^{-n}(k+1))$, there are two possibilities. Either $\Lambda_{(x,h_i)}^*(-.)$ does not meet $\lambda(-.)$ in $[-x+5^{-n}k,y)$, which implies $|\Lambda_{(x,h_i)}^*(-y)-\Lambda_{(x,h_i)}^*(x-5^{-n}k)|=|W_{-y}-W_{-x+5^{-n}k}| \leq 2^{-n-2}$, thus in particular $(-y,\Lambda_{(x,h_i)}^*(-y))\in A^{(i)}_{n,k}$. Or $\Lambda_{(x,h_i)}^*(-.)$ has met $\lambda(-.)$ in $[-x+5^{-n}k,y)$, and we denote $y'=\inf\{z \in [-x+5^{-n}k,y)\,|\,\Lambda_{(x,h_i)}^*(-z)=\lambda(-z)\}$ the meeting time. By what we have just shown, $|\Lambda_{(x,h_i)}^*(-y')-\Lambda_{(x,h_i)}^*(x-5^{-n}k)| \leq 2^{-n-2}$, hence $|\lambda(-y')-\Lambda_{(x,h_i)}^*(x-5^{-n}k)| \leq 2^{-n-2}$, which yields $a_{n,k}^{(i,-)}(-y)=\lambda(-y)-2^{-n}$, therefore $\Lambda_{(x,h_i)}^*(-y) > a_{n,k}^{(i,-)}(-y)$. Moreover, $\Lambda_{(x,h_i)}^*(-y) \leq W_{y}+\sup_{y' \leq z \leq y}(\lambda(-z)-W_{z})=\sup_{y' \leq z \leq y}(\lambda(-z)+(W_{y}-W_{-x+5^{-n}k})-(W_z-W_{-x+5^{-n}k})) \leq \sup_{-y \leq z\leq x-5^{-n}k}\lambda(z)+2^{-n-1} < a_{n,k}^{(i,+)}(-y)$. We deduce $(-y,\Lambda_{(x,h_i)}^*(-y))\in A^{(i)}_{n,k}$. Consequently, $\bigcap_{0 \leq k \leq 5^{n}K-1}\mathcal{A}_{i,n,k}$ implies $u_n^{(i)} \geq -x+K$.
 
 From this we deduce $\sum_{n\in\mathds{N}^*}\mathds{P}_{\lambda,\chi}(u_n^{(i)} < -x+K) \leq \sum_{n\in\mathds{N}^*}\mathds{P}_{\lambda,\chi}(\bigcup_{0 \leq k \leq 5^{n}K-1}\mathcal{A}_{i,n,k}^c)$. If for all $K \in \mathds{N}^*$ the latter sum is finite, then the Borel-Cantelli lemma yields that for all $K \in \mathds{N}^*$, almost surely when $n$ is large enough, $u_n^{(i)} \geq -x+K$, which implies Lemma \ref{lem_cons_back1}. So we only have to prove that for all $K \in \mathds{N}^*$, we have $\sum_{n\in\mathds{N}^*}\mathds{P}_{\lambda,\chi}(\bigcup_{0 \leq k \leq 5^{n}K-1}\mathcal{A}_{i,n,k}^c)<+\infty$. This comes from the fact that for $n\in\mathds{N}^*$, $k \in \mathds{N}$, the estimate on the Brownian motion in Lemma \ref{lem_BM_max} yields $\mathds{P}_{\lambda,\chi}(\mathcal{A}_{i,n,k}^c) \leq 2\frac{\sqrt{2v5^{-n}}}{2^{-n-2}\sqrt{\pi}}\exp(-\frac{2^{-2n-4}}{2v5^{-n}})$, hence $\sum_{n\in\mathds{N}^*}\mathds{P}_{\lambda,\chi}(\bigcup_{0 \leq k \leq 5^{n}K-1}\mathcal{A}_{i,n,k}^c) \leq \sum_{n\in\mathds{N}^*}8K5^{n/2}2^n\sqrt{\frac{2v}{\pi}}\exp(-\frac{1}{2^5 v}(\frac{5}{4})^n)<+\infty$, which ends the proof of Lemma \ref{lem_cons_back1}.
 \end{proof}
 
 \begin{proof}[Proof of Lemma \ref{lem_cons_back2}]
The idea of the proof is that before the processes meet, they are separated by at least some small distance, while the boxes stay very close to their respective processes when $n$ is large, so they will be disjoint for large $n$. We first prove that almost surely, when $n$ is large enough, $v_n \leq Y$. If $Y = +\infty$, it holds true. If $Y<+\infty$, we notice that if the two processes are still in their boxes at $Y$, then these boxes are not disjoint. Moreover, by Lemma \ref{lem_cons_back1}, almost surely $u_n^{(1)}$ and $u_n^{(2)}$ tend to $+\infty$ when $n$ tends to $+\infty$, so when $n$ is large enough, $u_n^{(1)} > Y$ and $u_n^{(2)} > Y$, hence if $k$ is such that $Y \in [-x+5^{-n}k,-x+5^{-n}(k+1))$, then $(-Y,\Lambda_{(x,h_1)}^*(-Y))\in A^{(1)}_{n,k}$ and $(-Y,\Lambda_{(x,h_2)}^*(-Y))\in A^{(2)}_{n,k}$, hence $A^{(1)}_{n,k} \cap A^{(2)}_{n,k} \neq \emptyset$, thus $v_n \leq -x+5^{-n}k \leq Y$.
  
  We now prove the almost sure convergence of $v_n$ to $Y$. We assume $Y < +\infty$; let $\ell \in \mathds{N}^*$, it is enough to prove that $v_n \geq Y-\frac{1}{\ell}$ when $n$ is large enough (if $Y=+\infty$, we can use a similar argument to prove that $v_n \geq \ell$ when $n$ is large enough). We will first control the fluctuations of $\Lambda_{(x,h_1)}^*(-.)$, $\Lambda_{(x,h_2)}^*(-.)$ and $\lambda$. One can prove as in Lemma 9.1(iv) of \cite{Toth_et_al1998} that $\Lambda_{(x,h_1)}^*(-.)$, $\Lambda_{(x,h_2)}^*(-.)$ are almost surely continuous, which yields the existence of $\varepsilon>0$ (random) so that for all $y \in [-x,Y-\frac{1}{2\ell}]$, we have $|\Lambda_{(x,h_1)}^*(-y)-\Lambda_{(x,h_2)}^*(-y)| \geq \varepsilon$. Moreover, $\lambda(-.)$ is uniformly continuous on $[-x,Y-\frac{1}{2\ell}]$, so there exists $n_0 \in \mathds{N}$ so that for $y,z \in [-x,Y-\frac{1}{2\ell}]$, if $|y-z| \leq 5^{-n_0}$, then $|\lambda(-y)-\lambda(-z)| \leq \varepsilon/6$. We now choose $n \geq n_0$, and $k\in\mathds{N}$ so that $-x+5^{-n}(k+1) \leq Y-\frac{1}{2\ell}$. We are going to prove that $A^{(1)}_{n,k} \cap A^{(2)}_{n,k} = \emptyset$ if $n$ is large enough (uniformly in $k$).
  
  In order to prove $A^{(1)}_{n,k} \cap A^{(2)}_{n,k} = \emptyset$, we are going to prove the upper and lower frontiers of these boxes stay close to their respective processes. We set $i\in\{1,2\}$, $y \in [x-5^{-n}(k+1),x-5^{-n}k]$, and we study $|a_{n,k}^{(i,\pm)}(y)-\Lambda_{(x,h_i)}^*(x-5^{-n}k)|$. We begin with $a_{n,k}^{(i,+)}(y)=\max(\Lambda_{(x,h_i)}^*(x-5^{-n}k)+2^{-n},\sup_{y \leq z\leq x-5^{-n}k}\lambda(z)+2^{-n})$. If $a_{n,k}^{(i,+)}(y)=\Lambda_{(x,h_i)}^*(x-5^{-n}k)+2^{-n}$, we have $|a_{n,k}^{(i,+)}(y)-\Lambda_{(x,h_i)}^*(x-5^{-n}k)| \leq 2^{-n}$. Otherwise, we have $\lambda(x-5^{-n}k) \leq \Lambda_{(x,h_i)}^*(x-5^{-n}k) \leq \sup_{y \leq z\leq x-5^{-n}k}\lambda(z)$, with $|\sup_{y \leq z\leq x-5^{-n}k}\lambda(z)-\lambda(x-5^{-n}k)| \leq \varepsilon/6$, hence $|\sup_{y \leq z\leq x-5^{-n}k}\lambda(z)-\Lambda_{(x,h_i)}^*(x-5^{-n}k)| \leq \varepsilon/6$, thus $|a_{n,k}^{(i,+)}(y)-\Lambda_{(x,h_i)}^*(x-5^{-n}k)| \leq 2^{-n}+\varepsilon/6$. Therefore, in all cases $|a_{n,k}^{(i,+)}(y)-\Lambda_{(x,h_i)}^*(x-5^{-n}k)| \leq 2^{-n}+\varepsilon/6$. We now consider $a_{n,k}^{(i,-)}(y)$. The argument is written for $y \neq x-5^{-n}k$, but the case $y=x-5^{-n}k$ is similar and easier. If $\lambda(z) < \Lambda_{(x,h_i)}^*(x-5^{-n}k)-2^{-n}$ for all $z \in (y,x-5^{-n}k]$, we have $|a_{n,k}^{(i,-)}(y)-\Lambda_{(x,h_i)}^*(x-5^{-n}k)| = 2^{-n}$. Otherwise, there exists $z \in (y,x-5^{-n}k]$ so that $\lambda(z) \geq \Lambda_{(x,h_i)}^*(x-5^{-n}k)-2^{-n}$. We also have $\Lambda_{(x,h_i)}^*(x-5^{-n}k)-2^{-n} \geq \lambda(x-5^{-n}k)-2^{-n}$, and $|\lambda(z)-\lambda(x-5^{-n}k)| \leq \varepsilon/6$, hence $|\lambda(z)-\Lambda_{(x,h_i)}^*(x-5^{-n}k)| \leq \varepsilon/6+2^{-n}$. Furthermore, we can write $|a_{n,k}^{(i,-)}(y)-\Lambda_{(x,h_i)}^*(x-5^{-n}k)|=|\lambda(y)-2^{-n}-\Lambda_{(x,h_i)}^*(x-5^{-n}k)| \leq |\lambda(y)-\lambda(z)|+|\lambda(z)-\Lambda_{(x,h_i)}^*(x-5^{-n}k)|+2^{-n} \leq \varepsilon/6+\varepsilon/6+2^{-n}+2^{-n}=\varepsilon/3+2^{-n+1}$. Thus in all cases $|a_{n,k}^{(i,-)}(y)-\Lambda_{(x,h_i)}^*(x-5^{-n}k)| \leq \varepsilon/3+2^{-n+1}$. We deduce that for any $(y,h) \in A^{(i)}_{n,k}$ we have $|h-\Lambda_{(x,h_i)}^*(x-5^{-n}k)| \leq \varepsilon/3+2^{-n+1}$. In addition, by assumption $|\Lambda_{(x,h_1)}^*(x-5^{-n}k)-\Lambda_{(x,h_2)}^*(x-5^{-n}k)| \geq \varepsilon$. Consequently, if $n$ is large enough to have $2^{-n+2} < \varepsilon/3$, we obtain $A^{(1)}_{n,k} \cap A^{(2)}_{n,k} = \emptyset$. We conclude that if $n$ is large enough, for all $k\in\mathds{N}$ so that $-x+5^{-n}(k+1) \leq Y-\frac{1}{2\ell}$, we have $v_n \geq -x+5^{-n}(k+1)$, thus if $n$ is large enough, $v_n \geq Y-\frac{1}{\ell}$, which ends the proof of Lemma \ref{lem_cons_back2}. 
 \end{proof}
 
 \subsection{Coalescence of the backward lines}\label{sec_backward_coal}
 
 In this section, we prove the backward lines coalesce when they meet. The idea of \cite{Toth_et_al1998} to deal with this was to notice that if two backward lines do not coalesce when they meet, there is a forward line $\Lambda_{(\tilde x_n,\tilde h_n)}$ between them, hence the meeting point is on $\Lambda_{(\tilde x_n,\tilde h_n)}$. Moreover, there will be at least three separate forward lines coming out from the immediate vicinity of the meeting point: one below the lowest backward line of the two, one above the topmost backward line, and one between the two lines. Therefore one only has to prove there cannot be three such forward lines close to a $\Lambda_{(\tilde x_n,\tilde h_n)}$.
 
 However, there are two main problems in adapting the approach of \cite{Toth_et_al1998} to our setting. The first is that it requires a bound on the probability three forward lines with close starting points are separate. The estimate of \cite{Toth_et_al1998} relied on the fact their barrier was the abscissa axis, so we need to replace it. This is done in Lemma \ref{lem_cons_3BMbis}; however, in the part of the space where the Brownian motions are absorbed, the bound only works away from $\lambda$, which forces us to handle the proof more carefully than in \cite{Toth_et_al1998}. The second main problem is that the barrier may have very wild variations, hence $\Lambda_{(\tilde x_n,\tilde h_n)}$ may also vary wildly, and if it does, ``close to $\Lambda_{(\tilde x_n,\tilde h_n)}$'' may be in a zone too wide to obtain good bounds. In order to solve this problem, we had to separate the cases depending on the behavior of $\lambda$. 
 
We begin by formulating the exact statement we need to prove. For any $(x,h)\in\mathds{R}_\lambda^2$, $\eta>0$, we denote 
  \[
  O^\eta(x,h)=\lim_{\varepsilon \to 0}\mathrm{card}(\{\Lambda_{(x',h')}(x+\eta)\,|\,(x',h')\in(x,x+\varepsilon)\times(h-\varepsilon,h+\varepsilon)\})
  \]
  the ``number of forward lines coming from the immediate vicinity of $(x,h)$ which are still separate at $x+\eta$''. As in \cite{Toth_et_al1998}, in order to prove backward lines coalesce when they meet, it is enough to prove the following Proposition \ref{prop_cons_OIbis} (indeed, it allows to prove an equivalent of Equation (9.38) of \cite{Toth_et_al1998}, which yields the result as in the proof of their Equation (9.39)).
  
  \begin{proposition}\label{prop_cons_OIbis}
    For any $n\in\mathds{N}$, $\eta>0$, $M\in\mathds{N}^*$ fixed, almost surely for all $x\in[\tilde x_n,\chi+M)$ so that $\Lambda_{(\tilde x_n,\tilde h_n)}(x)>\lambda(x)$ we have $O^\eta(x,\Lambda_{(\tilde x_n,\tilde h_n)}(x)) \leq 2$.
  \end{proposition}
 
  Before proving Proposition \ref{prop_cons_OIbis}, we first state and prove a bound on the probability there exist three $\Lambda_{(x,h)}$ with their $h$ close to each other which remain separate up to $x+\varepsilon$. This bound is in the spirit of Lemma 9.4 of \cite{Toth_et_al1998}, but works only away from $\lambda$ in the part of the space with absorption. For any $(x,h)\in\mathds{R}^2$, $\varepsilon>0$, $\delta>0$, we denote 
 \[
  \mathcal{D}(x,h,h+\delta,\varepsilon) = \{\exists\, h_1,h_2,h_3\in[h,h+\delta], \Lambda_{(x,h_1)}(x+\varepsilon)<\Lambda_{(x,h_2)}(x+\varepsilon)<\Lambda_{(x,h_3)}(x+\varepsilon)\}.
 \]
If $h\leq\lambda(x)$, then $\Lambda_{(x,h')}$ is not defined for $h' \leq \lambda(x)$, and $h_1,h_2,h_3\in[h,h+\delta]$ must be understood as $h_1,h_2,h_3\in(\lambda(x),h+\delta]$.

 \begin{lemma}\label{lem_cons_3BMbis}
  There exists a constant $C$ so that for any $(x,h)\in\overline{\mathds{R}_\lambda^2}$, $\varepsilon>0$, $\delta>0$, we have the following:
 \begin{itemize}
  \item If $\lambda(x)<h$, $\mathds{P}_{\lambda,\chi}(\{\forall\,y \in [x,x+\varepsilon],\Lambda_{(x,h)}(y)>\lambda(y)\} \cap \mathcal{D}(x,h,h+\delta,\varepsilon)) \leq C(\frac{\delta}{\sqrt{\varepsilon}})^3$.
  \item If $x+\varepsilon \leq \chi$ then $\mathds{P}_{\lambda,\chi}(\mathcal{D}(x,h,h+\delta,\varepsilon)) \leq C(\frac{\delta}{\sqrt{\varepsilon}})^3$.
 \end{itemize}
 \end{lemma}
 
 \begin{proof}
 For the first point, the idea is that if the $\Lambda_{(x,h_i)}$ do not meet $\lambda$, they behave as Brownian motions hence we can use Brownian motion estimates. For the second point, the idea is that being reflected on $\lambda$ will push the lowest $\Lambda_{(x,h_i)}$ towards the top, hence will only make it more likely to encounter the other $\Lambda_{(x,h_i)}$. We will prove there exists a constant $C'$ so that for any $(x,h)\in\mathds{R}_\lambda^2$, $\varepsilon>0$, $\delta>0$, we have the following:
 \begin{enumerate}
  \item $\mathds{P}_{\lambda,\chi}(\forall\,y \in [x,x+\varepsilon],\Lambda_{(x,h)}(y)>\lambda(y),\Lambda_{(x,h)}(x+\varepsilon)<\Lambda_{(x,h+\delta)}(x+\varepsilon)<\Lambda_{(x,h+2\delta)}(x+\varepsilon)) \leq C'(\frac{\delta}{\sqrt{\varepsilon}})^3$.
  \item If $x+\varepsilon \leq \chi$ then $\mathds{P}_{\lambda,\chi}(\Lambda_{(x,h)}(x+\varepsilon)<\Lambda_{(x,h+\delta)}(x+\varepsilon)<\Lambda_{(x,h+2\delta)}(x+\varepsilon)) \leq C'(\frac{\delta}{\sqrt{\varepsilon}})^3$.
 \end{enumerate}
 Indeed, if we have the above, then Lemma \ref{lem_cons_3BMbis} can be obtained for $(x,h)\in\mathds{R}_\lambda^2$ through the proof of Lemma 9.4 of \cite{Toth_et_al1998}, by substituting the above for Lemma A.1 of \cite{Toth_et_al1998}. Moreover, if $x+\varepsilon \leq \chi$ and $h=\lambda(x)$, we obtain Lemma \ref{lem_cons_3BMbis} by writing $\mathcal{D}(x,h,h+\delta,\varepsilon)$ as the increasing union of the $\mathcal{D}(x,h+\frac{1}{n},h+\delta,\varepsilon)$, $n\in\mathds{N}^*$.
 
 We thus seek to prove (i) and (ii). To do this, we need some notation. Let $(x,h)\in\mathds{R}_\lambda^2$, $\varepsilon>0$, $\delta>0$. Let $(W_1(y))_{y \geq x}$, $(W_2(y))_{y \geq x}$, $(W_3(y))_{y \geq x}$ be independent Brownian motions so that $W_1(x)=h$, $W_2(x)=h+\delta$, $W_3(x)=h+2\delta$. Let $(R_1(y))_{y \geq x}$, $(R_2(y))_{y \geq x}$, $(R_3(y))_{y \geq x}$ be the $(\lambda,\chi)$-RABs starting from $(x,h),(x,h+\delta),(x,h+2\delta)$ and driven by $(W_1(y))_{y \geq x}$, $(W_2(y))_{y \geq x}$, $(W_3(y))_{y \geq x}$. Let $Y=\inf\{y \geq x \,|\, \exists\,i,j\in\{1,2,3\}, R_i(y)=R_j(y)\}$ the first time at which two of the $R_i$ meet. It is enough to prove that $\mathds{P}_{\lambda,\chi}(\forall\,y \in [x,x+\varepsilon],R_1(y)>\lambda(y),Y>x+\varepsilon) < C'(\frac{\delta}{\sqrt{\varepsilon}})^3$, and that if $x+\varepsilon \leq \chi$ then $\mathds{P}_{\lambda,\chi}(Y>x+\varepsilon) < C'(\frac{\delta}{\sqrt{\varepsilon}})^3$. 
 
 Let $Y'=\inf\{y \geq x \,|\, \exists\,i,j\in\{1,2,3\}, W_i(y)=W_j(y)\}$ is the first time at which two of the $W_i$ meet. If for all $y \in [x,x+\varepsilon]$ we have $R_1(y)>\lambda(y)$ and $Y>x+\varepsilon$, then none of the $R_i$, $i \in \{1,2,3\}$ meets $\lambda$ in $[x,x+\varepsilon]$, hence for all $i \in \{1,2,3\}$, $y\in[x,x+\varepsilon]$, we have $R_i(y)=W_i(y)$, which implies $Y'\geq x+\varepsilon$. Consequently $\mathds{P}_{\lambda,\chi}(\forall\,y \in [x,x+\varepsilon],R_1(y)>\lambda(y),Y>x+\varepsilon) \leq \mathds{P}_{\lambda,\chi}(Y'\geq x+\varepsilon)\leq \tilde C (\frac{\delta}{\sqrt{\varepsilon}})^3$ by Lemma \ref{lem_3BM}. We now consider the case $x+\varepsilon \leq \chi$. If there exists $y \in [x,x+\varepsilon]$ so that $R_2(y)=\lambda(y)$, then $R_2(y) \leq R_1(y)$, hence $Y \leq y \leq x+\varepsilon$. Therefore if $Y>x+\varepsilon$, for all $y \in [x,x+\varepsilon]$ we have $R_2(y)>\lambda(y)$, hence $R_2(y)=W_2(y)$, and we get $R_3(y)=W_3(y)$ by the same argument. In addition, for any $y \in [x,x+\varepsilon]$ we have $W_1(y) \leq R_1(y)$. We deduce that if $Y>x+\varepsilon$ then $Y'\geq x+\varepsilon$. This implies $\mathds{P}_{\lambda,\chi}(Y>x+\varepsilon) \leq \tilde C (\frac{\delta}{\sqrt{\varepsilon}})^3$ by Lemma \ref{lem_3BM}, which ends the proof of (i) and (ii), hence of Lemma \ref{lem_cons_3BMbis}.
 \end{proof}
  
We can now prove Proposition \ref{prop_cons_OIbis}.

  \begin{proof}[Proof of Proposition \ref{prop_cons_OIbis}.]
 Let $n\in\mathds{N}$, $\eta>0$, $M<+\infty$. We will consider the case $\tilde x_n < \chi$ as the case $\tilde x_n \geq \chi$ is similar and simpler. Let $m\in\mathds{N}$ large, with $2^{-m/3} < \eta/2$. As in \cite{Toth_et_al1998} (see their Equations (9.48) and (9.49)), we will cut the trajectory of $\Lambda_{(\tilde x_n,\tilde h_n)}$ into intervals $[x_{j-1},x_j)$ of length much smaller than $\eta$, and we will construct ``bad events'' with small probability so that if there exists $x\in [x_{j-1},x_{j})$ so that $O^\eta(x,\Lambda_{(\tilde x_n,\tilde h_n)}(x)) > 2$, a bad event occurs. In \cite{Toth_et_al1998}, the bad event for $x\in [x_{j-1},x_{j})$ was of the form ``$\Lambda_{(x_{j-1},\Lambda_{(\tilde x_n,\tilde h_n)}(x_{j-1})+2^{-m})}(x_j) \geq \Lambda_{(\tilde x_n,\tilde h_n)}(x_{j-1})+2\cdot 2^{-m}$, or $\Lambda_{(x_{j-1},\Lambda_{(\tilde x_n,\tilde h_n)}(x_{j-1})-2^{-m})}(x_j) \leq \Lambda_{(\tilde x_n,\tilde h_n)}(x_{j-1})-2\cdot 2^{-m}$, or $\mathcal{D}(x_j,\Lambda_{(\tilde x_n,\tilde h_n)}(x_{j-1})-2\cdot 2^{-m},\Lambda_{(\tilde x_n,\tilde h_n)}(x_{j-1})+2\cdot 2^{-m},\eta/2)$''. Indeed, if there exists $x\in [x_{j-1},x_{j})$ so that $O^\eta(x,\Lambda_{(\tilde x_n,\tilde h_n)}(x)) > 2$, there are three $\Lambda_{(x',h')}$ with starting points close to $(x,\Lambda_{(\tilde x_n,\tilde h_n)}(x))$ that do not coalesce before $x+\eta$. Then they start between $\Lambda_{(x_{j-1},\Lambda_{(\tilde x_n,\tilde h_n)}(x_{j-1})-2^{-m})}$ and $\Lambda_{(x_{j-1},\Lambda_{(\tilde x_n,\tilde h_n)}(x_{j-1})+2^{-m})}$, so their values at $x_j$ are between $\Lambda_{(x_{j-1},\Lambda_{(\tilde x_n,\tilde h_n)}(x_{j-1})-2^{-m})}(x_j)$ and $\Lambda_{(x_{j-1},\Lambda_{(\tilde x_n,\tilde h_n)}(x_{j-1})+2^{-m})}(x_j)$, which implies the bad event (since in their case, at the right of $x_j$, almost surely $\Lambda_{(x',h')}=\Lambda_{(x_j,\Lambda_{(x',h')}(x_j))}$). 
 
 Our bad events are notably more complex. Firstly, the estimate we have on the probability of $\mathcal{D}(x_j,\Lambda_{(\tilde x_n,\tilde h_n)}(x_{j-1})-2\cdot 2^{-m},\Lambda_{(\tilde x_n,\tilde h_n)}(x_{j-1})+2\cdot 2^{-m},\eta/2)$ (Lemma \ref{lem_cons_3BMbis}) depends on whether we are in the part of the space where Brownian motions are reflected or absorbed, so we have to distinguish these two cases. In the part of the space with absorption, the estimate of Lemma \ref{lem_cons_3BMbis} is valid only away from $\lambda$, hence we work in the part of the trajectory of $\Lambda_{(\tilde x_n,\tilde h_n)}$ which is at distance $2^{-k}$ from $\lambda$, then make $k$ tend to $+\infty$. In the part of the space with reflection, if $\Lambda_{(\tilde x_n,\tilde h_n)}(x_{j-1})$ is too close to $\lambda(x_{j-1})$, then $\Lambda_{(x_j,\Lambda_{(\tilde x_n,\tilde h_n)}(x_{j-1})-2^{-m})}(x_j)$ may not be defined. \cite{Toth_et_al1998} does not seem to notice the fact, but in their case all the $\Lambda_{(x',h')}$ are above the abscissa axis, which allows to control their value at $x_j$ from below. Since our $\lambda$ may have a sharp decrease in $[x_{j-1},x_j)$, the control it would give us is not sufficient. We will thus separate into three cases: being away from $\lambda$, in which case we work as in \cite{Toth_et_al1998} (Page 437), being close to $\lambda$ when $\lambda$ has no sharp decrease, in which case saying the $\Lambda_{(x',h')}$ are above $\lambda$ at $x_j$ is sufficient, and being close to $\lambda$ when $\lambda$ has a sharp decrease. In this case, we consider (roughly) the $\Lambda_{(\tilde x,\tilde h)}$ starting from the point at which $\lambda$ starts to decrease too much; then the $\Lambda_{(x',h')}$ will be above $\Lambda_{(\tilde x,\tilde h)}$.

 Before defining our bad events, we define the following ``good events'', which will occur almost surely and help us prove the bad events occur. 
 \[
 \mathcal{G}_{r,m} = \bigcap_{n'\in\mathds{N}}\bigcap_{x\in \tilde x_n+5^{-m}\mathds{N}, \tilde x_{n'} \leq x<\chi}\{\text{if }\Lambda_{(\tilde x_{n'},\tilde h_{n'})}(x)>\lambda(x) \text{ then }\forall\,y\geq x, \Lambda_{(\tilde x_{n'},\tilde h_{n'})}(y)=\Lambda_{(x,\Lambda_{(\tilde x_{n'},\tilde h_{n'})}(x))}(y)\} \cap \qquad\qquad
 \]
 \[
  \qquad\qquad\qquad\qquad\qquad \{\text{if }\Lambda_{(\tilde x_{n'},\tilde h_{n'})}(x)=\lambda(x) \text{ then }\forall\, \varepsilon>0, \exists\, h>\lambda(x), \forall\,y\geq x+\varepsilon, \Lambda_{(\tilde x_{n'},\tilde h_{n'})}(y)=\Lambda_{(x,h)}(y)\},
 \]
 \[
 \mathcal{G}_{a,m} = \bigcap_{n'\in\mathds{N}}\bigcap_{x\in \chi+5^{-m}\mathds{N}, \tilde x_{n'} \leq x}\{\text{if }\Lambda_{(\tilde x_{n'},\tilde h_{n'})}(x)>\lambda(x) \text{ then }\forall\,y\geq x, \Lambda_{(\tilde x_{n'},\tilde h_{n'})}(y)=\Lambda_{(x,\Lambda_{(\tilde x_{n'},\tilde h_{n'})}(x))}(y)\}. 
 \]
 \begin{lemma}\label{lem_cons_continue}
  $\mathcal{G}_{r,m}$ and $\mathcal{G}_{a,m}$ occur almost surely. 
 \end{lemma}

 \begin{proof}
 Let $n'\in\mathds{N}$, $x\in (\tilde x_n+5^{-m}\mathds{N})\cup(\chi+5^{-m}\mathds{N})$ with $\tilde x_{n'} \leq x$. The fact that almost surely, if $\Lambda_{(\tilde x_{n'},\tilde h_{n'})}(x)>\lambda(x)$ then for all $y\geq x$, $\Lambda_{(\tilde x_{n'},\tilde h_{n'})}(y)=\Lambda_{(x,\Lambda_{(\tilde x_{n'},\tilde h_{n'})}(x))}(y)$ can be proven as in Lemma 8.3 of \cite{Toth_et_al1998}, by ``squeezing'' $\Lambda_{(x,\Lambda_{(\tilde x_{n'},\tilde h_{n'})}(x))}$ between two sequences $\Lambda_{(x,h(n))}$, $\Lambda_{(x,h'(n))}$ with $h(n)< \Lambda_{(\tilde x_{n'},\tilde h_{n'})}(x) < h'(n)$ so that $\Lambda_{(x,h(n))}$ and $\Lambda_{(x,h'(n))}$ coalesce quickly after $x$; to estimate the probability $\Lambda_{(x,h(n))}$ and $\Lambda_{(x,h'(n))}$ coalesce, we can use an argument similar to that of Lemma \ref{lem_cons_coal}. If we also have $x<\chi$, $\{$if $\Lambda_{(\tilde x_{n'},\tilde h_{n'})}(x)=\lambda(x)$ then $\forall\, \varepsilon>0, \exists\, h>\lambda(x), \forall\,y\geq x+\varepsilon, \Lambda_{(\tilde x_{n'},\tilde h_{n'})}(y)=\Lambda_{(x,h)}(y)\}$ happens almost surely thanks to Lemma \ref{lem_cons_back_Markov} (using a sequence of $\varepsilon$ tending to 0), which ends the proof of the lemma. 
\end{proof}

As we already mentioned, the proof of Proposition \ref{prop_cons_OIbis} depends on whether we are in the part of the space where the Brownian motions are reflected, at the left of $\chi$, or in the part of the space where they are absorbed, at the right of $\chi$. We begin with the latter, as the proof there is simpler.

\emph{Part of the space with absorption.}

 We introduce bad events so that if there exists $x\in[\chi,\chi+M)$ so that $\Lambda_{(\tilde x_n,\tilde h_n)}(x)>\lambda(x)$ and $O^\eta(x,\Lambda_{(\tilde x_n,\tilde h_n)}(x)) > 2$, then one of these bad events occurs. Let $k$ large and $Y_k=\inf\{x \geq \chi\,|\,\min_{y \in [x,x+2^{-k}]}|\Lambda_{(\tilde x_{n},\tilde h_{n})}(x)-\lambda(y)| \leq 2^{-k}\}$. The following quantities and events will depend on $m$ and sometimes $k$, but we do not reflect this in the notation to make it lighter. For all $j \in \{0,...,\lfloor 5^mM\rfloor\}$, we denote $x_{j}=\chi+5^{-m}j$, $h_{j}=\Lambda_{(\tilde x_n,\tilde h_n)}(x_{j})$. For $j \in \{1,...,\lfloor 5^mM\rfloor\}$, the bad event $\mathcal{A}_{j}$ is defined as follows.
 \[
  \mathcal{A}_{j}=\{x_{j-1} < Y_k\} \cap \left(\mathcal{D}(x_{j},h_{j-1}-2^{-m+1},h_{j-1}+2^{-m+1},2^{-\frac{m}{3}}) \cup \right.
 \]
\[
  \{\Lambda_{(x_{j-1},h_{j-1}+2^{-m})}(x_{j})\geq h_{j-1}+2^{-m+1}\}\cup\{\Lambda_{(\tilde x_n,\tilde h_n)}(x_{j})=\Lambda_{(x_{j-1},h_{j-1}+2^{-m})}(x_{j})\}
\]
\[
 \left.\cup \{\Lambda_{(x_{j-1},h_{j-1}-2^{-m})}(x_{j})\leq h_{j-1}-2^{-m+1}\}\cup\{\Lambda_{(\tilde x_n,\tilde h_n)}(x_{j})=\Lambda_{(x_{j-1},h_{j-1}-2^{-m})}(x_{j})\}\right).
\]
The events $\{\Lambda_{(\tilde x_n,\tilde h_n)}(x_{j})=\Lambda_{(x_{j-1},h_{j-1}+2^{-m})}(x_{j})\}$ and $\{\Lambda_{(\tilde x_n,\tilde h_n)}(x_{j})=\Lambda_{(x_{j-1},h_{j-1}-2^{-m})}(x_{j})\}$ are not present in \cite{Toth_et_al1998}, but they seem necessary, as if one of them happens, the value of $\Lambda_{(x',h')}$ with $x'$ at the right of the coalescence is not controlled by $\Lambda_{(x_{j-1},h_{j-1}-2^{-m})}(x_{j})$. The event $\{x_{j-1} < Y_k\}$ is there to ensure we are not too close to $\lambda$, where Lemma \ref{lem_cons_3BMbis} fails. We have the following lemmas.
  
\begin{lemma}\label{lem_cons_OI_bad_evts_good_a}
 When $m$ is large enough, for $j \in \{1,...,\lfloor 5^mM\rfloor\}$, almost surely if $x_{j-1} < Y_k$ and there exists $x\in [x_{j-1},x_{j})$ so that $O^\eta(x,\Lambda_{(\tilde x_n,\tilde h_n)}(x)) > 2$, then $\mathcal{A}_{j}$ occurs.
\end{lemma}

 \begin{lemma}\label{lem_cons_OI_prob_bad_evts_a}
  When $m$ is large enough, for any $j \in \{1,...,\lfloor 5^mM\rfloor\}$, $\mathds{P}_{\lambda,\chi}(\mathcal{A}_{j}) \leq 6\exp(-\frac{5^m}{2^{2m+5}v})+\exp(-\frac{1}{v}2^{\frac{m}{3}-2k-5})+C 2^{-\frac{17}{6}m+6}$.
 \end{lemma}
 
 Lemmas \ref{lem_cons_OI_bad_evts_good_a} and \ref{lem_cons_OI_prob_bad_evts_a} are enough to prove that almost surely for all $x\in[\chi, \chi+M)$ so that $\Lambda_{(\tilde x_n,\tilde h_n)}(x)>\lambda(x)$, we have $O^\eta(x,\Lambda_{(\tilde x_n,\tilde h_n)}(x)) \leq 2$, which is one half of Proposition \ref{prop_cons_OIbis}. Indeed, when $m$ is large, Lemma \ref{lem_cons_OI_bad_evts_good_a} yields that almost surely, if there exists $x\in[\chi, \min(Y_k,\chi+M-5^{-m}))$ so that $O^\eta(x,\Lambda_{(\tilde x_n,\tilde h_n)}(x)) > 2$, there exists $j \in \{1,...,\lfloor 5^mM\rfloor\}$ so that $\mathcal{A}_{j}$ occurs. Hence almost surely $\{\exists\,x\in[\chi, \min(Y_k,\chi+M)),O^\eta(x,\Lambda_{(\tilde x_n,\tilde h_n)}(x)) > 2\} \subset \bigcup_{m'} \bigcap_{m \geq m'}\bigcup_{j \in \{1,...,\lfloor 5^mM\rfloor\}}\mathcal{A}_{j}$. Moreover, Lemma \ref{lem_cons_OI_prob_bad_evts_a} implies that we have $\mathds{P}_{\lambda,\chi}(\bigcup_{m'} \bigcap_{m \geq m'}\bigcup_{j \in \{1,...,\lfloor 5^{m}M\rfloor\}}\mathcal{A}_{j})=\lim_{m' \to+\infty}\mathds{P}_{\lambda,\chi}(\bigcap_{m \geq m'}\bigcup_{j \in \{1,...,\lfloor 5^{m}M\rfloor\}}\mathcal{A}_{j}) \leq \lim_{m' \to+\infty} (5^{m'}M(6\exp(-\frac{5^{m'}}{2^{2m'+5}v})+\exp(-\frac{1}{v}2^{\frac{m'}{3}-2k-5})+C 2^{-\frac{17}{6}m'+6}))=0$. Therefore almost surely for all $x\in[\chi, \min(Y_k,\chi+M))$, $O^\eta(x,\Lambda_{(\tilde x_n,\tilde h_n)}(x)) \leq 2$. By making $k$ tend to $+\infty$ we obtain that almost surely for all $x\in[\chi, \chi+M)$ so that $\Lambda_{(\tilde x_n,\tilde h_n)}(x)>\lambda(x)$, we have $O^\eta(x,\Lambda_{(\tilde x_n,\tilde h_n)}(x)) \leq 2$. We now prove lemmas \ref{lem_cons_OI_bad_evts_good_a} and \ref{lem_cons_OI_prob_bad_evts_a}. 

\begin{proof}[Proof of Lemma \ref{lem_cons_OI_bad_evts_good_a}.]
 This can be seen easily, so we only give a sketch. Let $m$ much larger than $k$, $j \in \{1,...,\lfloor 5^mM\rfloor\}$ so that $x_{j-1} < Y_k$ and there exists $x\in [x_{j-1},x_{j})$ so that $O^\eta(x,\Lambda_{(\tilde x_n,\tilde h_n)}(x)) > 2$. Then for $\varepsilon$ small there exist three $\Lambda_{(x',h')}$ with $(x',h')\in(x,x+\varepsilon)\times(h-\varepsilon,h+\varepsilon)$ that do not merge before $x'+\eta$. If $\Lambda_{(\tilde x_n,\tilde h_n)}(x_{j}) \neq \Lambda_{(x_{j-1},h_{j-1}+2^{-m})}(x_{j})$, then $\Lambda_{(\tilde x_n,\tilde h_n)}$ is strictly below $\Lambda_{(x_{j-1},h_{j-1}+2^{-m})}$ around $x$, hence the $\Lambda_{(x',h')}$ are below $\Lambda_{(x_{j-1},h_{j-1}+2^{-m})}$. Therefore, if $\Lambda_{(x_{j-1},h_{j-1}+2^{-m})}(x_{j})< h_{j-1}+2^{-m+1}$, the $\Lambda_{(x',h')}(x_{j})$ are smaller than $h_{j-1}+2^{-m+1}$. In the same way, if $\Lambda_{(x_{j-1},h_{j-1}-2^{-m})}(x_{j})> h_{j-1}-2^{-m+1}$ and $\Lambda_{(\tilde x_n,\tilde h_n)}(x_{j})\neq\Lambda_{(x_{j-1},h_{j-1}-2^{-m})}(x_{j})$, the $\Lambda_{(x',h')}(x_{j})$ are bigger than $h_{j-1}-2^{-m+1}$. Moreover, we have the third point of Theorem \ref{thm_forward_bis}, by Lemma \ref{lem_cons_continue} we have that almost surely $\mathcal{G}_{a,m}$ occurs, and by assumption $x_{j-1} < Y_k$, therefore almost surely $\mathcal{D}(x_{j},h_{j-1}-2^{-m+1},h_{j-1}+2^{-m+1},2^{-\frac{m}{3}})$ occurs.
 \end{proof}

 \begin{proof}[Proof of Lemma \ref{lem_cons_OI_prob_bad_evts_a}.]
 The estimate on the probability of the bad event $\mathcal{A}_j$ will come from the fact that if it occurs, either a ``$\mathcal{D}$ event'' occurs, which can be dealt with thanks to Lemma \ref{lem_cons_3BMbis}, or the forward lines have large fluctuations, which can be treated using Lemma \ref{lem_RAB_max}. We first introduce an additional notation. If $\mathcal{G}_{a,m}$ occurs (which happens almost surely thanks to Lemma \ref{lem_cons_continue}), we can write $\mathcal{A}_{j}\subset \{\min_{y \in [x_{j-1},x_{j-1}+2^{-k}]}|h_{j-1}-\lambda(y)| > 2^{-k}\} \cap \mathcal{A}_j(h_{j-1})$, where for any $h>\lambda(x_{j-1})$ so that $\min_{y \in [x_{j-1},x_{j-1}+2^{-k}]}|h-\lambda(y)| > 2^{-k}$, we denote $\mathcal{A}_j(h)=\mathcal{D}(x_{j},h-2^{-m+1},h+2^{-m+1},2^{-\frac{m}{3}}) \cup \{\Lambda_{(x_{j-1},h+2^{-m})}(x_{j})\geq h+2^{-m+1}\}\cup\{\Lambda_{(x_{j-1},h)}(x_{j})=\Lambda_{(x_{j-1},h+2^{-m})}(x_{j})\} \cup \{\Lambda_{(x_{j-1},h-2^{-m})}(x_{j})\leq h-2^{-m+1}\}\cup\{\Lambda_{(x_{j-1},h)}(x_{j})=\Lambda_{(x_{j-1},h-2^{-m})}(x_{j})\}$. We will prove that for any $h>\lambda(x_{j-1})$ so that $\min_{y \in [x_{j-1},x_{j-1}+2^{-k}]}|h-\lambda(y)| > 2^{-k}$, $\mathcal{A}_j(h)$ is independent from $h_{j-1}$. Indeed, if we construct a sequence $(x_{n'},h_{n'})$, $n' \in \mathds{N}$ so that $(x_0,h_0)=(\tilde x_n,\tilde h_n)$, for all $n' \geq 1$ odd, $x_{n'}=x_{j-1}$, for all $n' \geq 1$ even, $x_{n'}=x_{j}$, the set $\{h_{n'} \,|\,n'\geq 1$ odd$\}$ is dense in $[\lambda(x_{j-1}),+\infty)$ and $\{h_{n'} \,|\,n'\geq 1$ even$\}$ is dense in $[\lambda(x_{j}),+\infty)$, then for any $N < +\infty$, the family $(\Lambda_{(x_{n'},h_{n'})})_{0 \leq n' \leq N}$ is a $(\lambda,\chi)$-FICRAB. Consequently $(\Lambda_{(\tilde x_n,\tilde h_n)}(y))_{\tilde x_n \leq y \leq x_{j-1}}$ is independent from the $((\Lambda_{(x_{j-1},h_{n'})})_{n' \geq 1\text{ odd}},(\Lambda_{(x_{j},h_{n'})})_{n' \geq 1\text{ even}})$, thus from the $((\Lambda_{(x_{j-1},h')})_{h'>\lambda(x_{j-1})},(\Lambda_{(x_{j},h')})_{h'>\lambda(x_{j})})$ thanks to the left-continuity of the process. This allows to show $\mathcal{A}_j(h)$ is independent from $h_{j-1}$. Hence to prove Lemma \ref{lem_cons_OI_prob_bad_evts_a}, it is enough to prove that for any $h>\lambda(x_{j-1})$ so that $\min_{y \in [x_{j-1},x_{j-1}+2^{-k}]}|h-\lambda(y)| > 2^{-k}$, we have $\mathds{P}_{\lambda,\chi}(\mathcal{A}_j(h)) \leq 6\exp(-\frac{5^m}{2^{2m+5}v})+\exp(-\frac{1}{v}2^{\frac{m}{3}-2k-5})+C 2^{-\frac{17}{6}m+6}$. 
 
 We now study $\mathds{P}_{\lambda,\chi}(\mathcal{A}_j(h))$ for $h>\lambda(x_{j-1})$ with $\min_{y \in [x_{j-1},x_{j-1}+2^{-k}]}|h-\lambda(y)| > 2^{-k}$. By Lemma \ref{lem_RAB_max}, we have $\mathds{P}_{\lambda,\chi}(\Lambda_{(x_{j-1},h+2^{-m})}(x_{j})\geq h+2^{-m+1}) \leq \mathds{P}_{\lambda,\chi}(\exists\, x\in[x_{j-1},x_j], |\Lambda_{(x_{j-1},h+2^{-m})}(x)-\Lambda_{(x_{j-1},h+2^{-m})}(x_{j-1})| \geq 2^{-m}) \leq \exp(-\frac{5^m}{2^{2m+1}v})$ when $m$ is large enough, and similarly $\mathds{P}_{\lambda,\chi}(\Lambda_{(x_{j-1},h-2^{-m})}(x_{j})\leq h-2^{-m+1}) \leq \exp(-\frac{5^m}{2^{2m+1}v})$ when $m$ is large enough. Furthermore, $\Lambda_{(x_{j-1},h)}(x_{j})<\Lambda_{(x_{j-1},h+2^{-m})}(x_{j})$ if we have $\max_{x \in [x_{j-1},x_{j}]}|\Lambda_{(x_{j-1},h)}(x)-\Lambda_{(x_{j-1},h)}(x_{j-1})| \leq 2^{-m-2}$ and $\max_{x \in [x_{j-1},x_{j}]}|\Lambda_{(x_{j-1},h+2^{-m})}(x)-\Lambda_{(x_{j-1},h+2^{-m})}(x_{j-1})| \leq 2^{-m-2}$. By Lemma \ref{lem_RAB_max}, the probability that one of these maxima is bigger than $2^{-m-2}$ is smaller than $\exp(-\frac{5^m}{2^{2m+5}v})$ when $m$ is large enough, therefore $\mathds{P}_{\lambda,\chi}(\Lambda_{(x_{j-1},h)}(x_{j})=\Lambda_{(x_{j-1},h+2^{-m})}(x_{j})) \leq 2\exp(-\frac{5^m}{2^{2m+5}v})$ when $m$ is large enough. The same arguments yield that $\mathds{P}_{\lambda,\chi}(\Lambda_{(x_{j-1},h)}(x_{j})=\Lambda_{(x_{j-1},h-2^{-m})}(x_{j})) \leq 2\exp(-\frac{5^m}{2^{2m+5}v})$ when $m$ is large enough. Moreover, when $m$ is large enough, we obtain $\mathds{P}_{\lambda,\chi}(\exists\, x \in [x_{j},x_{j}+2^{-\frac{m}{3}}], \Lambda_{(x_{j},h-2^{-m+1})}(x)=\lambda(x)) \leq \mathds{P}_{\lambda,\chi}(\exists\, x \in [x_{j},x_{j}+2^{-\frac{m}{3}}], |\Lambda_{(x_{j},h-2^{-m+1})}(x)-\Lambda_{(x_{j},h-2^{-m+1})}(x_{j})| \geq 2^{-k-2}) \leq \exp(-\frac{1}{v}2^{\frac{m}{3}-2k-5})$ by Lemma \ref{lem_RAB_max}. Finally, Lemma \ref{lem_cons_3BMbis} yields $\mathds{P}_{\lambda,\chi}(\{\forall\, x \in [x_{j},x_{j}+2^{-\frac{m}{3}}], \Lambda_{(x_{j},h-2^{-m+1})}(x)>\lambda(x)\} \cap \mathcal{D}(x_{j},h-2^{-m+1},h+2^{-m+1},2^{-\frac{m}{3}})\}) \leq C(2^{-m+2}/\sqrt{2^{-\frac{m}{3}}})^3=C 2^{-\frac{17}{6}m+6}$. This is enough to prove Lemma \ref{lem_cons_OI_prob_bad_evts_a}.
 \end{proof}
 
 \emph{Part of the space with reflection.}
 
 We will define bad events so that if there exists $x\in [\tilde x_n,\chi)$ so that $O^\eta(x,\Lambda_{(\tilde x_n,\tilde h_n)}(x)) > 2$, one of these bad events occurs. We will reuse some notations already used in the part of the space with absorption, but with a different meaning, as using different notations would make them much heavier and harder to read. For any $j\in\{0,...,\lfloor5^m(\chi-2^{-m/3}-\tilde x_n)\rfloor\}$, we denote $x_{j}=\tilde x_n+5^{-m}j$, $h_{j}=\Lambda_{(\tilde x_n,\tilde h_n)}(x_{j})$. For $j\in\{1,...,\lfloor5^m(\chi-2^{-m/3}-\tilde x_n)\rfloor\}$, we denote $m_{j}=\max_{x\in[x_{j-1},x_{j}]}\lambda(x)$ and define $\mathcal{M}_{j}^1=\{h_{j-1}-2^{-m+1} > m_{j}\}$. On this event, we are ``away from $\lambda$'' and can use arguments similar to those used in the part of the space with absorption. We now denote $\mathcal{M}_{j}^2=(\mathcal{M}_{j}^1)^c \cap \{\lambda(x_{j})\geq m_{j}-2^{-m}\}$. On this event, we are close to $\lambda$, but $\lambda$ does not go down too much after its maximum, so $\lambda$ can bound from below the forward lines starting close to $\Lambda_{(\tilde x_n,\tilde h_n)}$. We also define $\mathcal{M}_{j}^3=(\mathcal{M}_{j}^1)^c \cap \{\lambda(x_{j})< m_{j}-2^{-m}\}$, in which case the lower bound must be different. Let us denote $x_{j}'=\sup\{x \in [x_{j-1},x_{j}] \,|\, \lambda(x)=m_{j}\}$, $x_{j}''=\sup\{x \leq x_{j}\,|\,\lambda(x) = m_{j}-2^{-m}\}$, and $h_{j}'=\max(h_{j-1}+2^{-m},m_{j}+2^{-m+1})$. On the event $\mathcal{M}_{j}^3$, the lower bound we would like to use for the forward lines starting close to $\Lambda_{(\tilde x_n,\tilde h_n)}$ is the forward line starting from $(x_j'',\lambda(x_j''))$, but it is not defined, so instead we use a $\Lambda_{(\tilde x_{n'},\tilde h_{n'})}$ so that $\Lambda_{(\tilde x_{n'},\tilde h_{n'})}(x_{j}'')=\lambda(x_{j}'')$, if there is one, or $\inf_{\ell \in \mathds{N}^*}\Lambda_{(x_{j}'',\lambda(x_{j}'')+1/\ell)}$ otherwise. Our bad event $\mathcal{A}_{j}$ is constructed as $\mathcal{A}_{j}=\mathcal{A}_{j}^0 \cup(\mathcal{M}_{j}^1 \cap \mathcal{A}_{j}^1)\cup(\mathcal{M}_{j}^2 \cap \mathcal{A}_{j}^2)\cup(\mathcal{M}_{j}^3 \cap \mathcal{A}_{j}^3)$ with
 \[
  \mathcal{A}_{j}^0 = \{\Lambda_{(x_{j-1},h_{j}')}(x_{j})\geq h_{j}'+2^{-m}\}\cup\{\Lambda_{(\tilde x_n,\tilde h_n)}(x_{j})=\Lambda_{(x_{j-1},h_{j}')}(x_{j})\},
 \]
 \[
  \mathcal{A}_{j}^1 \!=\! \{\Lambda_{(x_{j-1},h_{j-1}-2^{-m})}(x_{j})\!\leq\! h_{j-1}-2^{-m+1}\}\cup\{\Lambda_{(\tilde x_n,\tilde h_n)}(x_{j})\!=\!\Lambda_{(x_{j-1},h_{j-1}-2^{-m})}(x_{j})\} 
  \cup \mathcal{D}(x_{j},h_{j-1}-2^{-m+1}\!,h_{j}'+2^{-m}\!,2^{-\frac{m}{3}}),
 \]
 \[
  \mathcal{A}_{j}^2 = \mathcal{D}(x_{j},\lambda(x_{j}),h_{j}'+2^{-m},2^{-\frac{m}{3}}),
 \]
 \[
  \mathcal{A}_{j}^3 = \{\Lambda_{(\tilde x_n,\tilde h_n)}(x_{j}'')=\lambda(x_{j}'')\}\cup \{\exists\, n'\in\mathds{N}, \Lambda_{(\tilde x_{n'},\tilde h_{n'})}(x_{j}'')=\lambda(x_{j}''), \Lambda_{(\tilde x_{n'},\tilde h_{n'})}(x_{j})\leq m_{j}-2^{-m+1}\} \cup
  \]
  \[
  \left\{\inf_{\ell \in \mathds{N}^*}\!\Lambda_{(x_{j}'',\lambda(x_{j}'')+1/\ell)}(x_{j}) \leq m_{j}-2^{-m+1}\text{ or }\Lambda_{(\tilde x_n,\tilde h_n)}(x_{j}) = \!\inf_{\ell \in \mathds{N}^*}\!\Lambda_{(x_{j}'',\lambda(x_{j}'')+1/\ell)}(x_{j}) \right\}  
   \cup \mathcal{D}(x_{j},m_{j}-2^{-m+1},h_{j}'+2^{-m},2^{-\frac{m}{3}}).
  \]
   $\mathcal{A}_{j}^0$ is involved with the upper bound on forward lines starting near $\Lambda_{(\tilde x_n,\tilde h_n)}$, the rest of the bad event with the lower bound. The last parameter in the events $\mathcal{D}$ is $2^{-\frac{m}{3}}$ instead of $\eta/2$ as in \cite{Toth_et_al1998} because the estimate on $\mathds{P}(\mathcal{D}(x,h,h+\delta,\varepsilon))$ given in Lemma \ref{lem_cons_3BMbis} which holds for $\lambda(x)=h$ is only valid for $x+\varepsilon \leq \chi$, so we can work only with $x_{j}+\varepsilon \leq \chi$. Choosing $\varepsilon=2^{-\frac{m}{3}}$ allows to make $x_{j}$ tend to $\chi$ when $m$ tends to $+\infty$. 
   
   Our bad events will satisfy the following Lemmas \ref{lem_cons_OI_bad_evts_good} and \ref{lem_cons_OI_prob_bad_evts}. A proof similar to the one used in the part of the space with absorption shows that if Lemmas \ref{lem_cons_OI_bad_evts_good} and \ref{lem_cons_OI_prob_bad_evts} hold, almost surely for any $x \in [\tilde x_n,\chi)$ we have $O^\eta(x,\Lambda_{(\tilde x_n,\tilde h_n)}(x)) \leq 2$, which is the second half of Proposition \ref{prop_cons_OIbis}. Therefore we only have to prove these two lemmas to end the proof of Proposition \ref{prop_cons_OIbis}.  

 \begin{lemma}\label{lem_cons_OI_bad_evts_good}
   Almost surely, for any $j\in\{1,...,\lfloor5^m(\chi-2^{-m/3}-\tilde x_n)\rfloor\}$, if there exists $x\in[x_{j-1},x_{j})$ so that we have $O^\eta(x,\Lambda_{(\tilde x_n,\tilde h_n)}(x)) > 2$, then $\mathcal{A}_{j}$ occurs.
 \end{lemma}
 
 \begin{lemma}\label{lem_cons_OI_prob_bad_evts}
  When $m$ is large enough, for any $j\in\{1,...,\lfloor5^m(\chi-2^{-m/3}-\tilde x_n)\rfloor\}$, we have $\mathds{P}_{\lambda,\chi}(\mathcal{A}_{j}) \leq 10\exp(-\frac{5^m}{2^{2m+5}v})+C 2^{10}2^{-\frac{5}{2}m}$. 
 \end{lemma}
 
 \begin{proof}[Proof of Lemma \ref{lem_cons_OI_bad_evts_good}]
 This lemma can be seen easily, so we only provide a sketch. The existence of $x\in[x_{j-1},x_{j})$ so that $O^\eta(x,\Lambda_{(\tilde x_n,\tilde h_n)}(x)) > 2$ would imply the existence of three $\Lambda_{(x',h')}$ with $(x',h')\in(x,x+\varepsilon)\times(h-\varepsilon,h+\varepsilon)$ that do not merge before $x+\eta$. If $\Lambda_{(x_{j-1},h_{j}')}$ has not merged with $\Lambda_{(\tilde x_{n},\tilde h_{n})}$ before $x_{j}$, it will be strictly above $\Lambda_{(\tilde x_{n},\tilde h_{n})}$, hence above these $\Lambda_{(x',h')}$, which are thus below $\Lambda_{(x_{j-1},h_{j}')}$. These $\Lambda_{(x',h')}$ must also be above the following (again if there is no merging): $\Lambda_{(x_{j-1},h_{j-1}-2^{-m)}}$ in the case $\mathcal{M}_{j}^1$, $\lambda$ in the case $\mathcal{M}_{j}^2$, and $\inf_{\ell \in \mathds{N}^*}\Lambda_{(x_{j}'',\lambda(x_{j}'')+1/\ell)}$ or $\Lambda_{(\tilde x_{n'},\tilde h_{n'})}$ in the case $\mathcal{M}_{j}^3$. Given our bounds on the values of these processes, the values of the $\Lambda_{(x',h')}$ at $x_{j}$ are in a ``small interval'' which depends on the case. Moreover, we have the third point of Theorem \ref{thm_forward_bis}, and that by Lemma \ref{lem_cons_continue} almost surely $\mathcal{G}_{r,m}$ occurs, which yields three processes $\Lambda_{(x_{j},h_1)}$, $\Lambda_{(x_{j},h_2)}$, $\Lambda_{(x_{j},h_3)}$, with $h_1,h_2,h_3$ in the ``small interval'', that do not merge before $x'+\eta$. 
  \end{proof}
  
 \begin{proof}[Proof of Lemma \ref{lem_cons_OI_prob_bad_evts}.]
 As in the proof of Lemma \ref{lem_cons_OI_prob_bad_evts_a}, we will notice that if $\mathcal{A}_j$ occurs, either a ``$\mathcal{D}$ event'' occurs (which can be dealt with thanks to Lemma \ref{lem_cons_3BMbis}), or the forward lines fluctuate a lot. However, these fluctuations are harder to control than in Lemma \ref{lem_cons_OI_prob_bad_evts_a}. Indeed, when the forward lines are away from $\lambda$, Lemma \ref{lem_RAB_max} gives straightforward bounds, but when the forward lines are close to $\lambda$ we have to use several tricks to link their fluctuations to those of the Brownian motion driving them. Since $\mathcal{M}_{j}^3 \cap \mathcal{A}_{j}^3$ involves an infinite number of forward lines, it requires a particularly delicate treatment.
 
 We first bound $\mathds{P}_{\lambda,\chi}(\mathcal{A}_{j}^0)$. For any $h \geq \lambda(x_{j-1})$, we denote $h'=\max(h+2^{-m},m_j+2^{-m+1})$ and $\mathcal{A}^0(h) = \{\Lambda_{(x_{j-1},h')}(x_{j})\geq h'+2^{-m}\}\cup\{\Lambda_{(x_{j-1},h+2^{-m-2})}(x_{j})=\Lambda_{(x_{j-1},h')}(x_{j})\}$. We have $\mathcal{A}_{j}^0 \subset \mathcal{A}_{1}(h_{j-1})$. As in the proof of Lemma \ref{lem_cons_OI_prob_bad_evts_a}, for any $h \geq \lambda(x_{j-1})$, we have $h_{j-1}$ and $\mathcal{A}^0(h)$ independent, hence it is enough to bound $\mathds{P}_{\lambda,\chi}(\mathcal{A}^0(h))$ for any such $h$. Lemma \ref{lem_RAB_max} yields $\mathds{P}_{\lambda,\chi}(\max_{y \in [x_{j-1},x_{j}]}|\Lambda_{(x_{j-1},h')}(y)-\Lambda_{(x_{j-1},h')}(x_{j-1})| \geq 2^{-m}) \leq \exp(-\frac{5^m}{2^{2m+1}v})$ when $m$ is large enough (as $h'$ is chosen well above $\lambda$), hence $\mathds{P}_{\lambda,\chi}(\Lambda_{(x_{j-1},h')}(x_{j})\geq h'+2^{-m})\leq \exp(-\frac{5^m}{2^{2m+1}v})$ when $m$ is large enough. Moreover, if $\Lambda_{(x_{j-1},h+2^{-m-2})}(x_{j})=\Lambda_{(x_{j-1},h')}(x_{j})$, then $\max_{y \in [x_{j-1},x_{j}]}|\Lambda_{(x_{j-1},h')}(y)-\Lambda_{(x_{j-1},h')}(x_{j-1})| \geq 2^{-m-1}$ or $\max_{y \in [x_{j-1},x_{j}]}\Lambda_{(x_{j-1},h+2^{-m-2})}(y) \geq h'-2^{-m-1}$. The probability of the first event is smaller than $\exp(-\frac{5^m}{2^{2m+3}v})$ when $m$ is large enough by Lemma \ref{lem_RAB_max}. We now consider the second one. Let $(W_y)_{y \geq x_j}$ the Brownian motion driving $\Lambda_{(x_{j-1},h+2^{-m-2})}$. If $\max_{y \in [x_{j-1},x_{j}]}\Lambda_{(x_{j-1},h+2^{-m-2})}(y) \geq h'-2^{-m-1}$, then $\max_{y \in [x_{j-1},x_{j}]}|W_y-W_{x_{j-1}}| \geq 2^{-m-2}$. Indeed, if $\Lambda_{(x_{j-1},h+2^{-m-2})}$ does not meet $\lambda$ before hitting the maximum it is obvious, and if it does, let $z_1$ the first point at which the maximum is reached and $z_2 = \sup\{y \leq z_1\,|\,\Lambda_{(x_{j-1},h+2^{-m-2})}(y)=\lambda(y)\}$, then $W_{z_2}-W_{z_1} \geq 2^{-m}$. We deduce $\mathds{P}_{\lambda,\chi}(\max_{y \in [x_{j-1},x_{j}]}\Lambda_{(x_{j-1},h+2^{-m-2})}(y) \geq h'-2^{-m-1}) \leq \mathds{P}_{\lambda,\chi}(\max_{y \in [x_{j-1},x_{j}]}|W_y-W_{x_{j-1}}| \geq 2^{-m-2}) \leq \exp(-\frac{5^m}{2^{2m+5}v})$ when $m$ is large enough by Lemma \ref{lem_BM_max}. Consequently, when $m$ is large enough, $\mathds{P}_{\lambda,\chi}(\mathcal{A}^0(h)) \leq 3\exp(-\frac{5^m}{2^{2m+5}v})$, hence $\mathds{P}_{\lambda,\chi}(\mathcal{A}_j^0) \leq 3\exp(-\frac{5^m}{2^{2m+5}v})$.
 
 We now bound $\mathds{P}_{\lambda,\chi}(\mathcal{M}_{j}^1 \cap \mathcal{A}_{j}^1)$. Again, for any $h \geq \lambda(x_{j-1})$, if $h-2^{-m+1}>m_j$, we define the event $\mathcal{A}_{1}(h)=\{\Lambda_{(x_{j-1},h-2^{-m})}(x_{j})\leq h-2^{-m+1}\}\cup\{\Lambda_{(x_{j-1},h)}(x_{j})=\Lambda_{(x_{j-1},h-2^{-m})}(x_{j})\} \cup \mathcal{D}(x_{j},h-2^{-m+1},h+2^{-m+1},2^{-\frac{m}{3}})$. If $\mathcal{G}_{r,m}$ occurs (which happens almost surely by Lemma \ref{lem_cons_continue}), we have $\mathcal{M}_{j}^1 \cap \mathcal{A}_{j}^1 \subset \{h_{j-1}-2^{-m+1}>m_j\} \cap \mathcal{A}_{1}(h_{j-1})$. As in the proof of Lemma \ref{lem_cons_OI_prob_bad_evts_a}, for any $h$ with $h-2^{-m+1}>m_j$, $h_{j-1}$ and $\mathcal{A}_{1}(h)$ are independent, hence it is enough to bound $\mathds{P}_{\lambda,\chi}(\mathcal{A}_{1}(h))$ for any such $h$. We begin by noticing the increments of $\Lambda_{(x_{j-1},h-2^{-m})}$ on $[x_{j-1},x_{j}]$ are above those of the Brownian motion driving it, hence if $(W_y)_{y \geq x_{j-1}}$ is a Brownian motion, by Lemma \ref{lem_BM_max} we have $\mathds{P}_{\lambda,\chi}(\Lambda_{(x_{j-1},h-2^{-m})}(x_{j})\leq h-2^{-m+1}) \leq \mathds{P}_{\lambda,\chi}(W_{x_{j}}-W_{x_{j-1}} \leq -2^{-m}) \leq \exp(-\frac{5^m}{2^{2m+1}v})$ when $m$ is large enough. We now study $\mathds{P}_{\lambda,\chi}(\Lambda_{(x_{j-1},h)}(x_{j})=\Lambda_{(x_{j-1},h-2^{-m})}(x_{j}))$. If $\max_{y \in [x_{j-1},x_{j}]}|\Lambda_{(x_{j-1},h)}(y)-\Lambda_{(x_{j-1},h)}(x_{j-1})| \leq 2^{-m-2}$ and $\max_{y \in [x_{j-1},x_{j}]}|\Lambda_{(x_{j-1},h-2^{-m})}(y)-\Lambda_{(x_{j-1},h-2^{-m})}(x_{j-1})| \leq 2^{-m-2}$, then $\Lambda_{(x_{j-1},h)}(x_{j})\neq\Lambda_{(x_{j-1},h-2^{-m})}(x_{j})$. Moreover, by Lemma \ref{lem_RAB_max}, the probability one of these maxima is bigger than $2^{-m-2}$ is smaller than $\exp(-\frac{5^m}{2^{2m+5}v})$ when $m$ is large enough, hence $\mathds{P}_{\lambda,\chi}(\Lambda_{(x_{j-1},h)}(x_{j})=\Lambda_{(x_{j-1},h-2^{-m})}(x_{j})) \leq 2\exp(-\frac{5^m}{2^{2m+5}v})$. In addition, Lemma \ref{lem_cons_3BMbis} yields $\mathds{P}_{\lambda,\chi}(\mathcal{D}(x_{j},h-2^{-m+1},h+2^{-m+1},2^{-\frac{m}{3}})) \leq C(2^{-m+2}/\sqrt{2^{-\frac{m}{3}}})^3=C 2^{-\frac{5}{2}m+6}$. We thus obtain $\mathds{P}_{\lambda,\chi}(\mathcal{A}_{1}(h)) \leq 3\exp(-\frac{5^m}{2^{2m+5}v})+C 2^{-\frac{5}{2}m+6}$, therefore $\mathds{P}_{\lambda,\chi}(\mathcal{M}_{j}^1 \cap \mathcal{A}_{j}^1)\leq 3\exp(-\frac{5^m}{2^{2m+5}v})+C 2^{-\frac{5}{2}m+6}$. 
  
 We now bound $\mathds{P}_{\lambda,\chi}(\mathcal{M}_{j}^2 \cap \mathcal{A}_{j}^2)$. If $\mathcal{M}_{j}^2$ occurs we have $h_{j}'+2^{-m} \leq m_{j}+2^{-m+2} \leq \lambda(x_{j})+2^{-m+2}+2^{-m}$, so $\mathds{P}_{\lambda,\chi}(\mathcal{M}_{j}^2 \cap \mathcal{A}_{j}^2)\leq \mathds{P}_{\lambda,\chi}(\mathcal{D}(x_{j},\lambda(x_{j}),\lambda(x_{j})+3 \cdot 2^{-m+1},2^{-\frac{m}{3}})
 ) \leq C(3 \cdot 2^{-m+1}/\sqrt{2^{-\frac{m}{3}}})^3=C 6^3 2^{-\frac{5}{2}m}$ by Lemma \ref{lem_cons_3BMbis}.
 
 We finally bound $\mathds{P}_{\lambda,\chi}(\mathcal{M}_{j}^3 \cap \mathcal{A}_{j}^3)$. If $\mathcal{M}_{j}^3 \cap \mathcal{A}_{j}^3$ occurs, then $x''_{j} \in [x'_{j},x_{j}]$, for all $y\in[x''_{j},x_{j}]$ we have $\lambda(y)\leq m_{j}-2^{-m}$, and the following event occurs: $\{\min_{y \in [x'_{j},x_{j}]}\Lambda_{(\tilde x_n,\tilde h_n)}(y)< m_{j}-2^{-m-2}\}\cup \{\exists\, n'\in\mathds{N}, \Lambda_{(\tilde x_{n'},\tilde h_{n'})}(x_{j}'')=\lambda(x_{j}''), \Lambda_{(\tilde x_{n'},\tilde h_{n'})}(x_{j})\leq m_{j}-2^{-m+1}\} \cup\{\inf_{\ell \in \mathds{N}^*}\Lambda_{(x_{j}'',\lambda(x_{j}'')+1/\ell)}(x_{j}) \leq m_{j}-2^{-m+1}\}\cup\{\Lambda_{(x_{j}'',\lambda(x_{j}'')+2^{-m-2})}(x_{j}) > m_{j}-2^{-m-1}\} \cup \mathcal{D}(x_{j},m_{j}-2^{-m+1},m_{j}+2^{-m+2},2^{-\frac{m}{3}})$, so it is enough to bound the probability of this latter event in the case $x''_{j} \in [x'_{j},x_{j}]$ and $\lambda(y)\leq m_{j}-2^{-m}$ for all $y\in[x''_{j},x_{j}]$. We first notice $\Lambda_{(\tilde x_n,\tilde h_n)}(x'_{j}) \geq m_{j}$, and that the increments of $\Lambda_{(\tilde x_n,\tilde h_n)}$ in $[x'_{j},x_{j}]$ are above those of a Brownian motion, so if $(W_y)_{y \geq x_{j}'}$ is a Brownian motion, $\mathds{P}_{\lambda,\chi}(\min_{y \in [x'_{j},x_{j}]}\Lambda_{(\tilde x_n,\tilde h_n)}(y)< m_{j}-2^{-m-2}) \leq \mathds{P}_{\lambda,\chi}(\min_{y \in [x'_{j},x_{j}]}(W_{y}-W_{x'_{j}})< -2^{-m-2})$. Thus Lemma \ref{lem_BM_max} yields that when $m$ is large enough, 
 \begin{equation}\label{eq_cons_OI_badlambda}
  \mathds{P}_{\lambda,\chi}\left(\min_{y \in [x'_{j},x_{j}]}\Lambda_{(\tilde x_n,\tilde h_n)}(y)< m_{j}-2^{-m-2}\right) \leq \exp\left(-\frac{5^{m}}{2^{2m+5}v}\right). 
 \end{equation}
 In addition, denoting $N=\inf\{n'\in\mathds{N} \,|\, \tilde x_{n'} < x_{j}'', \Lambda_{(\tilde x_{n'},\tilde h_{n'})}(x_{j}'')=\lambda(x_{j}'')\}$ (which may be infinite), we have 
 \[
 \mathds{P}_{\lambda,\chi}(\exists\, n'\in\mathds{N}, \Lambda_{(\tilde x_{n'},\tilde h_{n'})}(x_{j}'')=\lambda(x_{j}''), \Lambda_{(\tilde x_{n'},\tilde h_{n'})}(x_{j})\leq m_{j}-2^{-m+1})
 \]
 \[
  \leq \sum_{n' \in \mathds{N}, \tilde x_{n'} < x_{j}''}\mathds{P}_{\lambda,\chi}\left(N=n',\Lambda_{(\tilde x_{n'},\tilde h_{n'})}(x_{j})\leq m_{j}-2^{-m+1}\right).
 \]
Moreover, for any $n' \in \mathds{N}$ so that $\tilde x_{n'} < x_{j}''$, the family $(\Lambda_{(\tilde x_{n''},\tilde h_{n''})})_{n'' \leq n', \tilde x_{n''} < x_{j}''}$ is a $(\lambda,\chi)$-FICRAB, hence it is Markov, thus $\mathds{P}_{\lambda,\chi}(N=n',\Lambda_{(\tilde x_{n'},\tilde h_{n'})}(x_{j})\leq m_{j}-2^{-m+1})=\mathds{P}_{\lambda,\chi}(N=n')\mathds{P}_{\lambda,\chi}(R(x_{j})\leq m_{j}-2^{-m+1})$ where $(R(y))_{y \geq x_{j}''}$ is a barrier-starting $(\lambda,\chi)$-RAB starting from $(x_{j}'',\lambda(x_{j}''))$ (see Remark \ref{rem_barrier_RABs}). This implies $\mathds{P}_{\lambda,\chi}(\exists\, n'\in\mathds{N}, \Lambda_{(\tilde x_{n'},\tilde h_{n'})}(x_{j}'')=\lambda(x_{j}''), \Lambda_{(\tilde x_{n'},\tilde h_{n'})}(x_{j})\leq m_{j}-2^{-m+1}) \leq \mathds{P}_{\lambda,\chi}(R(x_{j})\leq m_{j}-2^{-m+1}) \leq \mathds{P}_{\lambda,\chi}(R(x_{j})\leq \lambda(x_{j}'')-2^{-m})$. Since the increments of $R$ in $[x_{j}'',x_{j}]$ are above those of the Brownian motion driving it, if $(W_y)_{y \geq x_{j}'}$ is a Brownian motion, $\mathds{P}_{\lambda,\chi}(R(x_{j})\leq \lambda(x_{j}'')-2^{-m}) \leq \mathds{P}_{\lambda,\chi}(W(x_{j})-W(x_{r,j,m}'')\leq-2^{-m}) \leq \exp(-\frac{5^m}{2^{2m+1}v})$ when $m$ is large enough by Lemma \ref{lem_BM_max}. Therefore, when $m$ is large enough  
\begin{equation}\label{eq_cons_OI_badlambda2}
 \mathds{P}_{\lambda,\chi}(\exists\, n'\in\mathds{N}, \Lambda_{(\tilde x_{n'},\tilde h_{n'})}(x_{j}'')=\lambda(x_{j}''), \Lambda_{(\tilde x_{n'},\tilde h_{n'})}(x_{j})\leq m_{j}-2^{-m+1}) \leq \exp\left(-\frac{5^m}{2^{2m+1}v}\right). 
\end{equation}
We now consider $\mathds{P}_{\lambda,\chi}(\inf_{\ell \in \mathds{N}^*}\Lambda_{(x_{j}'',\lambda(x_{j}'')+1/\ell)}(x_{j}) \leq m_{j}-2^{-m+1}) \leq \mathds{P}_{\lambda,\chi}(\bigcup_{\ell \in \mathds{N}^*}\{\Lambda_{(x_{j}'',\lambda(x_{j}'')+1/\ell)}(x_{j}) < m_{j}-2^{-m+1}+2^{-m-1}\})=\lim_{\ell \to +\infty}\mathds{P}_{\lambda,\chi}(\Lambda_{(x_{j}'',\lambda(x_{j}'')+1/\ell)}(x_{j}) < m_{j}-2^{-m+1}+2^{-m-1})$ since it is an increasing union. Now, for any $\ell \in \mathds{N}^*$, we have $\mathds{P}_{\lambda,\chi}(\Lambda_{(x_{j}'',\lambda(x_{j}'')+1/\ell)}(x_{j}) < m_{j}-2^{-m+1}+2^{-m-1}) \leq \mathds{P}_{\lambda,\chi}(\Lambda_{(x_{j}'',\lambda(x_{j}'')+1/\ell)}(x_{j}) < \lambda(x_{j}'')-2^{-m-1}) \leq \mathds{P}_{\lambda,\chi}(\Lambda_{(x_{j}'',\lambda(x_{j}'')+1/\ell)}(x_{j})-\Lambda_{(x_{j}'',\lambda(x_{j}'')+1/\ell)}(x_{j}'') < -2^{-m-1}) \leq \exp(-\frac{5^m}{2^{2m+3}v})$ when $m$ is large enough, by the arguments used before with $R$. Consequently, when $m$ is large enough, 
\begin{equation}\label{eq_cons_OI_badlambda3}
 \mathds{P}_{\lambda,\chi}\left(\inf_{\ell \in \mathds{N}^*}\Lambda_{(x_{j}'',\lambda(x_{j}'')+1/\ell)}(x_{j}) \leq m_{j}-2^{-m+1}\right) \leq \exp\left(-\frac{5^m}{2^{2m+3}v}\right).
\end{equation}
We now study $\mathds{P}_{\lambda,\chi}(\Lambda_{(x_{j}'',\lambda(x_{j}'')+2^{-m-2})}(x_{j}) > m_{j}-2^{-m-1})$. For all $y\in[x''_{j},x_{j}]$ we have $\lambda(y)\leq m_{j}-2^{-m}=\lambda(x_{j}'')$, so by Lemma \ref{lem_RAB_max} we have $\mathds{P}_{\lambda,\chi}(\max_{y\in[x''_{j},x_{j}]}|\Lambda_{(x_{j}'',\lambda(x_{j}'')+2^{-m-2})}(y)-\Lambda_{(x_{j}'',\lambda(x_{j}'')+2^{-m-2})}(x_{j})| \geq 2^{-m-2}) \leq \exp(-\frac{5^m}{2^{m+5}v})$ when $m$ is large enough. Therefore, when $m$ is large enough, 
\begin{equation}\label{eq_cons_OI_badlambda4}
 \mathds{P}_{\lambda,\chi}(\Lambda_{(x_{j}'',\lambda(x_{j}'')+2^{-m-2})}(x_{j}) > m_{j}-2^{-m-1}) \leq \exp\left(-\frac{5^m}{2^{m+5}v}\right). 
\end{equation}
Finally we consider $\mathds{P}_{\lambda,\chi}(\mathcal{D}(x_{j},m_{j}-2^{-m+1},m_{j}+2^{-m+2},2^{-\frac{m}{3}})) \leq C (2^{-m+3}/\sqrt{2^{-\frac{m}{3}}})^3=C 2^9 2^{-\frac{5}{2}m}$ by Lemma \ref{lem_cons_3BMbis}. Along with \eqref{eq_cons_OI_badlambda}, \eqref{eq_cons_OI_badlambda2}, \eqref{eq_cons_OI_badlambda3} and \eqref{eq_cons_OI_badlambda4}, this yields that when $m$ is large enough, $\mathds{P}_{\lambda,\chi}(\mathcal{M}_{j}^3 \cap \mathcal{A}_{j}^3) \leq 4\exp(-\frac{5^m}{2^{2m+5}v})+C 2^9 2^{-\frac{5}{2}m}$, which ends the proof of Lemma \ref{lem_cons_OI_prob_bad_evts}. 
 \end{proof}
 \end{proof}
 
 \section{$(\lambda,\chi)$-true self-repelling motion: proof of the results in Section \ref{sec_TSRM}}\label{sec_TSRM_proofs}
 
 In this section we give the proofs for the construction and properties of the $(\lambda,\chi)$-true self-repelling motion and its local time, assuming $(\lambda,\chi)$ is a good barrier. The proof of Theorem \ref{thm_loctime} is the same as that of Theorem 4.2 of \cite{Toth_et_al1998}, while the proof of Theorem \ref{thm_RK} stems directly from the definitions, so we do not detail them. However, Propositions \ref{prop_P_singleton}, \ref{prop_X_basic} and \ref{prop_L_basic} need more attention, hence we will give their proofs. As in \cite{Toth_et_al1998}, we begin by defining, for any $(x,h)\in\mathds{R}_\lambda^2$, $D(x,h) = \{(x',h')\in\mathds{R}^2\,|\, \lambda(x)< h' \leq \bar\Lambda_{(x,h)}(x')\}$, and stating the following result, whose proof is the same as the proof of Proposition 3.1 of \cite{Toth_et_al1998}. 

\begin{proposition}\label{prop_31}
 Almost surely, for any $(x_1,h_1)\neq (x_2,h_2)$ in $\mathds{R}_\lambda^2$, one of the following occurs:
 \begin{itemize}
  \item either $(x_1,h_1)\in D(x_1,h_1) \subset D(x_2,h_2)$, $(x_2,h_2) \not\in D(x_1,h_1)$ and $D(x_2,h_2)\setminus D(x_1,h_1)$ contains a non-empty open set;
  \item or $(x_2,h_2)\in D(x_2,h_2) \subset D(x_1,h_1)$, $(x_1,h_1) \not\in D(x_2,h_2)$ and $D(x_1,h_1)\setminus D(x_2,h_2)$ contains a non-empty open set. 
 \end{itemize}
\end{proposition}
 
 Before proving Propositions \ref{prop_P_singleton}, \ref{prop_X_basic} and \ref{prop_L_basic}, we also need the following two technical results, Lemmas \ref{lem_D_bounded} and \ref{lem_T_large}. They are easy in the classical case and so were not proven in \cite{Toth_et_al1998}. However, they are more complicated with our more general barriers, which is why we show them. In particular, the first of these results, stating all the $D(x,h)$ are bounded, is what we need the assumption of a good barrier for.  

\begin{lemma}\label{lem_D_bounded}
 Almost surely, for all $(x,h)\in\mathds{R}_\lambda^2$, the set $D(x,h)$ is bounded. 
\end{lemma}

 \begin{proof}[Proof of Lemma \ref{lem_D_bounded}.]
The idea is that if the barrier is good, each $\Lambda_{(x,h)}$ will reach $\lambda$ at the right of $\chi$, and then be absorbed by $\lambda$, which will bound $D(x,h)$ on the right, with a similar argument on the left. However, since the lemma has to hold almost surely for all $(x,h)\in\mathds{R}_\lambda^2$, we need some care. Since $(\lambda,\chi)$ is good, the event $\{\forall\,k,\ell\in\mathds{Z}$, if $\ell>\lambda(k)$ there exists $y >k$ so that $\Lambda_{(k,\ell)}(y)=\lambda(y)\}$ is almost sure. Moreover, on this event, for any $(x,h)\in\mathds{R}_\lambda^2$, we can choose an integer $k>\max(\chi,x)$, and an integer $\ell>\Lambda_{(x,h)}(k)$. There exists $y >k$ so that $\Lambda_{(k,\ell)}(y)=\lambda(y)$, hence $\Lambda_{(x,h)}(y)=\lambda(y)$, thus $\Lambda_{(x,h)}(z)=\lambda(z)$ for all $z\geq y$. Since $\lambda$ and $\Lambda_{(x,h)}$ are continuous, $\{(x',h')\in\mathds{R}^2\,|\,x'\geq x,\lambda(x')< h' \leq \Lambda_{(x,h)}(x')\}$ is bounded. $\{(x',h')\in\mathds{R}^2\,|\,x'\leq x,\lambda(x')< h' \leq \Lambda_{(x,h)}^*(x')\}$ is bounded by a similar argument, which ends the proof.
 \end{proof}

 We recall that $T(x,h)=\int_{-\infty}^{+\infty}(\bar\Lambda_{(x,h)}(y)-\lambda(y))\mathrm{d}y$. Our second technical lemma is as follows.
 
\begin{lemma}\label{lem_T_large}
 Almost surely, for all $x\in\mathds{R}$, we have $\lim_{h \to +\infty} T(x,h)=+\infty$. 
\end{lemma}

\begin{proof}
The idea is that when $h$ is large, $\Lambda_{(x,h)}$ will start so much above $\lambda$ that $T(x,h)$ will be large, but since this has to hold almost surely for all $x\in\mathds{R}$, we again need to be careful. For any $m\in\mathds{Z}$, $K\in\mathds{N}^*$, by Lemma \ref{lem_RAB_max}, when $h$ is large enough, $\mathds{P}(\max_{y\in[m,m+1]} \lambda(y) \leq K,\inf_{y\in[m,m+1]}\Lambda_{(m,h)}(y)\leq 2K) \leq 2\frac{\sqrt{2v}}{(h-2K)\sqrt{\pi}}\exp(-\frac{(h-2K)^2}{2v})$ which tends to 0 when $h$ tends to $+\infty$, hence $\mathds{P}(\max_{y\in[m,m+1]} \lambda(y) \leq K,\forall\, h \in \mathds{N}\cap[2K,+\infty), \inf_{y\in[m,m+1]}\Lambda_{(m,h)}(y)\leq 2K)=0$. Moreover, since $\lambda$ is continuous, almost surely there exists $K_m \in \mathds{N}^*$ so that $\sup_{y\in[m,m+1]}\lambda(y) \leq K_m$. Then almost surely, for any $x\in\mathds{R}$, for any integer $K \geq K_{\lfloor x \rfloor}$ there exists $h \in \mathds{N}\cap[2K,+\infty)$ so that $\inf_{y\in[m,m+1]}\Lambda_{(m,h)}(y)\geq 2K$. Now, let $h'>\Lambda_{(m,h)}(x)$, we then have $\bar \Lambda_{(x,h')}(y) \geq \Lambda_{(m,h)}(y)$ for all $y\in[m,m+1]$, thus $T(x,h') \geq K$, which is enough to prove the lemma.
\end{proof}

We now give the proof of Proposition \ref{prop_P_singleton}, which allows to construct the $(\lambda,\chi)$-true self-repelling motion.

\begin{proof}[Proof of Proposition \ref{prop_P_singleton}.] The proof of this proposition is based on that of Lemma 3.4 of \cite{Toth_et_al1998}. The proof of Lemma 3.4 of \cite{Toth_et_al1998} itself can be directly applied in our setting, but it relies on Lemmas 3.2 and 3.3 of that work, whose proofs require some changes. The first change is that the proof of Lemma 3.2 of \cite{Toth_et_al1998} uses $T(\mathds{R}_\lambda^2) \cap [0,\alpha] \neq 0$ for all $\alpha>0$, which is not obvious in our setting. This can be circumvented by setting their $D$ to $\emptyset$ if this is not the case. The second and most important problem lies in the proof of Lemma 3.3 of \cite{Toth_et_al1998}, where an argument relies on the fact the marginals of a reflected Brownian motion have no atoms, which is true in their setting but not in ours, since they can have an atom on the barrier. This requires a non-trivial fix, which we explain. The arguments of \cite{Toth_et_al1998} provide $r,q,q',r'\in\mathds{D}$ so that $r < q < q' < r'$, and $k,k'\in\mathds{N}$ so that $\tilde x_k \leq r$, $\tilde x_{k'} \geq r'$, chosen as in their Equations (3.11) and (3.12). In the setting of \cite{Toth_et_al1998}, almost surely $\Lambda_{(\tilde x_k,\tilde h_k)}(q) \neq \Lambda_{(\tilde x_{k'},\tilde h_{k'})}^*(q)$, because these two quantities are independent and one of them is the marginal of a reflected Brownian motion, thus has no atoms. This does not hold in our case, so we need another argument. In order to finish the proof, it is actually enough to prove there exists almost surely a rational $s\in[q,q']$ with $\Lambda_{(\tilde x_k,\tilde h_k)}(s) \neq \Lambda_{(\tilde x_{k'},\tilde h_{k'})}^*(s)$. Let $\ell,\ell' \in\mathds{N}$ so that $\tilde x_\ell \leq q \leq q' \leq \tilde x_{\ell'}$. We notice that for any rational $s\in[q,q']$, conditionally to $(\lambda,\chi)$, if $(\tilde x_{\ell},\tilde h_{\ell}),(\tilde x_{\ell'},\tilde h_{\ell'})\in\mathds{R}_\lambda^2$, the process $(\Lambda_{(\tilde x_{\ell},\tilde h_{\ell})}(y))_{\tilde x_{\ell} \leq y \leq s}$ is independent from $\Lambda_{(\tilde x_{\ell'},\tilde h_{\ell'})}^*(s)$ (which can be shown as in \cite{Toth_et_al1998} just before their Lemma 9.2), hence by Lemma \ref{lem_cons_prob0}, $\mathds{P}_{\lambda,\chi}(\Lambda_{(\tilde x_\ell,\tilde h_\ell)}(s) = \Lambda_{(\tilde x_{\ell'},\tilde h_{\ell'})}^*(s)>\lambda(s))=0$. This implies $\mathds{P}(\exists\, s \in [q,q'] \cap \mathds{Q}, \Lambda_{(\tilde x_k,\tilde h_k)}(s) = \Lambda_{(\tilde x_{k'},\tilde h_{k'})}^*(s)>\lambda(s))=0$. Consequently, it is enough to show that if $(\tilde x_{\ell},\tilde h_{\ell}),(\tilde x_{\ell'},\tilde h_{\ell'})\in\mathds{R}_\lambda^2$, then $\mathds{P}_{\lambda,\chi}(\forall\,s\in[q,q'] \cap \mathds{Q}, \Lambda_{(\tilde x_\ell,\tilde h_\ell)}(s) = \Lambda_{(\tilde x_{\ell'},\tilde h_{\ell'})}^*(s)=\lambda(s))=0$. In the case $q< \chi$, this comes from the continuity of $\lambda$ and $\Lambda_{(\tilde x_\ell,\tilde h_\ell)}$, as well as from Lemma \ref{lem_no_stick_lambda} which yields $\mathds{P}_{\lambda,\chi}(\forall\, s \in [q,\min(q',\chi)], \Lambda_{(\tilde x_\ell,\tilde h_\ell)}(s) = \lambda(s))=0$. If $q \geq \chi$, then $q'>\chi$, and the result comes from Theorem \ref{thm_cons_backward} and Lemma \ref{lem_no_stick_lambda} which yield $\mathds{P}_{\lambda,\chi}(\forall \, s \in [\max(\chi,q),q'], \Lambda_{(\tilde x_{\ell'},\tilde h_{\ell'})}^*(s)=\lambda(s))=0$. 
\end{proof} 

We now give the proof of Proposition \ref{prop_X_basic}, which gathers basic properties of the $(\lambda,\chi)$-true self-repelling motion.

\begin{proof}[Proof of Proposition \ref{prop_X_basic}.]
To prove that almost surely for all $t \geq 0$ we have $H_t \geq \lambda(X_t)$, it is enough to notice $(X_t,H_t)$ is in the closure of a subset of $\mathds{R}_\lambda^2$. The fact that almost surely all the $\{t \geq 0 \,|\, X_t=x\}$, $x \in \mathds{R}$ are unbounded comes from Lemma \ref{lem_T_large}. We now prove almost sure continuity of $(X_t,H_t)_{t \geq 0}$. The proof of Proposition 3.5 of \cite{Toth_et_al1998} is enough to show that almost surely, for any $t \geq 0$, if $(t_n)_{n\in\mathds{N}}$ is a sequence converging to $t$ so that $(X_{t_n},H_{t_n})_{n\in\mathds{N}}$ has a limit, then this limit is $(X_t,H_t)$. Consequently, it is enough to prove that almost surely, for any such $t \geq 0$ and $(t_n)_{n\in\mathds{N}}$ converging to $t$, the sequence $(X_{t_n},H_{t_n})_{n\in\mathds{N}}$ is bounded. In order to do that, we notice Lemma \ref{lem_T_large} almost surely provides us with $(x_0,h_0)\in\mathds{R}_\lambda^2$ so that $T(x_0,h_0) \geq t+1$. Moreover, when $n$ is large enough, $t_n \leq t+1/2$. For such $n$, $(X_{t_n},H_{t_n}) \in \overline{\{(x,h)\in\mathds{R}_\lambda^2, T(x,h)\in(t_n-1/2,t_n+1/2)\}}$, and for any $(x,h)\in\mathds{R}_\lambda^2$ so that $T(x,h)\in(t_n-1/2,t_n+1/2)$, we have $T(x,h)<T(x_0,h_0)$, hence Proposition \ref{prop_31} yields $(x,h) \in D(x_0,h_0)$. This implies that when $n$ is large enough, $(X_{t_n},H_{t_n}) \in \overline{D(x_0,h_0)}$, which is almost surely bounded by Lemma \ref{lem_D_bounded}. This ends the proof of almost sure continuity of $(X_t,H_t)_{t \geq 0}$. 
\end{proof}

We now deal with the proof of Proposition \ref{prop_L_basic}, which states basic properties of the local times.

\begin{proof}[Proof of Proposition \ref{prop_L_basic}.]
 The fact that almost surely all the $L_t(x)$, $t \geq 0$, $x\in\mathds{R}$ are finite comes from Lemma \ref{lem_T_large}. The functions $t \mapsto L_t(x)$ are non-decreasing by definition, and continuous thanks to Proposition \ref{prop_31} which implies $T(x,.)$ is strictly increasing. To prove that almost surely for all $t \geq 0$ we have $H_t=L_t(X_t)$, we notice that if we had $H_t>L_t(X_t)$ (the case $H_t<L_t(X_t)$ can be dealt with in the same way), then by the definition of $L_t$ and since $T(X_t,.)$ is strictly increasing, for $h\in(L_t(X_t),H_t)$ we would have $T(X_t,h)>t$, hence by Proposition \ref{prop_31} the set $P_t=\bigcap_{\varepsilon>0} \overline{\{(x,h)\in\mathds{R}_\lambda^2, T(x,h)\in(t-\varepsilon,t+\varepsilon)\}}$ would be below $\bar \Lambda_{(X_t,h)}$, so we would have $H_t \leq h$, which is a contradiction.
\end{proof}
 
 \section*{Acknowledgements}
 
 The author wishes to thank Thomas Mountford for his helpful comments. She would also like to thank Brune Massoulié and her collaborators for sharing their insights on the zero-temperature ``true'' self-avoiding walk.

\end{document}